\documentclass[11pt]{article}

\usepackage{amsmath, amssymb, amsthm}
\usepackage{graphicx}
\usepackage{subcaption}
\usepackage{color}
\usepackage{url}
\usepackage{enumitem}
\usepackage{algorithm}
\usepackage{algpseudocode}
\usepackage{hyperref}
\usepackage{multirow}
\usepackage{makecell}

\hypersetup{
    colorlinks=true,
    linkcolor=blue,
    citecolor=blue,
    urlcolor=blue
}

\title{A Conservative Cascade Semi-Lagrangian Method for Solving the Vlasov Equation}

\author{
  Chunyang Xu\thanks{Email: \texttt{25b312011@stu.hit.edu.cn}} \\
  {\small School of Mathematics, Harbin Institute of Technology, Harbin 150001, China}
  \and
  Michel Mehrenberger \thanks{Email: \texttt{michel.mehrenberger@univ-amu.fr}}\\
  {\small Institut de Mathématiques de Marseille (I2M), Aix Marseille Univ, CNRS, Marseille, France}
  \and
  Chang Yang \thanks{Email: \texttt{yangchang@hit.edu.cn}}\\
  {\small School of Mathematics, Harbin Institute of Technology, Harbin 150001, China}
}

\date{\today}

\newtheorem{proposition}{Proposition}[section]
\newtheorem{lemma}{Lemma}[section]

\begin{document}

\maketitle

\begin{abstract}
The cascade remapping method, originally proposed by Nair et al.\ (2002) for atmospheric modeling, enables efficient and mass-conservative semi-Lagrangian (SL) transport through successive one-dimensional remapping. While widely used in geophysical flows, its application to plasma kinetics remains limited. To exploit its potential advantages in conservation and scalability, this work applies the conservative cascade semi-Lagrangian (CCSL) scheme to the Vlasov equation and related plasma models. A consistency analysis shows that the scheme attains second-order spatial accuracy, with the dominant error arising from the geometric approximation of the backtracked region. Moreover, two improvements are introduced: a freestream-preserving correction that ensures exact volume conservation, and a maximum-principle limiter that suppresses spurious oscillations while maintaining positivity and mass conservation. Numerical tests—including linear advection, guiding-center, and relativistic Vlasov–Maxwell models—confirm the high accuracy, robustness, and long-term stability of the improved CCSL method. Compared with the conservative semi-Lagrangian (CSL) and the backward semi-Lagrangian (BSL) schemes, it better preserves physical invariants under divergence-free conditions, providing a robust and efficient framework for high-fidelity plasma kinetic simulations with good parallel scalability.
\end{abstract}

\vspace{1em}
\vskip 1.5cm

\section{Introduction}

The evolution of plasmas can be described by the Vlasov equation self-consistently coupled with the corresponding field equations such as the Poisson equation or the Maxwell equations. These models are widely used in plasma physics, astrophysics, and accelerator physics to characterize the transport and deformation of the charged-particle distribution function in high-dimensional phase space. Due to the strong nonlinearity of the governing equations and their coupling with electromagnetic fields, analytical solutions are rarely available. Therefore, the development of efficient, stable, and conservative numerical algorithms has become a central topic in kinetic plasma simulations.

Numerical methods for solving Vlasov-type equations can be broadly classified into two categories: deterministic and statistical approaches. Deterministic methods include finite-difference \cite{banks2019high,Whealton1986A}, finite-volume \cite{MR2451425,FILBET2003247,FILBET2001166,5518448,Xiong2008A}, finite-element \cite{Douglas1982,GUO2013108,QIU20118386,ROSSMANITH20116203}, spectral \cite{GIRALDO1998114,LEBOURDIEC2006528}, and semi-Lagrangian (SL) \cite{Besse2008,SONNENDRUCKER1999201,Charles2013,EINKEMMER2019937,BESSE2003341,Nicolas2014A,yang2021} methods. These approaches can preserve the conservative structure of the equations and achieve high accuracy, though they are often limited by the Courant-Friedrichs–Levy (CFL) condition and computational cost in high-dimensional problems. Statistical methods, such as the particle-in-cell (PIC) technique \cite{Cottet1986,Victory1991,Filbet2016,DEGOND20105630}, approximate the distribution function by sampling a large number of macroparticles. While computationally efficient, PIC schemes tend to exhibit significant statistical noise in low-macroparticle-density regions.

Among deterministic solvers, the SL method has attracted considerable attention for its combination of accuracy and stability. The method traces characteristics backward in time on a fixed grid, allowing for large time steps while maintaining high-order accuracy and robustness. Owing to these advantages, SL schemes have been successfully applied in atmospheric modeling, fluid dynamics, and plasma kinetic simulations.

Conventional SL methods often rely on operator splitting \cite{Cheng1976The,Klimas1994A}, which decomposes multidimensional advection problems into a sequence of one-dimensional subproblems. This strategy simplifies implementation and facilitates parallelization. However, for nonlinear systems—such as the guiding-center model or the relativistic Vlasov model—splitting errors may accumulate at large time steps, resulting in a noticeable loss of global accuracy \cite{HUOT2003512}. To overcome this issue, several non-splitting SL approaches \cite{Nicolas2014A,Cai2019,CAI2018529} have been proposed in recent years to reduce splitting errors and improve spatio-temporal consistency. 
On the other hand, the Vlasov equation possesses multiple conservation properties under source-free conditions, including mass conservation and the invariance of $L^1$ and $L^2$ norms. To preserve these invariants, numerous conservative semi-Lagrangian (CSL) schemes have been developed. Crouseilles et al. \cite{CROUSEILLES20101927} proposed a CSL method based on volume remapping to ensure global mass conservation; Qiu and Christlieb \cite{QIU20101130} constructed a high-order conservative SL scheme with WENO reconstruction to improve spatial resolution; In the framework of semi-Lagrangian finite volume WENO schemes, Zheng et al. \cite{ZHENG2022114973} also adopted a positivity-preserving limiter to maintain the positivity of numerical solutions and suppress nonphysical oscillations, which ensures the reliability of simulations for kinetic models and fluid equations. These studies have greatly advanced the conservative properties of SL algorithms.

Among the conservative SL schemes, the cascade remapping method proposed by Nair et al. \cite{Nair2002Efficient,Ramachandran1999Monotonic} in atmospheric sciences provides a means to achieve strict mass conservation within the SL framework. This approach decomposes multidimensional conservative transport into a sequence of one-dimensional remapping operations, thereby maintaining both global and local mass conservation while retaining the large time-step and high-order features of SL methods. Building upon this idea, the conservative cascade semi-Lagrangian (CCSL) method introduces intermediate cells along each coordinate direction and performs successive conservative remapping to realize efficient multidimensional volume transport. Owing to its compact structure, straightforward implementation, and favorable computational scalability, CCSL has emerged as a promising direction for conservative SL algorithms.

Beyond conservation, researchers have recognized that freestream preservation \cite{NONOMURA2015242,Zhu2017,Kopriva2006} plays a critical role in maintaining long-term accuracy and stability. Ideally, a numerical scheme should exactly preserve a uniform distribution in the absence of sources; however, geometric approximation of the backtracked region and interpolation errors may introduce small deviations, leading to cumulative errors in physical invariants. Moreover, high-order interpolation in SL schemes can generate spurious oscillations or negative values. Therefore, preserving the maximum principle becomes another key requirement for ensuring stability.

To address these challenges, this work proposes two improvements to the CCSL method. First, a freestream-preserving correction is developed by adjusting the geometry of the backtracked region and reconstructing fluxes, ensuring exact freestream preservation and enhanced geometric consistency. Second, a maximum-principle limiter is constructed to suppress nonphysical oscillations while maintaining mass conservation and positivity. A consistency analysis confirms that the improved scheme achieves second-order spatial accuracy and exhibits enhanced robustness and precision in representative kinetic simulations.

The remainder of this paper is organized as follows. Section \ref{sec_method} presents the mathematical model, the conventional SL formulation, and the principles of the CSL and CCSL methods. Section \ref{sec_improvement} presents the consistency analysis, freestream-preserving correction, and the maximum-principle limiter. Section \ref{sec_results} provides numerical experiments for the linear advection equation, the guiding-center model, and the relativistic Vlasov–Maxwell system to validate the accuracy, conservation laws, and freestream-preserving and maximum-principle properties of the proposed method. Section \ref{sec_conclusion} concludes the paper and discusses extensions toward higher-dimensional and parallel implementations.

\section{Semi-Lagrangian method}
\label{sec_method}
\subsection{Mathematical model and characteristics}
We consider the following transport equation:
\begin{equation}
    \label{Eq.transport equation}
    \frac{\partial f}{\partial t}+\nabla_{\mathbf{x}}\cdot(\mathbf{a}f)=0,\, \mathbf{x}\in\Omega\subset\mathbb{R}^d,
\end{equation}
where $f$ is an unknown function dependent on time $t$ and space $\mathbf{x}$ and $\mathbf{a}$ is a divergence free vector (i.e., $\nabla_{\mathbf{x}}\cdot \mathbf{a}=0$). Eq. \eqref{Eq.transport equation} is equivalent to the advection equation:
\begin{equation}
    \label{Eq.advection equation}
    \frac{\partial f}{\partial t}+\mathbf{a}\cdot \nabla_{\mathbf{x}}f=0,\, \mathbf{x}\in\Omega\subset\mathbb{R}^d,
\end{equation}
as $\nabla_{\mathbf{x}}\cdot(\mathbf{a}f)=\mathbf{a}\cdot \nabla_{\mathbf{x}}f+f\nabla_{\mathbf{x}}\cdot\mathbf{a}$.
From the properties of the advection equation, we have the characteristics of Eq. \eqref{Eq.advection equation}:
\begin{equation}
    \label{Eq.characteristics}
    \left\{
    \begin{aligned}
        &\frac{d\mathbf{X}}{dt}=\mathbf{a}(\mathbf{X},t),\\
        &\mathbf{X}(s)=\mathbf{x},
    \end{aligned} 
    \right.
\end{equation}
where we denote by $\mathbf{X}(t;s,\mathbf{x})$ the solution of Eq. \eqref{Eq.characteristics}. Let the initial condition of the equation be $f(\mathbf{x},0)=f_0(\mathbf{x})$. Then we have 
\begin{equation*}
    f(\mathbf{x},t)=f_0(\mathbf{X}(0;t,\mathbf{x})).
\end{equation*}
The Semi-Lagrangian method achieves the iterative update of the unknown from one time step to the next by solving the characteristics and employing interpolation methods. As shown in the stability analysis of high-order Semi-Lagrangian schemes \cite{Falcone1998}, this method is well-suited for linear advection problems and permits large time steps; thus it is free from the CFL constraint.

\subsection{Classical semi-Lagrangian schemes}
\label{Sec. SL}
The semi-Lagrangian method has several variants, including forward and backward schemes, as well as point-to-point and volume-to-volume approaches. Among these variants, the point-to-point backward semi-Lagrangian method (BSL) is relatively classical. Its computational accuracy depends on the temporal accuracy of characteristics calculation and the interpolation accuracy, with a simple structure and strong robustness. However, the BSL method is not a conservative scheme and cannot guarantee mass conservation. In contrast, the Conservative Semi-Lagrangian (CSL) method \cite{CROUSEILLES20101927}, which is based on volume backtracking, is capable of ensuring mass conservation.

Let's consider the 1D transport equation:
\begin{equation*}
\label{Eq. 1D transport equation}
\frac{\partial f}{\partial t} + a \frac{\partial f}{\partial x} = 0,\;x\in I \subset \mathbb{R},
\end{equation*}
where $a$ is a constant. For $N\in \mathbb{N^*}$, we define the grid points
\begin{equation*}
    x_i=x_\text{min}+i\Delta x,\;i\in\frac{1}{2}\mathbb{Z}\ \;\text{with}\;\Delta x=(x_\text{max}-x_\text{min})/N \; \text{and} \; I=[x_\text{min},x_\text{max}].
\end{equation*}

For the BSL method, assume that we already know the solution values $f_i^n = f(t_n,x_i)$, $i=0,...,N$; the specific steps to compute $f_i^{n+1} = f(t_{n+1}, x_i)$ are as follows:
\begin{enumerate}
    \item Solve the characteristics to compute $X(t_n; t_{n+1}, x_i)$. According to the characteristic of the transport equation, the solution satisfies:
    \[
    f(t_{n+1},x_i) = f(t_n,X(t_n;t_{n+1},x_i));\\
    \]
    \item Perform interpolation based on the above relationship. Since the backtracked point $X(t_n;t_{n+1},x_i)$ usually does not lie on the grid, interpolation, (e.g. Lagrange interpolation, cubic spline interpolation) is required to approximate the value of $f(t_n,X(t_n;t_{n+1},x_i))$ from the known grid values of $f_i^n$.
\end{enumerate}

For the CSL method, assuming that the values $\overline{f}_i^n=\frac{1}{\Delta x}\int_{x_{i-\frac{1}{2}}}^{x_{i+\frac{1}{2}}}f(t_n,x)\,dx\;i=0,...,N$ are known, we adopt the same uniform 1D mesh as defined for the BSL method to implement volume-based backtracking and mass integration for computing the values $\overline{f}_i^{n+1}$.

Thanks to the characteristic, the volume-based backtracking is defined as:
\begin{equation}
\label{Eq. CSL backward}
\begin{aligned}
\int_{x_{i-\frac{1}{2}}}^{x_{i+\frac{1}{2}}} f(t_{n+1}, x) \, dx &= \int_{x_{i-\frac{1}{2}}^*}^{x_{i+\frac{1}{2}}^*} f(t_n, x) \, dx,
\end{aligned}
\end{equation}
where $x_{i-\frac{1}{2}}^*=X(t_n;t_{n+1},x_{i-\frac{1}{2}})$, $i=0,...,N+1$. Denote $m_i = \int_{x_{i-\frac{1}{2}}}^{x_{i+\frac{1}{2}}} f(t_{n+1}, x) \, dx$ and $m_i^* = \int_{x_{i-\frac{1}{2}}^*}^{x_{i+\frac{1}{2}}^*} f(t_n, x) \, dx$. Using Eq. \eqref{Eq. CSL backward}, we then have
\[
\overline{f}_i^{n+1} = \frac{m_i}{\Delta x} = \frac{m_i^*}{\Delta x}.
\]
We perform a conservative approximation for $m_i^*$. Let $x_{j-\frac{1}{2}}$ be the nearest half-grid point to the left side of $x_{i-\frac{1}{2}}^*$; thus, we have:
\begin{equation*}
    m_i^*=\int_{x_{i-\frac{1}{2}}^*}^{x_{i+\frac{1}{2}}^*} f(t_n, x) \, dx=\int_{x_{j+\frac{1}{2}}}^{x_{i+\frac{1}{2}}^*} f(t_n, x) \, dx+\Delta x\overline{f}_j^n-\int_{x_{j-\frac{1}{2}}}^{x_{i-\frac{1}{2}}^*} f(t_n, x) \, dx.
\end{equation*}

Now for instance to approximate $\int_{x_{j-\frac{1}{2}}}^{x_{i-\frac{1}{2}}^*}f(t_n, x)\, dx$, the following steps are performed:
\begin{enumerate}
    \item Let $F_j(x)$ be the primitive function of $f(t,x)$ with respect to $x$ over the interval $[x_{j-\frac{1}{2}}, x_{j+\frac{1}{2}}] $; that is,
    \begin{equation*}
        F_j(x)=\int_{x_{j-\frac{1}{2}}}^xf(t_n,\xi)\,d\xi;
    \end{equation*}
    \item Construct $F^{2d+1}_j(x), \;d\in \mathbb{N}$ as an interpolation polynomial of degree $2d+1$ for $F_j(x)$;
    \item Approximate the integral thanks to the interpolation polynomial : $\int_{x_{j-\frac{1}{2}}}^{x_{i-\frac{1}{2}}^*}f(t_n, x)\, dx=F_j(x_{i-\frac{1}{2}}^*) \approx F_j^{2d+1}(x_{i-\frac{1}{2}}^*)$.
\end{enumerate}

Through the above steps, we obtain $m_i^*$ and get $\overline{f}_i^{n+1}=\frac{m_i^*}{\Delta x}$, thus completing the iteration for one time step.

We denote by $\mathcal{T}$ the one-dimensional conservative reconstruction operator in the CSL method, 
defined as
\[
\mathcal{T} : \bigl(\{x_{i-\frac{1}{2}}^*\}_{i=0}^{N+1},\, \{\overline{f}_i^n\}_{i=0}^{N}\bigr)
\;\longmapsto\;
\{m_i^*\}_{i=0}^{N}.
\]
Here, $\mathcal{T}$ reconstructs the target cell masses $\{m_i^*\}$ at time $t_n$ from the set of backtracked 
cell-face positions $\{x_{i-\frac{1}{2}}^*\}$ and the known cell-averaged values $\{\overline{f}_i^n\}$ at time $t_n$, 
following the integral approximation procedure described above.

The 1D problem is straightforward because the volume corresponds to a 1D interval, simplifying the calculations. However, for 2D problems, the backtracked region has a complex geometry that complicates direct interpolation. Therefore, the 2D cascade remapping method \cite{Nair2002Efficient} is employed to address this issue.

\subsection{Conservative cascade semi-Lagrangian (CCSL) method}
\label{Sec. Desc:CCSL}
For Eq. \eqref{Eq.transport equation} in the 2D case, we define the computational domain as $\Omega=[x_\text{min}, x_\text{max}] \times [y_\text{min}, y_\text{max}]$. For $ N_x, N_y\in \mathbb{N}^*$, we set the grid points as $x_i=x_\text{min}+i\Delta x,\,y_j=y_\text{min}+j\Delta y, \;i,j\in\frac{1}{2}\mathbb{Z}$ with $\Delta x=(x_\text{max}-x_\text{min})/N_x,\;\Delta y=(y_\text{max}-y_\text{min})/N_y$.

Based on the property of mass conservation, considering the $\Omega_{ij}$ $([x_{i-\frac{1}{2}},x_{i+\frac{1}{2}}]\times[y_{j-\frac{1}{2}},y_{j+\frac{1}{2}}])$ in Fig. \ref{Fig_2D CCSL figure}, we have:
\begin{equation}
\label{Eq. mass conservation}
\overline{f}_{ij}^{n+1}\Delta x\Delta y=\iint_{\Omega_{ij}}f(t_{n+1},x,y)\,dxdy = \iint_{\Omega_{ij}^*} f(t_n,x,y) \, dx dy
\end{equation}
where $\overline{f}_{ij}^{n+1}$ denotes the average density of $\Omega_{ij}$ (region $ABCD$) at time $t_{n+1}$ and $\Omega^*_{ij}$ (region $A'B'C'D'$) denotes the backward region of $\Omega_{ij}$ under the backward characteristics at time $t_n$. 

\begin{figure}[htbp]
\centering
\includegraphics[width=0.6\textwidth]{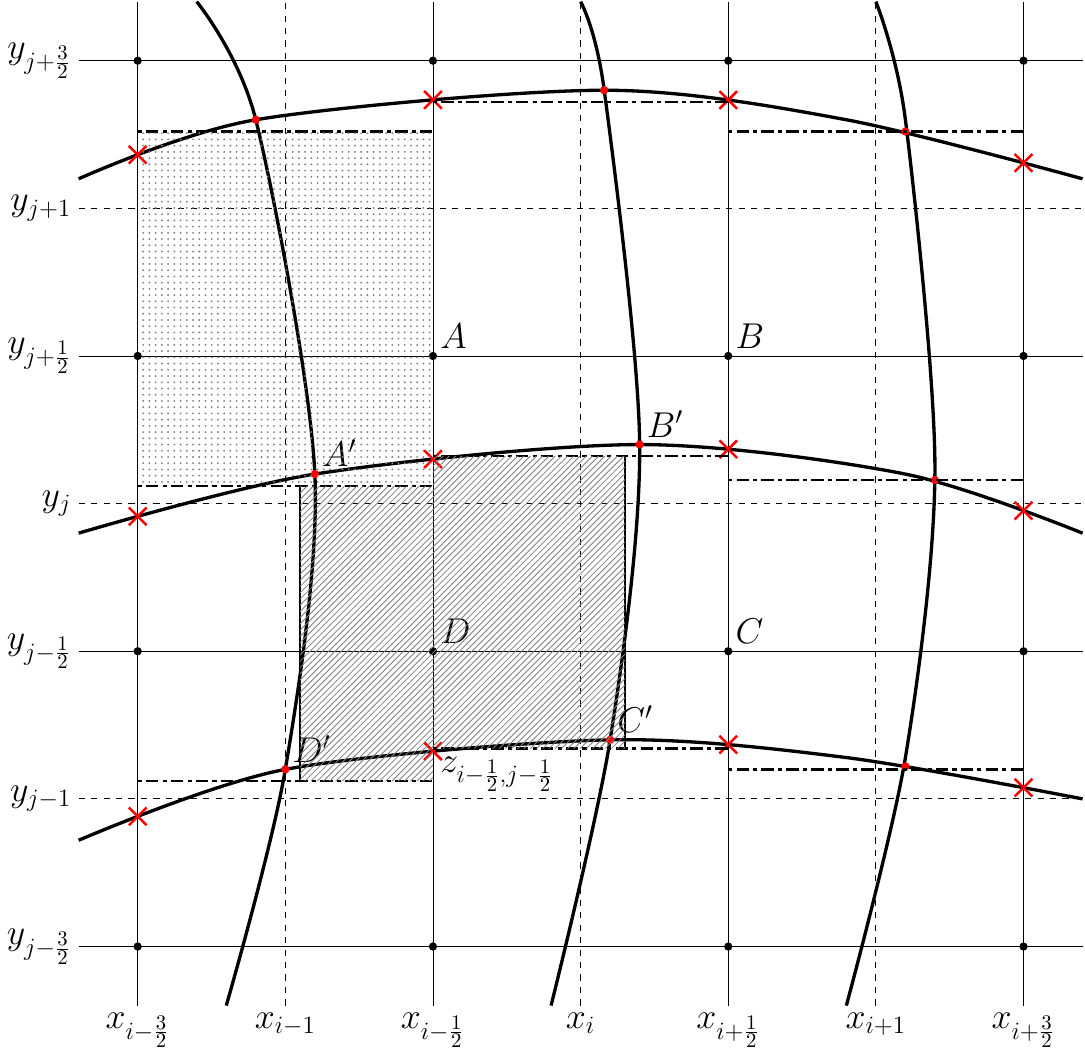}
\caption{Schematic diagram of the CCSL method for 2D case.}
\label{Fig_2D CCSL figure}
\end{figure}

Calculating the mass of the backtracked irregular region using discrete values is challenging. Therefore, we adopt the 2D cascade remapping scheme, which approximates the mass of this region by introducing intermediate cells to bridge the discrete computation and the approximation of the irregular region. 

For clarity, we provide the definitions of the intermediate cell and the backtracked cell:
\begin{itemize}
    \item The intermediate cell is demarcated by vertical half-grid lines on its left and right, and by the averaged y-coordinates of adjacent intermediate grid points at its top and bottom. In particular, an intermediate cell corresponds to the dot-filled region shown in Fig.~\ref{Fig_2D CCSL figure}.
    \item The backtracked cell is formed by clipping and splicing the intermediate cells, thereby approximating the backtracked region. A specific example is the shaded region in Fig. \ref{Fig_2D CCSL figure}, which serves as an approximation to region \( A'B'C'D' \)
\end{itemize}

The first step is to compute the intermediate cell $I_{ij}$ (detailed calculation steps for the intermediate cell will be provided later to ensure clarity of the current remapping workflow). For each column of cells, we use the reconstruction scheme in the 1D CSL method to compute the mass of the intermediate cell, denoted as $m_{I_{ij}}$. On this basis, clipping and splicing operations are performed on each row of intermediate cells to form the backtracked cell \(B_{ij}\), which corresponds to the shaded region in Fig.~\ref{Fig_2D CCSL figure} and provides the approximation to the target irregular region. Once the backtracked cell $B_{ij}$ is constructed, we again apply the 1D CSL reconstruction to derive its mass $m_{B_{ij}}$; finally, substituting this mass into Eq. \eqref{Eq. mass conservation} allows us to solve for $\overline{f}_{ij}^{n+1}$.

To clarify how the intermediate cell $I_{ij}$ — a core component of the above remapping process — is constructed, its calculation steps are detailed as follows:

\begin{enumerate}
    \item Determine intermediate grid points: The intermediate grid point $z_{i-\frac{1}{2}, j-\frac{1}{2}}$ is geometrically defined as the intersection of the backtracked curve from the line $y = y_{j-\frac{1}{2}}$ (at $t_{n+1}$) and the Eulerian grid line $x = x_{i-\frac{1}{2}}$. Its coordinates are given by $(x_{i-\frac{1}{2}}, \widetilde{y}_{i-\frac{1}{2}, j-\frac{1}{2}})$, where the $y$-component $\widetilde{y}_{i-\frac{1}{2}, j-\frac{1}{2}}$ is computed using an interpolation method.

    \item Determine the intermediate cell: With the intermediate grid points $z_{i\pm\frac{1}{2}, j\pm\frac{1}{2}}$ as the fundamental vertices, the intermediate cell $I_{ij}$ is uniquely determined. Specifically, its left and right boundaries are defined by the vertical Euler grid lines $x_{i\pm\frac{1}{2}}$, while its upper and lower boundaries are computed as the average of adjacent intermediate grid point coordinates, satisfying $\overline{y}_{i,j\pm\frac{1}{2}} = \frac{1}{2}(\widetilde{y}_{i-\frac{1}{2}, j\pm\frac{1}{2}} + \widetilde{y}_{i+\frac{1}{2}, j\pm\frac{1}{2}})$.
\end{enumerate}

Building on the intermediate cell construction, the 2D Cascade remapping, and the 1D conservative reconstruction notation defined previously, the complete implementation steps of the 2D CCSL method are systematically organized as follows:

\begin{enumerate}
\item Solve the characteristics to perform backtracking: For the half-grid points $(x_{i-\frac{1}{2}}, y_{j-\frac{1}{2}})$ at time $t_{n+1}$, trace their positions backward to time $t_n$ along the characteristics, yielding the corner points $(x^*_{i-\frac{1}{2},j-\frac{1}{2}}, y^*_{i-\frac{1}{2},j-\frac{1}{2}})$;

\item Generate the intermediate cells $I_{ij}$: Using the corner points from Step 1, construct the intermediate cells $I_{ij}$ following the geometric definition detailed previously;

\item Apply the one-dimensional CSL reconstruction in the column direction. 
For each column $i$:
\begin{itemize}
  \item Collect the intermediate cell boundaries 
  $\{\overline{y}_{i,j-\frac{1}{2}}\}_{j=0}^{N_y+1}$ along the vertical direction;
  \item Extract the cell-averaged values 
  $\{\overline{f}_{i,j}^n\}_{j=0}^{N_y}$ at time $t_n$ for all cells in the column;
  \item Apply the one-dimensional conservative reconstruction operator $\mathcal{T}$, defined as
  \[
  \mathcal{T} :
  \bigl(\{\overline{y}_{i,j-\frac{1}{2}}\}_{j=0}^{N_y+1},\,
  \{\Delta x\, \overline{f}_{i,j}^n\}_{j=0}^{N_y}\bigr)
  \;\longmapsto\;
  \{m_{I_{ij}}\}_{j=0}^{N_y},
  \]
  where the output $\{m_{I_{ij}}\}$ represents the mass of the intermediate cells $I_{ij}$.
\end{itemize}
This procedure is repeated for all columns to compute the mass of the intermediate cells over the entire domain.

\item Apply the one-dimensional CSL reconstruction in the row direction. 
For each row $j$:
\begin{itemize}
  \item Collect the averaged face positions 
  $\{\overline{x}_{i-\frac{1}{2},j}\}_{i=0}^{N_x+1}$, 
  where $\overline{x}_{i-\frac{1}{2},j}
  = \tfrac{1}{2}\bigl(x^*_{i-\frac{1}{2},j-\frac{1}{2}} + x^*_{i-\frac{1}{2},j+\frac{1}{2}}\bigr)$;
  \item Use the intermediate cell averages $\{m_{I_{ij}}/\Delta x\}_{i=0}^{N_x}$ 
  from Step~3 as input cell-averaged values;
  \item Apply the same reconstruction operator $\mathcal{T}$:
  \[
  \mathcal{T} :
  \bigl(\{\overline{x}_{i-\frac{1}{2},j}\}_{i=0}^{N_x+1},\,
  \{m_{I_{ij}}/\Delta x\}_{i=0}^{N_x}\bigr)
  \;\longmapsto\;
  \{m_{B_{ij}}\}_{i=0}^{N_x},
  \]
  where $\{m_{B_{ij}}\}$ denotes the mass of the backtracked cells $B_{ij}$.
\end{itemize}
Finally, substitute the computed backtracked-cell masses into Eq.~\eqref{Eq. mass conservation} 
to obtain the updated cell-averaged values $\{\overline{f}_{ij}^{\,n+1}\}$ at time $t_{n+1}$.
\end{enumerate}

From the above steps, it can be seen that the CCSL method exhibits excellent parallel scalability. Therefore, the application to high dimensional is highly feasible.

For the 2D CCSL method, although it is not restricted by the CFL condition, the method fails to work if the grid is excessively distorted. Herein, we present a necessary condition for the validity of the CCSL method.

\begin{proposition}
\label{Prop_ CCSL condition}
For the 2D CCSL method to be applicable to Eq. \eqref{Eq.transport equation}, the following constraints must be satisfied:
\begin{equation}
    \label{Eq. limitation of CCSL}
    \frac{\Delta a^x}{\Delta x} \cdot \Delta t < 1 \quad (\text{for } x\text{-direction}), \quad \frac{\Delta a^y}{\Delta y} \cdot \Delta t < 1 \quad (\text{for } y\text{-direction}),
\end{equation}
where:
\begin{enumerate}
    \item $a^x_{ij},\,a^y_{ij}$ represent the velocity components at grid point $(x_i, y_j)$;
    \item $ \Delta a^x=\max\left| a_{i+1,j}^x-a_{i,j}^x \right|$ (resp. $ \Delta a^y=\max\left| a_{i,j+1}^y - a_{i,j}^y \right| $) denotes the maximum velocity difference between adjacent grid points in the $x$- (resp. $ y $-) direction;
    \item $\Delta t$ is the time step.
\end{enumerate}
\end{proposition}

\begin{proof}
For the intermediate cell construction in the CCSL method to remain valid, 
the spatial ordering of grid points must be preserved after backtracking. 
This requirement arises from the embedded one-dimensional CSL reconstruction, 
which is defined only when the grid-point ordering is maintained throughout the backtracking process.

The necessity of Eq.~\eqref{Eq. limitation of CCSL} can be demonstrated by contraposition. 
Suppose the condition is violated; we then show that the point ordering fails to be preserved.

Consider an order reversal in the $x$-direction, 
which occurs when a grid point is backtracked to the left of its right-hand neighbor:
\[
x_i - a_{i,j}^x \Delta t > x_{i+1} - a_{i+1,j}^x \Delta t,
\]
given that $x_i < x_{i+1}$. 
Rearranging this inequality gives
\[
x_{i+1} - x_i < (a_{i+1,j}^x - a_{i,j}^x) \Delta t.
\]
By defining $\Delta x = x_{i+1} - x_i$ and 
$\Delta a^x = \max|a_{i+1,j}^x - a_{i,j}^x|$, 
the condition for order reversal becomes
\[
\frac{\Delta a^x}{\Delta x} \, \Delta t > 1.
\]
An identical argument applies to the $y$-direction, yielding
\[
\frac{\Delta a^y}{\Delta y} \, \Delta t > 1.
\]
To prevent order reversal and ensure a valid reconstruction, 
the negations of these conditions must hold, 
which leads directly to the constraints in Eq.~\eqref{Eq. limitation of CCSL}:
\[
\frac{\Delta a^x}{\Delta x} \, \Delta t < 1, 
\qquad 
\frac{\Delta a^y}{\Delta y} \, \Delta t < 1.
\]
This completes the proof.
\end{proof}

\section{Error analysis and enhanced methods}
\label{sec_improvement}

\subsection{Error analysis of the 2D CCSL method}
\label{subsec:CCSL}
Different from the 1D CSL method, the 2D CCSL method introduces an additional geometric error. This error arises from approximating the backtracked region through rectangular splicing. When the accuracy of characteristic calculation and interpolation is sufficiently high, the geometric error becomes dominant. Therefore, this subsection focuses on the discussion of geometric error.

Let's refer to the region approximated by the intermediate cell as the true intermediate cell. Before presenting the proof, we summarize the key steps for clarity:

\begin{itemize}
  \item Step 1: Error between the intermediate cell and the true intermediate cell. As shown in Fig.~\ref{Fig_figure_intermediate_cell}, the red dashed region represents the clipped and spliced intermediate cells (i.e., the backtracked cells obtained by the CCSL method), whose mass is computed using the one-dimensional CSL reconstruction. The blue dashed region corresponds to the clipped and spliced true intermediate cells, whose mass can be obtained in the same manner. Based on the previously derived error of the intermediate cell, we can evaluate the mass discrepancy between the red and blue regions.

  \item Step 2: Error between the clipped and spliced true intermediate cells and the true backtracked region. The next step is to estimate the error between the blue region and the true backtracked region. We introduce a mapping that transforms the blue region into a rectangle. Applying the same mapping to the true backtracked region converts it into a domain with straight upper and lower boundaries and curved side boundaries, 
  which facilitates the evaluation of the mass error using the same technique as in Step 1.
\end{itemize}

\begin{figure}[htbp]
\centering
\includegraphics[width=0.35\textwidth]{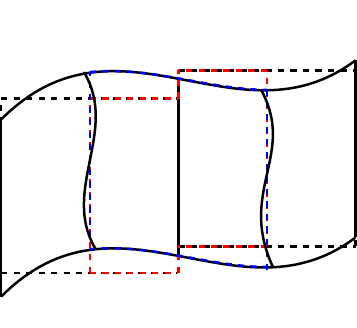}
\caption{Schematic diagram of the backtracked cells (red), the clipped and spliced true intermediate cells (blue), and the true backtracked region (black).}
\label{Fig_figure_intermediate_cell}
\end{figure}

By using the blue region as a bridge, we obtain the mass error between the backtracked cell of the CCSL method and the true backtracked region, thereby determining the error convergence order of the CCSL method.

We now proceed with a detailed proof. The grid setup follows that in Section \ref{Sec. Desc:CCSL}. Without loss of generality, we assume a uniform grid spacing $h=\Delta x = \Delta y$ for the error analysis. First, we need to calculate the mass error between the intermediate cell and the true intermediate cell. Here, a lemma is required to assist in our error estimation.

\begin{lemma}\label{Lemm:integral_error}
    Assume that $p\in C^2([x_1, x_2]) $, and that $f$ is twice continuously differentiable in a neighborhood of the curve 
    $\{(u, p(u)) : u \in [x_1, x_2]\}$.  
    Then the integral error
    \[
    \epsilon = \int_{x_1}^{x_2} \int_{\frac{1}{2}\left[p(x_1) +     p(x_2)\right]}^{p(u)} f(u, v)\, dv\, du
    \]
    admits the asymptotic expansion
    \[
    \epsilon =\left(\frac{1}{12} \frac{\partial f}{\partial u}\, p'+ \frac{1}{24} \frac{\partial f}{\partial v}\, (p')^2-\frac{1}{12} f\, p''\right)\Big|_{(x_1,\,p(x_1))} h^3+ \mathcal{O}(h^4),
    \]
    where $h = x_2 - x_1$.
\end{lemma}
\begin{proof}
This proof is presented in \ref{Ap. Proof of Lemma.integral}.
\end{proof}

Consider the 2D characteristic equations for $\mathbf{X} = (X_1, X_2)$ (with $X_1, X_2$ denoting the $x$- and $y$-components of $\mathbf{X}$):  
\begin{equation*}
\left\{
    \begin{aligned}
        &\frac{d\mathbf{X}}{dt} = \mathbf{a}(\mathbf{X}, t), \\
        &\mathbf{X}(t_{n+1}) = \mathbf{x}_0 = (x_0, y_0), \quad (x_0, y_0) \in \Omega.
    \end{aligned} 
\right.
\end{equation*}  
Let $\mathbf{X}(t_n; t_{n+1}, \mathbf{x}_0) = \mathbf{x}^* = (x^*, y^*)$ denote the solution at $t = t_n$.  
\begin{itemize}
\item Fix $y_0 = \xi$: The relationship between $x_0$ and $x^*$ is $X_1(t_n; t_{n+1}, (x_0, \xi)) = x^*$, denoted as $\mathcal{X}_{\xi}(x_0) = x^*$. Define $p(x, \xi): \Omega \to \mathbb{R}$ by  
  \begin{equation}	
      \label{Eq. p(x)}
      p(x, \xi) = X_2\Big(t_n; t_{n+1},(\mathcal{X}_{\xi}^{-1}(x), \xi)\Big),
  \end{equation}  
  where $\mathcal{X}_{\xi}^{-1}$ is the inverse of $\mathcal{X}_{\xi}$.  

\item Fix $x_0 = \eta$: The relationship between $y_0$ and $y^*$ is $X_2(t_n; t_{n+1}, (\eta, y_0)) = y^*$, denoted as $\mathcal{Y}_{\eta}(y_0) = y^*$. Define $q(y, \eta): \Omega \to \mathbb{R}$ by  
  \begin{equation*}
      q(y, \eta) = X_1\Big(t_n; t_{n+1}, (\eta, \mathcal{Y}_{\eta}^{-1}(y)\Big),
  \end{equation*}  
  where $\mathcal{Y}_{\eta}^{-1}$ is the inverse of $\mathcal{Y}_{\eta}$.
\end{itemize}

When the characteristics are smooth, after defining the above functions, we can observe that for the grid line $y=y_j$ at time $t_{n+1}$, its backtracked curve is $p(x,y_j)$ at time $t_n$. Similarly, for the grid line $x=x_i$, its backtracked curve is $q(y,x_i)$.
  
\begin{proposition}
\label{Prop_intermediate}
When the characteristics satisfy the condition specified in Eq. \eqref{Eq. limitation of CCSL}, the following proposition holds: if the distribution function $ f(x,y) $ and the functions $ p(x,\eta) $ (that enclose the true intermediate cell) are both twice continuously differentiable, then the mass error between the true intermediate cell and the intermediate cell is $ \mathcal{O}(h^4)$.
\end{proposition}
\begin{proof}
    This proof is presented in \ref{Ap. Proof of Proposition Prop_intermediate}.
\end{proof}
We now need to define a mapping to "straighten" the true intermediate cell.
\begin{proposition}
\label{Prop_T}
As illustrated in Fig. \ref{Fig_T}, there exists a bijective mapping $T:\Omega\subset\mathbb{R}^2\to\hat{\Omega}\subset\mathbb{R}^2$ satisfying:
\begin{enumerate}
    \item  The mapping $T$ converts both the upper and lower smooth non-intersecting boundaries $p_1(x)$ and $p_2(x)$ into a pair of parallel straight edges $y=0$ and $y=\Delta x$;
    \item  $\pi_x(T(x,y)) = \pi_x(x,y)$ for all $(x,y)\in\Omega$, where $\pi_x(x,y)=x$ is the projection operator.
\end{enumerate}
\end{proposition}
\begin{figure}[htbp]
    \centering
    \includegraphics[width=0.4\textwidth]{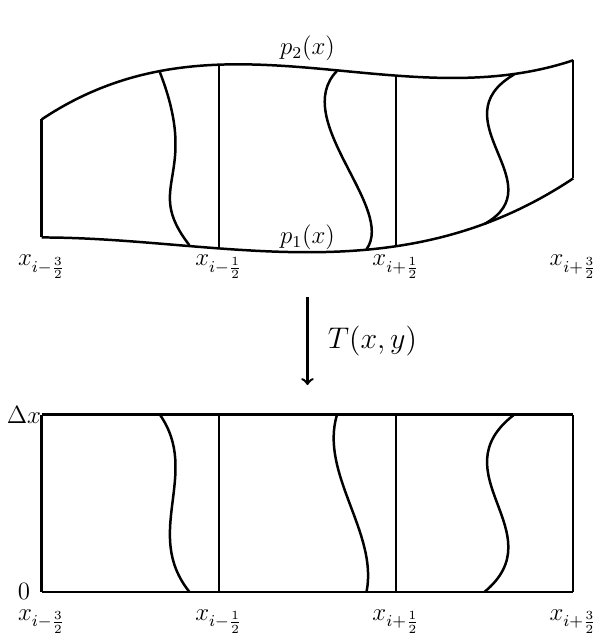} 
    \caption{Coordinate transformation in 2D CCSL method} 
    \label{Fig_T} 
\end{figure}
\begin{proof}
The bijection $T$ is explicitly given by:
\[
T(x,y) =(x,c_x(y)) =\left(x, \frac{y - p_1(x)}{p_2(x) - p_1(x)}\Delta x\right),
\]
where $c_x(y)=\frac{y-p_1(x)}{p_2(x)-p_1(x)}\Delta x$. This mapping satisfies:

\begin{enumerate}
    \item \textbf{Bijectivity}: For each fixed $x$, $T$ is a linear and bijective transformation in $y$ between $[p_1(x),p_2(x)]$ and $[0,\Delta x]$.
    
    \item \textbf{Boundary Straightening}:
    \begin{align*}
        T(x,p_1(x)) &= (x, 0) ,\\
        T(x,p_2(x)) &= (x, \Delta x),
    \end{align*}
    converting curved boundaries to parallel straight lines.
    
    \item \textbf{$x$-preservation}: $\pi_x(T(x,y)) = \pi_x(x,y)$.
    
\end{enumerate}
\end{proof}

Having established the propositions above, we now proceed to prove the order of truncation error for the CCSL method.

\begin{proposition}
\label{Prop_ CCSL error order}
The numerical update computed by the CCSL scheme satisfies
\[
\left|
\frac{1}{h^2} \iint_{\Omega_{k,l}} f(x,y,t^{n+1}) \, dx\,dy
- \overline{f}_{k,l}^{\,n+1}
\right| = \mathcal{O}(h^2),
\]
for all cell indices $k = 0,\ldots,N_x$ and $l = 0,\ldots,N_y$, 
provided that the corresponding cell averages $\{\overline{f}_{k,l}^{\,n}\}$ are exact, and that $f$, $p$, and $q$ are twice continuously differentiable functions.
\end{proposition}
\begin{proof}
    Without loss of generality, we consider the case shown in the Fig. \ref{Fig_CCSL_error} (for more general cases, detailed derivations and verification are deferred to \ref{Ap. General cases of CCSL method} to avoid disrupting the logical flow of the main text). In this scenario, we have:
    \begin{align}
        &\left| \frac{1}{h^2} \iint_{\Omega_{k,l}} f(x,y,t^{n+1}) \, dx\,dy - \overline{f}_{k,l}^{n+1} \right|\notag\\
        =&\frac{1}{h^2}\left| \iint_{\Omega_{k,l}} f(x,y,t^{n+1}) \, dx\,dy - h^2\overline{f}_{k,l}^{n+1} \right|\notag\\
        =&\frac{1}{h^2} \left|\iint_{A'B'C'D'} f(x,y,t^{n}) \, dx\,dy - h^2\overline{f}_{k,l}^{n+1}\right|\notag\\
        =&\frac{1}{h^2} \left|M_{\text{true}} - M_{\text{CCSL}}\right|,
        \label{Eq. CCSL_error}
    \end{align}
    where $M_{\text{true}}$ denotes the true mass of the backtracked region and $M_{\text{CCSL}}$ denotes the numerically approximated cell mass computed using CCSL method.
    \begin{figure}[htbp]
    \centering
    \includegraphics[width=0.4\textwidth]{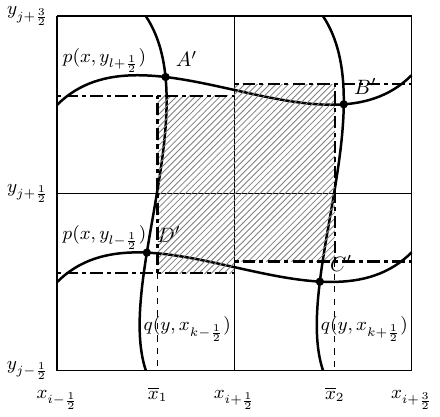} 
    \caption{CCSL scheme grid configuration with standard and reconstructed coordinates.} 
    \label{Fig_CCSL_error} 
\end{figure}

    To compute the numerical approximation $M_{\text{CCSL}}$, we have
    \begin{equation}
    \label{Eq. mass of CCSL}
        \begin{aligned}
            M_{\text{CCSL}}
            &= m_{I_{i+1,l}}^{(1)} + m_{I_{i,l}} - m_{I_{i,l}}^{(1)} \\
            &= G_{i+1,l}^{3}(\overline{x}_2) + m_{I_{i,l}} - G_{i,l}^{3}(\overline{x}_1),
        \end{aligned}
    \end{equation}
    where $m_{I_{i,l}}$ denotes the total mass within the intermediate cell $I_{i,l}$ bounded by $x_{i-\frac{1}{2}}$ and $x_{i+\frac{1}{2}}$. The quantity $m_{I_{i,l}}^{(s)}$ represents the mass of the $s$-th clipped segment 
    of the intermediate cell $I_{i,l}$ (counted from left to right) that appears during the backtracked-cell reconstruction process, as illustrated in Fig.~\ref{Fig_CCSL_error}. 

    The averaged horizontal coordinates of the backtracked corner points are given by
    \[
    \overline{x}_1 = \tfrac{1}{2}(x_{A'} + x_{D'}), 
    \qquad 
    \overline{x}_2 = \tfrac{1}{2}(x_{B'} + x_{C'}).
    \]
    Here, $G_{i,l}^{3}(x)$ denotes the cubic reconstruction function in the $x$-direction, constructed from the local intermediate-cell masses $\{m_{I_{i-1,l}},\, m_{I_{i,l}},\, m_{I_{i+1,l}}\}$. It satisfies the following interpolation conditions:
    \[
    \begin{aligned}
    &G_{i,l}^{3}(x_{i-\frac{3}{2}}) = -m_{I_{i-1,l}}, \qquad
    &&G_{i,l}^{3}(x_{i-\frac{1}{2}}) = 0,\\
    &G_{i,l}^{3}(x_{i+\frac{1}{2}}) = m_{I_{i,l}} - m_{I_{i-1,l}}, \qquad
    &&G_{i,l}^{3}(x_{i+\frac{3}{2}}) = m_{I_{i+1,l}} + m_{I_{i,l}} - m_{I_{i-1,l}}.
    \end{aligned}
    \]
    The function $G_{i,l}^{3}(x)$ thus represents a one-dimensional third-order polynomial used to interpolate the cumulative mass distribution along each row. For higher-order reconstructions, $G_{i,l}^{r}(x)$ is defined analogously using $r$ neighboring cell masses, as detailed in \ref{Ap. Reconstruction of function G}.
    
    From Proposition \ref{Prop_intermediate}, it follows that the mass error between an intermediate cell and its corresponding true intermediate cell is $ \mathcal{O}(h^4) $, i.e.,
    \begin{align*}
        m_{I_{i,l}}=\int_{x_{i-\frac{1}{2}}}^{x_{i+\frac{1}{2}}}\int_{p(x,y_{l-\frac{1}{2}})}^{p(x,y_{l+\frac{1}{2}})}f(x,y,t^n)\,dy\,dx+\mathcal{O}(h^4).
    \end{align*}

    Let $ M^* $ denote the mass of the region formed by clipping and splicing the true intermediate cells; then
    \begin{equation}
        \label{Eq. true intermediate mass}
        M^*=\int_{\overline{x}_1}^{\overline{x}_2}\int_{p(x,y_{l-\frac{1}{2}})}^{p(x,y_{l+\frac{1}{2}})}f(x,y,t^n)\,dy\,dx.
    \end{equation}
    In this case, the error between $M^*$ and $M_{\text{CCSL}}$ is given by
    \begin{equation*}
        M_{\text{CCSL}}-M^*=\epsilon=\mathcal{O}(h^4),
    \end{equation*}
    where $\epsilon$ encompasses both interpolation error and geometric error. When the interpolation accuracy is sufficiently high, the geometric error dominates.

    By Proposition \ref{Prop_T}, we can transform the region $A'B'C'D'$ by the mapping $T$; in this case, the true mass of the backtracked region $M_{\text{true}}$ is expressed as
    \begin{align}
    \label{Eq. after mapping true}
        M_{\text{true}}
        &= \iint_{A'B'C'D'} f(x,y,t^{n}) \, dx\,dy \\
        &= \iint_{T:A'B'C'D'} f\!\bigl(x, c_x^{-1}(y), t^n\bigr)\, |J_T|\, dy\, dx \notag\\
        &= \int_{y_1^*}^{y_2^*}
           \int_{q^*(y,x_{k-\frac{1}{2}})}^{q^*(y,x_{k+\frac{1}{2}})}
           f^*(x,y,t^n)\, dx\, dy. \notag
    \end{align}
    Here,
    \[
        f^*(x,y,t^n)
        = f\!\bigl(x,\,c_x^{-1}(y),\,t^n\bigr)\, |J_T|,
    \]
    denotes the transformed distribution function under the mapping $T$. Based on the property that the mapping $ T $ preserves the horizontal coordinate, we have
    \begin{equation*}
        \begin{aligned}
        q^*(y_2^*,x_{k+\frac{1}{2}})=x_{A'},\,q^*(y_2^*,x_{k-\frac{1}{2}})=x_{B'},\\
        q^*(y_1^*,x_{k-\frac{1}{2}})=x_{C'},\,q^*(y_1^*,x_{k+\frac{1}{2}})=x_{D'}.\\
        \end{aligned}
    \end{equation*}
    
    Similarly, applying the mapping $ T $ to the right-hand side of Eq. \eqref{Eq. true intermediate mass}, we have
    \begin{equation}
        \label{Eq. after mapping true inter}
        M^*=\int_{y_1^*}^{y_2^*}\int_{\overline{x}_1}^{\overline{x}_2}f^*(x,y,t^n)\,dy\,dx.
    \end{equation}
    
    Subtracting Eq. \eqref{Eq. after mapping true} from Eq. \eqref{Eq. after mapping true inter}, we find that their error form is identical to that of Eq. \eqref{Eq. error_Min}. By Lemma \ref{Lemm:integral_error},
    \begin{equation*}
        \big|M_{\text{true}}-M^*\big|=\mathcal{O}(h^4).
    \end{equation*}
    
    Thus,
        \begin{equation}
        \label{Eq. error between CCSL and true}
            \big|M_{\text{true}}-M_{\text{CCSL}}|\le\big|M_{\text{true}}-M^*|+\big|M^*-M_{\text{CCSL}}|=\mathcal{O}(h^4).
        \end{equation}
        
    Substituting Eq. \eqref{Eq. error between CCSL and true} into Eq. \eqref{Eq. CCSL_error},
    \begin{equation*}
        \left| \frac{1}{h^2} \iint_{\Omega_{k,l}} f(x,y,t^{n+1}) \, dx\,dy - \overline{f}_{k,l}^{n+1} \right|=\frac{1}{h^2} \left|M_{\text{true}} - M_{\text{CCSL}}\right|\le \frac{1}{h^2}\mathcal{O}(h^4)=\mathcal{O}(h^2).
    \end{equation*}
    Therefore, the  error of CCSL is $\mathcal{O}(h^2)$, which completes the proof.
\end{proof}

\subsection{Enhanced CCSL method}

\subsubsection{Freestream preservation}
The CCSL method, in its original form, fails to satisfy the freestream preservation property. This deficiency can lead to numerical instabilities, as it introduces spurious sources or sinks in the discretization of a uniform flow field. As emphasized by Colella et al. (2011, p. 2953) \cite{COLELLA20112952}, ''Freestream preservation is an important requirement for the discretization of conservation laws in mapped coordinates. This property ensures that a uniform flow is unaffected by the choice of mapping and discretization.'' Therefore, improvements to the CCSL method are mandatory. A question arises: if we use smooth curves instead of straight lines to approximate the backtracked region, can we achieve freestream preservation while reducing geometric errors. In fact, our analysis shows that under the premise of retaining the original CCSL framework—i.e., performing interpolation dimension by dimension—enhancing geometric approximation does not effectively reduce geometric errors. We refer to the CCSL method that uses curve approximation as the CCSL' method.

Without loss of generality, we only need to compare the mass errors of the intermediate cells between the CCSL method and the CCSL' method. Using the same proof method and mesh as in Proposition \ref{Prop_intermediate}, for $ F_s,\;s=1,2,$ in Eq. \eqref{Eq. primitive function}, the mass of the intermediate cell calculated by the CCSL' method, denoted as $ M_{\text{CCSL'}} $, is 
\begin{equation*}
    \begin{aligned}
        M_{\text{CCSL'}}=&\int_0^1F_2(p(x_{i-\frac{1}{2}}+sh,y_{l+\frac{1}{2}})\,ds+m_{i,j}-\int_0^1F_1(p(x_{i-\frac{1}{2}}+sh,y_{l-\frac{1}{2}})\,ds\\
        =&\frac{1}{h}\int_{x_{i-\frac{1}{2}}}^{x_{i+\frac{1}{2}}}\int_{x_{i-\frac{1}{2}}}^{x_{i+\frac{1}{2}}}\int_{p(u,y_{l-\frac{1}{2}})}^{p(u,y_{l+\frac{1}{2}})}f(x,y)\,dy\,dx\,du+R^{2d+1},
    \end{aligned}
\end{equation*}
where $R^{2d+1}$ is the interpolation error of order $2d+1$.

At this point, the error between $ M_{\text{CCSL}'} $ and $ M_{\text{true}} $ is:
\begin{equation}
    \label{Eq. error of CCSL'}
    \begin{aligned}
    |M_{\text{CCSL'}}-M_{\text{true}}|=&\Bigg|\frac{1}{h}\int_{x_{i-\frac{1}{2}}}^{x_{i+\frac{1}{2}}}\int_{x_{i-\frac{1}{2}}}^{x_{i+\frac{1}{2}}}\int_{p(u,y_{l-\frac{1}{2}})}^{p(u,y_{l+\frac{1}{2}})}f(x,y)\,dy\,dx\,du+R^{2d+1}\\
    &-\int_{x_{i-\frac{1}{2}}}^{x_{i+\frac{1}{2}}}\int_{p(x,y_{l-\frac{1}{2}})}^{p(x,y_{l+\frac{1}{2}})} f(x,y)\,dy\,dx\Bigg|\\
    \leq&\Bigg|\frac{1}{h}\int_{x_{i-\frac{1}{2}}}^{x_{i+\frac{1}{2}}}\int_{x_{i-\frac{1}{2}}}^{x_{i+\frac{1}{2}}}\int_{p(x,y_{l+\frac{1}{2}})}^{p(u,y_{l+\frac{1}{2}})}f(x,y)\,dy\,dx\,du\\
    &-\frac{1}{h}\int_{x_{i-\frac{1}{2}}}^{x_{i+\frac{1}{2}}}\int_{x_{i-\frac{1}{2}}}^{x_{i+\frac{1}{2}}}\int_{p(x,y_{l-\frac{1}{2}})}^{p(u,y_{l-\frac{1}{2}})}f(x,y)\,dy\,dx\,du\Bigg|+|R^{2d+1}|.
    \end{aligned}
\end{equation}

Similarly, we introduce Lemma \ref{Lemm:integral_error_curve}:
\begin{lemma}
    \label{Lemm:integral_error_curve}
    Let $p\in C^2([x_1, x_2])$ and $f\in C^2$ in a neighborhood of the curve $ (u, p(u)) $ $u\in[x_1,x_2]$. The integral error  
    \[
    \epsilon = \frac{1}{h}\int_{x_1}^{x_2}\int_{x_1}^{x_2} \int^{p(t)}_{p(u)} f(u, v) \, dv\,dt \, du
    \]  
    admits the asymptotic expansion:  
    \[
    \begin{aligned}
    \epsilon = -\frac{1}{12}\frac{\partial f}{\partial u}p'\bigg|_{x=x_1,y=p(x_1)} h^3 + \mathcal{O}(h^4),
    \end{aligned}
    \]  
    where $ h = x_{2} - x_{1} $.  
\end{lemma}

Using Lemma \ref{Lemm:integral_error_curve} to evaluate Eq. \eqref{Eq. error of CCSL'}, we obtain:
\begin{equation*}
    |M_{\text{CCSL'}}-M_{\text{true}}|=\mathcal{O}(h^4).
\end{equation*}

It can be observed that the error order is identical to that in Proposition~\ref{Prop_intermediate}. Therefore, a more accurate geometric approximation does not reduce the geometric error. This occurs because the reconstruction is performed through successive one-dimensional interpolations, which inherently cannot capture multi-dimensional correlations within the flow field. Furthermore, achieving freestream preservation with a more complex geometric approximation would necessitate a genuinely multi-dimensional interpolation scheme, which is computationally expensive and thus not cost-effective. Consequently, we retain the original CCSL framework and instead fine-tune the boundaries of the intermediate and backtracked cells to ensure freestream preservation.

We denote the volume of the intermediate cell as $ V_{I_{ij}} $, with its corresponding upper and lower boundaries given by $\overline{y}_{i,j\pm\frac{1}{2}}$. For the backtracked cell, we denote its volume as $ V_{B_{ij}} $ (the symbol can be adjusted based on specific notation conventions) and its corresponding left and right boundaries as $\overline{x}_{i\pm\frac{1}{2},j} $. Under periodic boundary conditions, the following properties hold:
\begin{equation}
\label{Eq. inter property}
    \displaystyle\sum_{i=0}^{N_x-1}\displaystyle\sum_{j=0}^{N_y-1}V_{I_{ij}}=N_xN_y\Delta x\Delta y,\;\displaystyle\sum_{j=0}^{N_y-1}V_{I_{ij}}=N_y\Delta x\Delta y,\,\;\displaystyle\sum_{i=0}^{N_x-1}V_{I_{ij}}=N_x\Delta x\Delta y.\
\end{equation}

The first two properties are naturally satisfied. We fix the upper boundary $\overline{y}_{i,l+\frac{1}{2}} $ of the intermediate cells in a specific middle layer $l$, and then adjust the boundaries of intermediate cells from this middle layer toward both sides.
Let's denote $\Delta V_j = \displaystyle\sum_{i=0}^{N_x-1} V_{I_{ij}} - N_x \Delta x \Delta y \;,j=0,...,N_y-1$. For $j$ ranging from $l+1$ to $N_y-1$ (and from $ l $ to $0$):
\begin{equation*}
\overline{y}_{i,j+\frac{1}{2}}^* = \overline{y}_{i,j+\frac{1}{2}} + \frac{1}{N_x} \sum_{c=l+1}^j \Delta V_c \;\left( \text{and }\; \overline{y}_{i,j-\frac{1}{2}}^* = \overline{y}_{i,j+\frac{1}{2}} - \frac{1}{N_x} \sum_{c=j}^l \Delta V_c \right).
\end{equation*}
After the above adjustments, the intermediate cells satisfy Eq. \eqref{Eq. inter property}.

Subsequently, for each row, we fix the right boundary $\overline{x}_{k+\frac{1}{2},j} $ of a specific middle backtracked cell. For $i$ ranging from $k+1$ to $N_x-1 $ (and from $ k $ to $ 0 $), we recalculate each boundary $\overline{x}_{i+\frac{1}{2},j}$ ($\overline{x}_{i-\frac{1}{2},j}$) such that the volume of the backtracked cell satisfies $ V_{B_{ij}} = \Delta x \Delta y $. Although this adjustment method does not achieve high precision, it is extremely simple and effective. Its effectiveness will be demonstrated in the subsequent numerical test.
\subsubsection{Maximum principle limiter}

Preserving the maximum principle is essential for the Conservative Semi-Lagrangian (CSL) method. In the classical CSL framework, the interpolation function $F_i^{2d+1}(x)$ determines the local remapping accuracy. When the interpolation order is linear ($d = 0$), the scheme automatically satisfies the maximum principle, but suffers from excessive numerical diffusion. Conversely, high-order interpolations ($d \ge 1$) provide higher accuracy but often violate the maximum principle by generating nonphysical undershoots or overshoots. To combine the advantages of both schemes, we introduce a blending limiter that adaptively mixes the linear and high-order reconstructions.

The modified interpolation function is defined as
\begin{equation*}
    F_i^*(x) = \alpha_i F_i^{1}(x) + (1-\alpha_i) F_i^{2d+1}(x), \qquad d \ge 1,
\end{equation*}
where $\alpha_i \in [0,1]$ is a local smoothness indicator that controls the balance between monotonicity and accuracy. When $\alpha_i = 1$, the method reduces to the linear CSL interpolation, strictly preserving the maximum principle; when $\alpha_i = 0$, the full high-order interpolation is recovered.

To ensure that $F_i^*(x)$ respects the maximum principle, $\alpha_i$ is computed from the most restrictive local constraint among all subintervals of $(x_{i-\frac{1}{2}}, x_{i+\frac{1}{2}})$. Let $\{\widetilde{x}_c\}_{c=0}^{k+1}$ denote the set of auxiliary points defined as
\begin{equation*}
    \begin{cases}
        \widetilde{x}_0 = x_{i-\frac{1}{2}}, \\
        \widetilde{x}_c = x^*_{j+\frac{1}{2}+c}, & 1 \le c \le k, \\
        \widetilde{x}_{k+1} = x_{i+\frac{1}{2}},
    \end{cases}
\end{equation*}
where $x^*_{j+\frac{1}{2}+c}$ are the backtracked points lying within the cell $(x_{i-\frac{1}{2}}, x_{i+\frac{1}{2}})$, and $k$ is their count.

For each subinterval $(\widetilde{x}_c, \widetilde{x}_{c+1})$, a local coefficient $\alpha_{ic}$ is determined by
\begin{align*}
    \alpha_{ic} = \min \Bigg\{ \alpha \in (0,1) \Bigg| \;
    m_{\min}\frac{\widetilde{x}_{c+1}-\widetilde{x}_c}{\Delta x}
    \le \alpha \!\left[ F_i^{1}(\widetilde{x}_{c+1}) - F_i^{1}(\widetilde{x}_c) \right] \\
    + (1-\alpha)\!\left[ F_i^{2d+1}(\widetilde{x}_{c+1}) - F_i^{2d+1}(\widetilde{x}_c) \right]
    \le m_{\max}\frac{\widetilde{x}_{c+1}-\widetilde{x}_c}{\Delta x}
    \Bigg\},
\end{align*}
where $m_{\min}$ and $m_{\max}$ are the minimum and maximum physical masses corresponding to the initial extrema of $f$, i.e.,
$m_{\min} = f_{\min}\Delta x \Delta y$ and $m_{\max} = f_{\max}\Delta x \Delta y$. 
Finally, the limiting coefficient is obtained as
\begin{equation*}
    \alpha_i = 
    \begin{cases}
        1, & k = 0, \\
        \displaystyle\max_{0 \leq c \leq k} \alpha_{ic}, & k \ge 1.
    \end{cases}
\end{equation*}

This limiter enforces the maximum principle without sacrificing the conservative property of the CSL method, as both $F_i^{1}(x)$ and $F_i^{2d+1}(x)$ are constructed in a conservative form. 
In smooth regions, $\alpha_i$ remains small, allowing high-order accuracy; near discontinuities, $\alpha_i$ approaches unity, suppressing unphysical oscillations.

\section{Numerical tests}
\label{sec_results}
In the numerical tests, unless specified, a 5th-degree Lagrange interpolation is used. This corresponds to the case \(2d+1 = 5\) in Section~\ref{Sec. SL}, i.e., \(d = 2\) in the reconstruction framework. Meanwhile, the improved CCSL method—with the freestream-preserving correction and the maximum-principle limiter—is employed for all simulations.
\subsection{2D Linear transport equations}

In this subsection, two types of linear transport equations are utilized to assess the performance of the CCSL method. Particularly, smooth initial conditions are employed to verify the order of accuracy, which ensures that its numerical accuracy is consistent with the theoretical analysis provided in Proposition \ref{Prop_ CCSL error order}.

Consider the equation
\begin{equation}
\label{Eq. linear equation}
    \begin{aligned}
        \frac{\partial f}{\partial t}+\frac{\partial}{\partial x}(a_1f)+\frac{\partial}{\partial y}(a_2f)=0,\;(x,y)\in[x_{\text{min}},x_{\text{max}}]\times[y_{\text{min}},y_{\text{max}}],
    \end{aligned}
\end{equation}
where $\mathbf{a}=(a_1,a_2)$ denotes the velocity field (divergence free), and $[x_{\text{min}}, x_{\text{max}}] \times [y_{\text{min}}, y_{\text{max}}]$ represents the computational domain.

\subsubsection{Rotation}
\label{Sec. rotation}
For the rotation test, zero boundary conditions are employed. The computational domain is $[-\pi,\pi]^2$. The initial condition is defined as:
\begin{equation}
\label{Eq. linear equation initial}
    f(x,y,0)=
    \begin{cases}
        r_0\cos(\dfrac{\pi r(x,y)}{2r_0})^6,\;&r(x,y)<r_0,\\
        0&\text{otherwise,}
    \end{cases}
\end{equation}
where $r_0=0.3\pi$, $r(x,y)=\sqrt{(x-x_0)^2+(y-y_0)^2}$ and $(x_0,y_0)=(0.3\pi,0)$. The velocity field is set as:
\begin{equation}
\label{Eq. rotation}
    \mathbf{a}(x, y) = \bigg(-\frac{1}{2}\pi y,\frac{1}{2}\pi x\bigg).
\end{equation}

In Table~\ref{Tab. rotation order}, we present the $L^2$ errors and the corresponding orders of accuracy at times $t=1$ and $t=4$. For this test, the characteristics can be backtracked exactly, so no temporal error is introduced; the numerical error therefore comes purely from spatial discretization, consisting of the geometric error and the interpolation error. Since a fifth-degree Lagrange interpolation is employed, the geometric error becomes the dominant component. At $t=1$, the observed order of accuracy is close to two. Although numerical errors may accumulate over time, the results indicate that the overall accuracy at this time remains second order, which is consistent with the expected behavior of the CCSL method.
At $t=4$, a higher apparent order is observed. We conjecture that this behavior is related to the fact that the flow returns to its initial configuration after a full rotation, causing a partial cancellation of errors. A more detailed analysis of this phenomenon will be pursued in future work.

\begin{table}[htbp]
\centering
\caption{Rotation test: $L^2$ errors and corresponding orders of accuracy of the CCSL method for the Eq. \eqref{Eq. rotation} with the initial condition \eqref{Eq. linear equation initial} at $t=1,4$. The time step $\Delta t=0.25$.}
\label{Tab. rotation order}
\begin{tabular}{c c c c c c c c}
\Xhline{2pt}
\multirow{2}{*}{mesh} & \multirow{2}{*}{CFL} & & \multicolumn{2}{c }{$t=1$} & & \multicolumn{2}{c }{$t=4$}  \\
\cline{4-5} \cline{7-8}
 & & & $L^2$ error & order & &  $L^2$ error & order  \\
\hline
40$\times$40 & 11.10 & & 5.39e-03 & -- & & 1.58e-02 & -- \\

80$\times$80 & 22.20 & & 5.83e-04 & 3.21 & & 4.13e-04 & 5.26 \\

160$\times$160 & 44.40 & & 1.45e-04 & 2.01 & & 7.11e-06 & 5.86  \\

320$\times$320 & 88.80 & & 3.63e-05 & 2.00 & & 3.19e-07 & 4.48  \\

640$\times$640 & 177.59 & & 9.06e-06 & 2.00 & & 2.51e-08 & 3.67  \\

1280$\times$1280 & 355.19 & & 2.29e-06 & 1.98 & & 8.62e-10 & 4.86  \\
\Xhline{2pt}
\end{tabular}
\end{table}

\subsubsection{Swirling deformation flow}
\label{Sec. Swirling}
For the swirling deformation flow test, we adopt the same boundary conditions, initial conditions, and computational domain as those used in Section \ref{Sec. rotation}. The velocity field is defined as follows:
\begin{equation}
\label{Eq. Swirling}
    \mathbf{a}(x,y)=\Big(-2\pi \cos\! ^2(\frac{x}{2})\sin(y)g(t),2\pi\sin(x)\cos\! ^2(\frac{y}{2})g(t)\Big),
\end{equation}
where $g(t)=\cos(\frac{\pi t}{T})$ and $T=2$.

Similarly, in Table~\ref{Tab. sdf order}, we present the $L^2$ errors and the corresponding orders of accuracy at times $t=1$ and $t=4$. In this case, the characteristic equations cannot be solved exactly, so temporal errors are present; the time step is fixed at $\Delta t = 0.125$. At $t=1$, when the deformation of the profile reaches its maximum, the observed order of accuracy remains close to two, consistent with the behavior observed in the rotation test. At $t=4$, the flow returns to its initial configuration, which leads to a noticeably smaller error
compared with the result at $t=1$ under the same grid resolution. We conjecture that this improvement is caused by a partial cancellation of errors after the full cycle. However, due to the presence of temporal errors, the overall order observed at $t=4$ is reduced. A more detailed analysis of this behavior may be pursued in future work.

\begin{table}[htbp]
\centering
\caption{Swirling deformation flow: $L^2$ errors and corresponding orders of accuracy of the CCSL method for the Eq. \eqref{Eq. Swirling} with the initial condition \eqref{Eq. linear equation initial} at $t=1,4$. The time step $\Delta t=0.125$.}
\label{Tab. sdf order}
\begin{tabular}{c c c c c c c c}
\Xhline{2pt}
\multirow{2}{*}{mesh} & \multirow{2}{*}{CFL} & & \multicolumn{2}{c }{$t=1$} & & \multicolumn{2}{c }{$t=4$}  \\
\cline{4-5} \cline{7-8}
 & & & $L^2$ error & order & &  $L^2$ error & order  \\
\hline
40$\times$40 & 5 & & 3.55e-02 & -- & & 5.48e-02 & -- \\

80$\times$80 & 10 & & 6.07e-03 & 2.55 & & 4.13e-03 & 3.73 \\

160$\times$160 & 20 & & 1.40e-03 & 2.12 & & 3.33e-04 & 3.63 \\

320$\times$320 & 40 & & 3.45e-04 & 2.02 & & 9.65e-05 & 1.79  \\

640$\times$640 & 80 & & 8.36e-05 & 2.05 & & 3.25e-05 & 1.57  \\

1280$\times$1280 & 160 & & 1.95e-05 & 2.10 & & 1.32e-05 &  1.3 \\
\Xhline{2pt}
\end{tabular}
\end{table}

To further assess the capability of the CCSL method in capturing solutions with complex structures, let's consider the initial condition to use a discontinuous one:
\begin{equation}
\label{Eq. discontinuous initial}
    f(x,y,0)=
    \begin{cases}
        1,\;& \begin{aligned}[t]
            &\sqrt{x^2+(y-0.5\pi)^2}\le r_0 \;\text{and}\;\\
            &(|x|\ge0.05\pi\;\text{or}\;y\ge0.5\pi),
        \end{aligned}\\
        1-\dfrac{1}{r_0}\sqrt{x^2+(y+0.5\pi)^2},\;&\sqrt{x^2+(y+0.5\pi)^2}\le r_0,\\
        \dfrac{1+\cos\!\Big(\dfrac{\pi}{r_0}\sqrt{(x+0.5\pi)^2+y^2}\Big)}{4},\;&\sqrt{(x+0.5\pi)^2+y^2}\le r_0,\\
        0,\;&\text{otherwise,}
    \end{cases}
\end{equation}
where $r_0=0.3\pi$. In Fig. \ref{Fig_SDF}, we plot the contour maps of the solution under the CCSL method at times $t = 0$, $t = 1$, and $t = 2$. It can be observed that the CCSL method can well capture complex structures and maintain the maximum principle.
\begin{figure}[htbp]
    \centering
    \begin{subfigure}[b]{0.3\textwidth}
        \includegraphics[width=\textwidth]{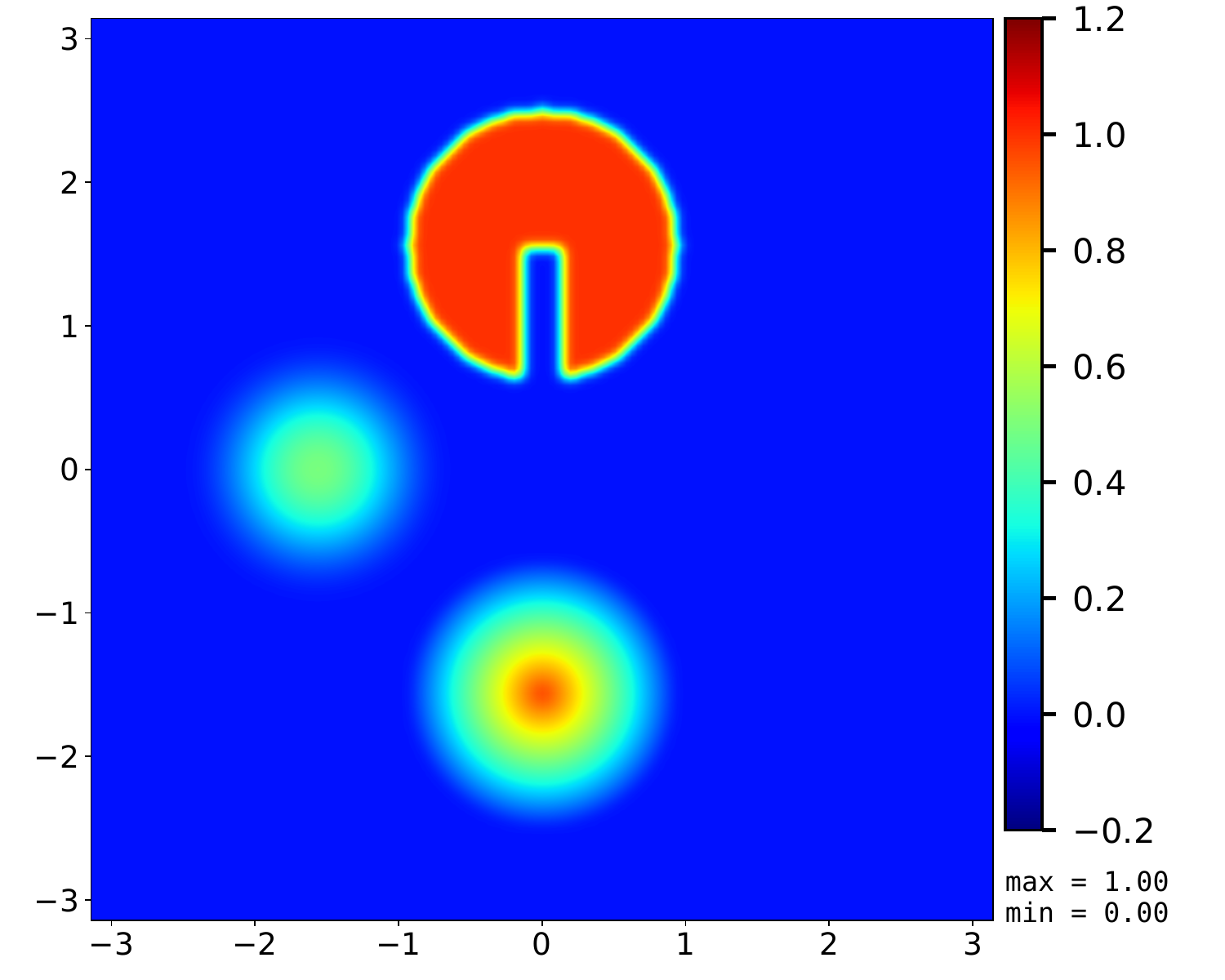}
        \caption{t=0.0}
    \end{subfigure}
    %\hfill
    \begin{subfigure}[b]{0.3\textwidth}
        \includegraphics[width=\textwidth]{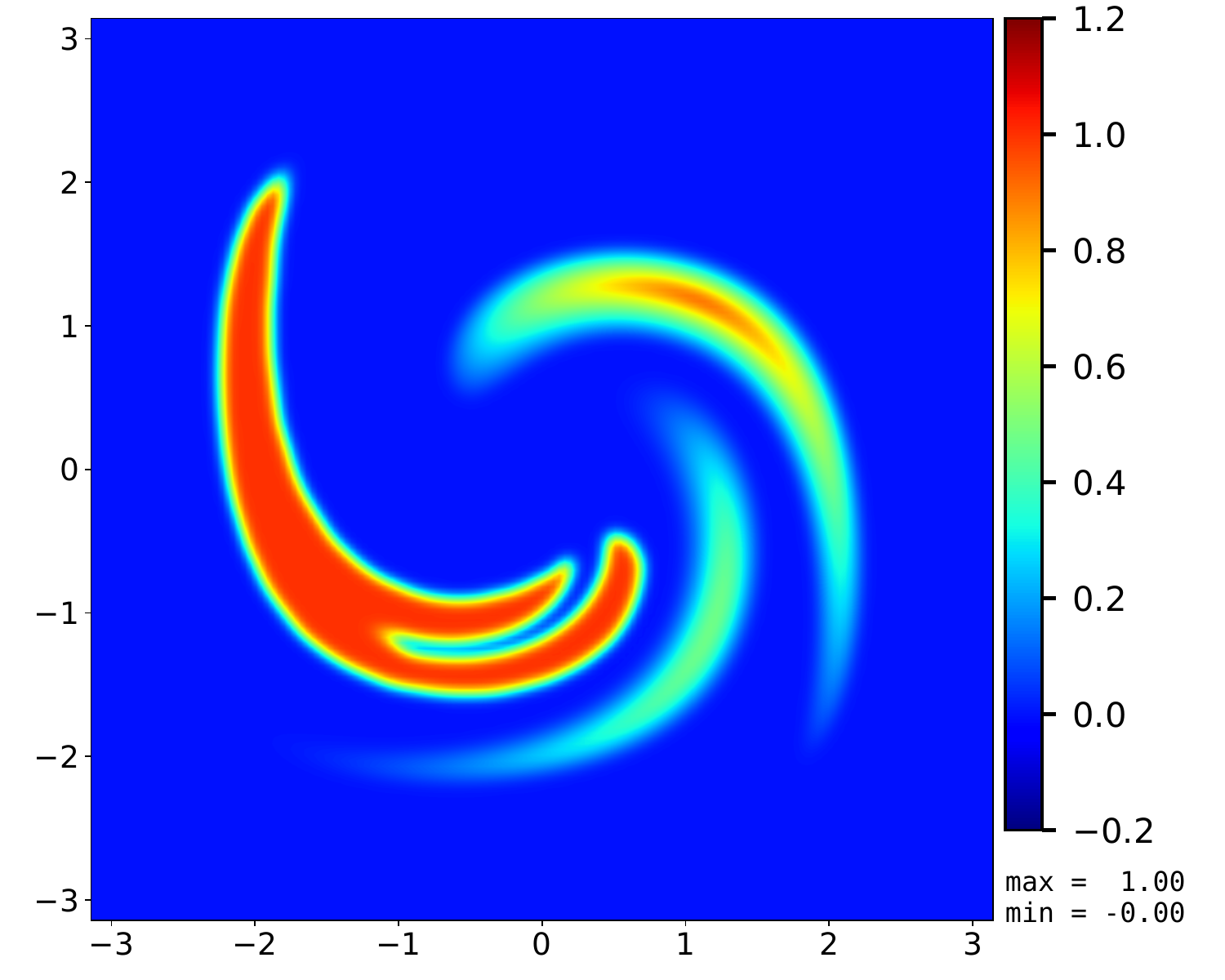}
        \caption{t=1.0}
    \end{subfigure}
    %\hfill
    \begin{subfigure}[b]{0.3\textwidth}
        \includegraphics[width=\textwidth]{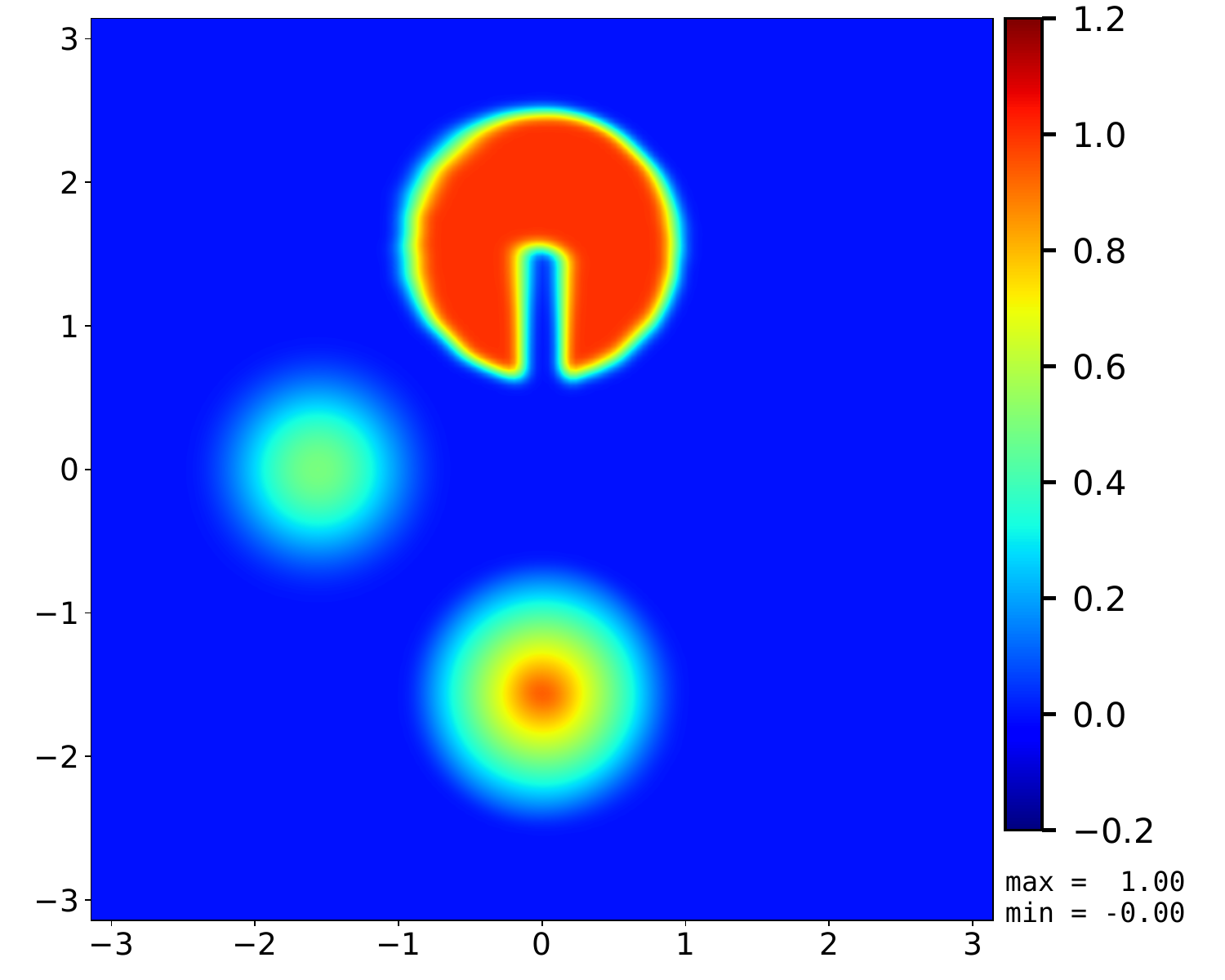}
        \caption{t=2.0}
    \end{subfigure}
    \caption{Swirling deformation flow experiment with complex initial condition Eq. \eqref{Eq. discontinuous initial} ($N_x = N_y = 160,\;\Delta t=0.03125$, $\text{CFL}=5.0$), corresponding to times $t = 0$, $t = 1$, and $t = 2$}
    \label{Fig_SDF}
\end{figure}

\subsection{Guiding center model}
Considering the guiding center model, which is capable of describing highly magnetized plasma in the transverse plane, given by:
\begin{equation*}
    \begin{cases}
        \dfrac{\partial f}{\partial t}+\nabla\cdot(\mathbf{E}^{\perp}f)=0,\\
        -\Delta\phi=f,\;\mathbf{E}^{\perp}=\left(-\dfrac{\partial \phi}{\partial y},\dfrac{\partial\phi}{\partial x}\right),
    \end{cases}
\end{equation*}
where $\mathbf{E}$ is the electric field.

% Unlike the linear transport equation, the Poisson equation needs to be solved for this case to get the electric field. Furthermore, as the electric field varies with time, accurate characteristic backtracking cannot be achieved. Here, the predictor-corrector method is adopted to solve the system of equations, and the Fast Fourier Transform (FFT) solver is used to solve the Poisson equation. Assuming $f^n_{i,j}$ is known and $f^{n+1}_{i,j}$ needs to be computed, the specific steps are as follows:
% \begin{enumerate}
%     \item Solve the Poisson equation using $f^n_{i,j}$ to obtain $\mathbf{E}^n_{i,j}$;
%     \item Utilize $\mathbf{E}^n_{i,j}$ for backtracking over half a time step $\frac{\Delta t}{2}$, converting $f^n_{i,j}$ to $f^*_{i,j}$;
%     \item Solve the Poisson equation using $f^*_{i,j}$ to obtain $\mathbf{E}^*_{i,j}$;
%     \item Use $\mathbf{E}^*_{i,j}$ for backtracking over a full time step $\Delta t$, converting $f^n_{i,j}$ to $f^{n+1}_{i,j}$;
% \end{enumerate}
% One time-step iteration is completed through the above steps.

Unlike the linear transport equation, the Poisson equation needs to be solved for this case to obtain the electric field. Furthermore, as the electric field varies with time, accurate characteristic backtracking cannot be achieved. Here, the predictor-corrector method is adopted to solve the system of equations: the Fast Fourier Transform (FFT) solver is used for solving the Poisson equation under periodic boundary conditions, while the finite difference method is employed for Dirichlet boundary conditions. Assuming $f^n_{i,j}$ is known and $f^{n+1}_{i,j}$ needs to be computed, the specific steps are as follows:
\begin{enumerate}
    \item Solve the Poisson equation using $f^n_{i,j}$ to obtain $\mathbf{E}^n_{i,j}$;
    \item Utilize $\mathbf{E}^n_{i,j}$ for backtracking over half a time step $\frac{\Delta t}{2}$, converting $f^n_{i,j}$ to $f^*_{i,j}$;
    \item Solve the Poisson equation using $f^*_{i,j}$ to obtain $\mathbf{E}^*_{i,j}$;
    \item Use $\mathbf{E}^*_{i,j}$ for backtracking over a full time step $\Delta t$, converting $f^n_{i,j}$ to $f^{n+1}_{i,j}$.
\end{enumerate}
One time-step iteration is completed through the above steps.

We consider the diocotron test case \cite{J2009Non}. The initial condition is given by
\begin{equation*}
    f(x,y,0)=
    \begin{cases}
        (1+\varepsilon\cos(l\theta))\exp{(-4(r-6.5)^2)},\;&r^-\le\sqrt{x^2+y^2}\le r^+,\\
        0,\;&\text{otherwise,}
    \end{cases}
\end{equation*}
where $r = \sqrt{x^2 + y^2}$ and $\theta=\text{atan}2(y,x)$. We set the parameters as $\varepsilon = 0.1$, $r^-=5$, $r^+ = 8$, $l=6$. A square computational domain $[-15, 15]^2$ is adopted, with a uniform grid resolution of $1024 \times 1024$ in the $x$- and $y$-directions. Zero boundary conditions are imposed for both $f$ and the electric potential $\phi$; accordingly, the Poisson equation is solved using the finite difference method. We fix the time step size as $\Delta t = 1$ ($\text{CFL}\approx30$). 

In Fig. \ref{Fig_diocotron curve}, we present the time evolution of the total mass, $L^1$-norm, $L^2$-norm, and total energy conservation for the four methods: the CCSL method, the improved CCSL method, the splitting CSL method, and the BSL method. In Fig. \ref{Fig_mass_error}, compared with the BSL method, the CCSL, the improved CCSL, and the splitting CSL methods all preserve mass conservation. In Fig. \ref{Fig_l1norm_error}, since the improved CCSL scheme preserves the maximum principle, it maintains the $L^1$-norm conservation, exhibiting a significantly better performance than the other three schemes. In Fig. \ref{Fig_l2norm_error}, compared with the splitting CSL method, the other three approaches employ non-splitting formulations, yielding more accurate results. In Fig. \ref{Fig_energy_error}, the three mass-conservative methods show similar behaviors and consistently outperform the BSL method.  

\begin{figure}[htbp]
    \centering
    \begin{subfigure}[b]{0.4\textwidth}
        \centering
        \includegraphics[width=\linewidth]{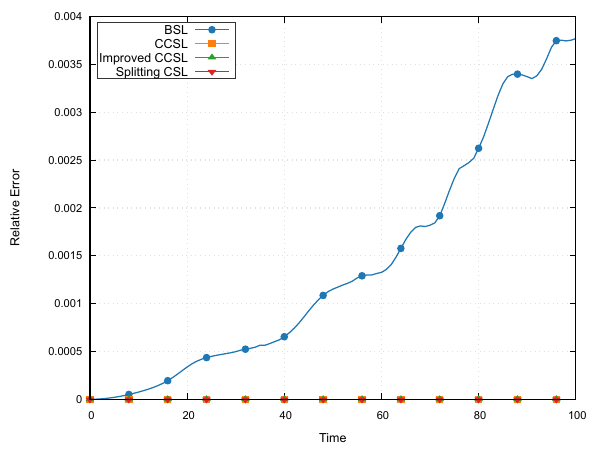}
        \caption{Mass.}
        \label{Fig_mass_error}
    \end{subfigure}
     \vspace{0.5em}
    \begin{subfigure}[b]{0.4\textwidth}
        \centering
        \includegraphics[width=\linewidth]{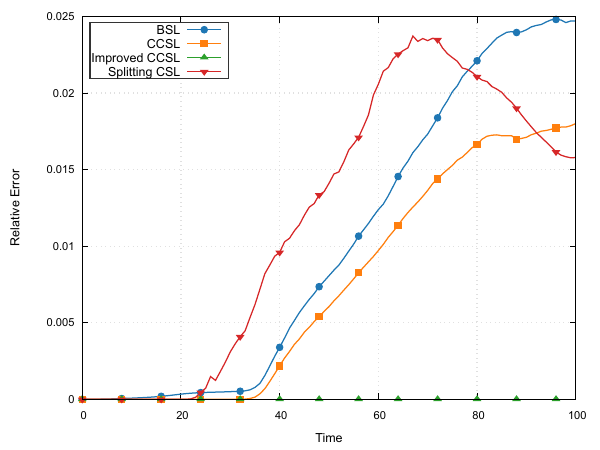}
        \caption{$L^1$-norm.}
        \label{Fig_l1norm_error}
    \end{subfigure}
    
    \begin{subfigure}[b]{0.4\textwidth}
        \centering
        \includegraphics[width=\linewidth]{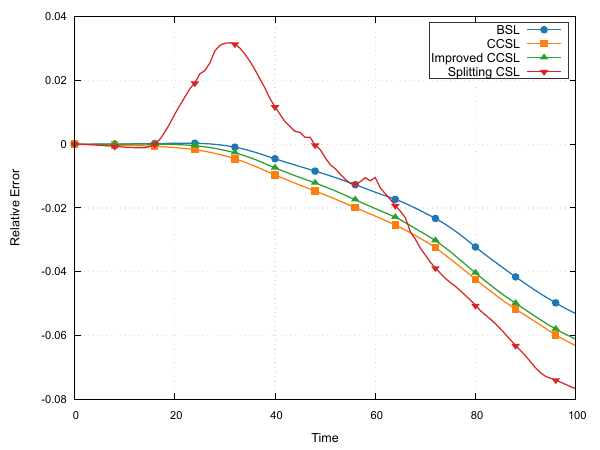}
        \caption{$L^2$-norm.}
        \label{Fig_l2norm_error}
    \end{subfigure}
     \vspace{0.5em}
    \begin{subfigure}[b]{0.4\textwidth}
        \centering
        \includegraphics[width=\linewidth]{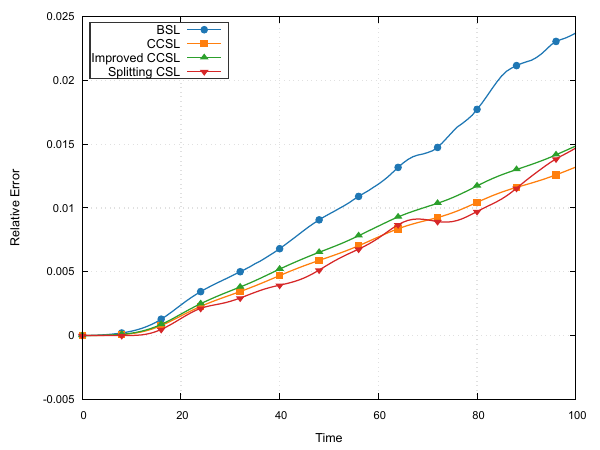}
        \caption{Total energy.}
        \label{Fig_energy_error}
    \end{subfigure}

    \caption{Diocotron test: time evolution of mass, $L^1$-norm, $L^2$-norm, and total energy conservation 
    obtained by four different methods: CCSL, the improved CCSL, the splitting CSL, and BSL.}
    \label{Fig_diocotron curve}
\end{figure}

Fig. \ref{Fig_density_diocotron} displays the density distributions obtained by the improved CCSL and BSL methods at different time spots. Both methods capture filament structures effectively; however, since the improved CCSL scheme preserves the maximum principle, it provides overall better results.

\begin{figure}[htbp]
  \centering
  \setlength{\tabcolsep}{3pt} % 控制左右间距
  \renewcommand{\arraystretch}{0}

  \begin{tabular}{cc}
    \includegraphics[width=0.35\textwidth]{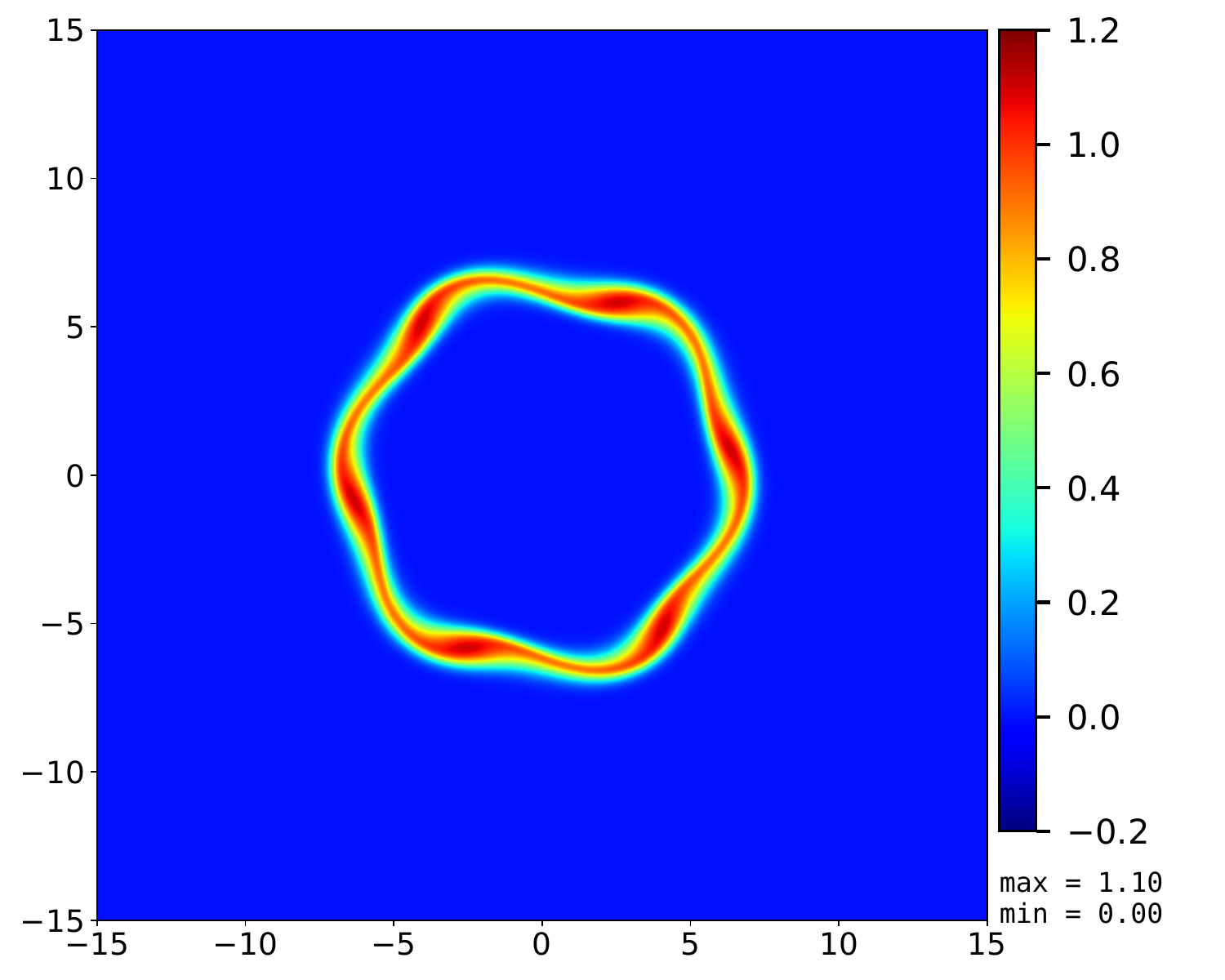} &
    \includegraphics[width=0.35\textwidth]{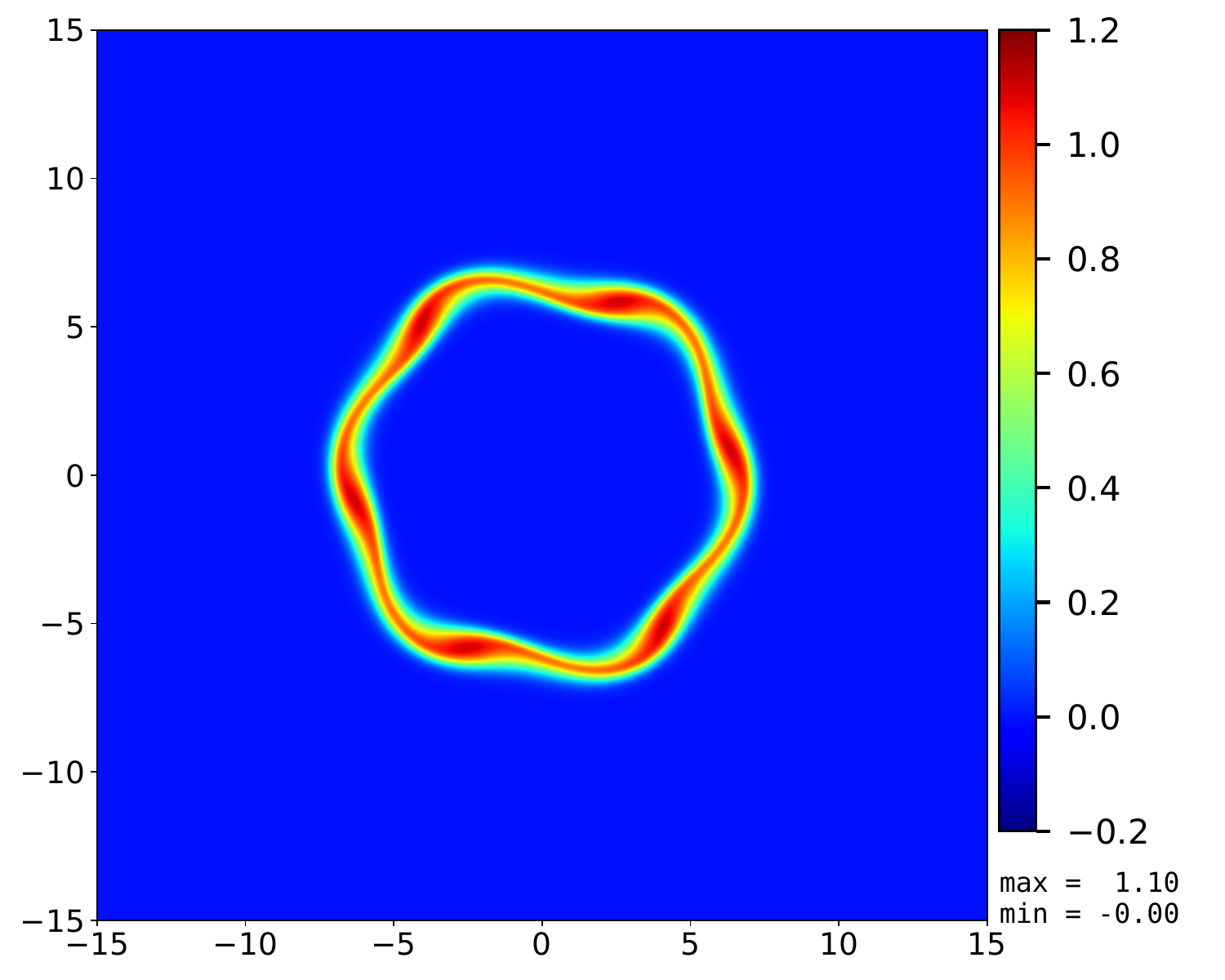} \\
    \includegraphics[width=0.35\textwidth]{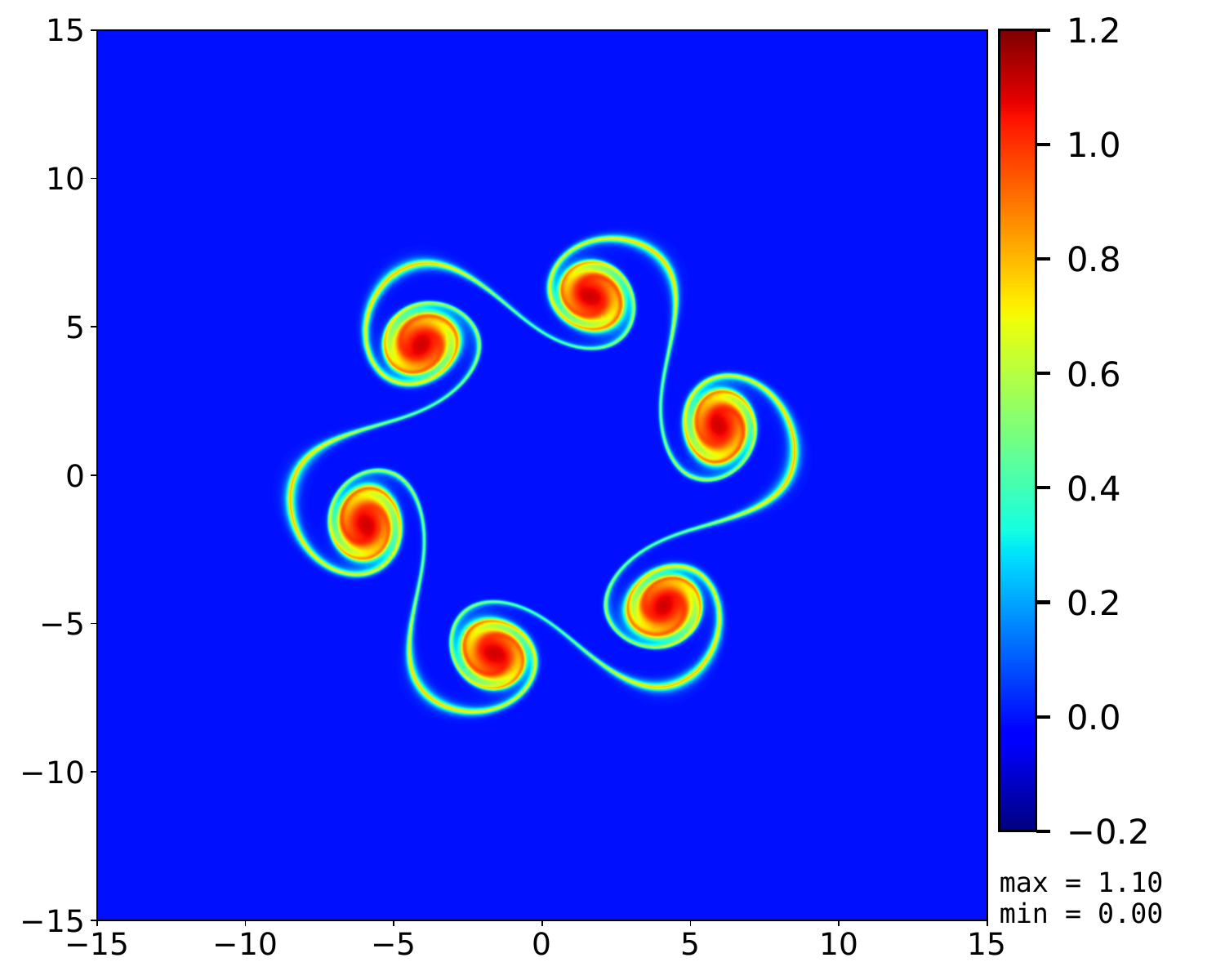} &
    \includegraphics[width=0.35\textwidth]{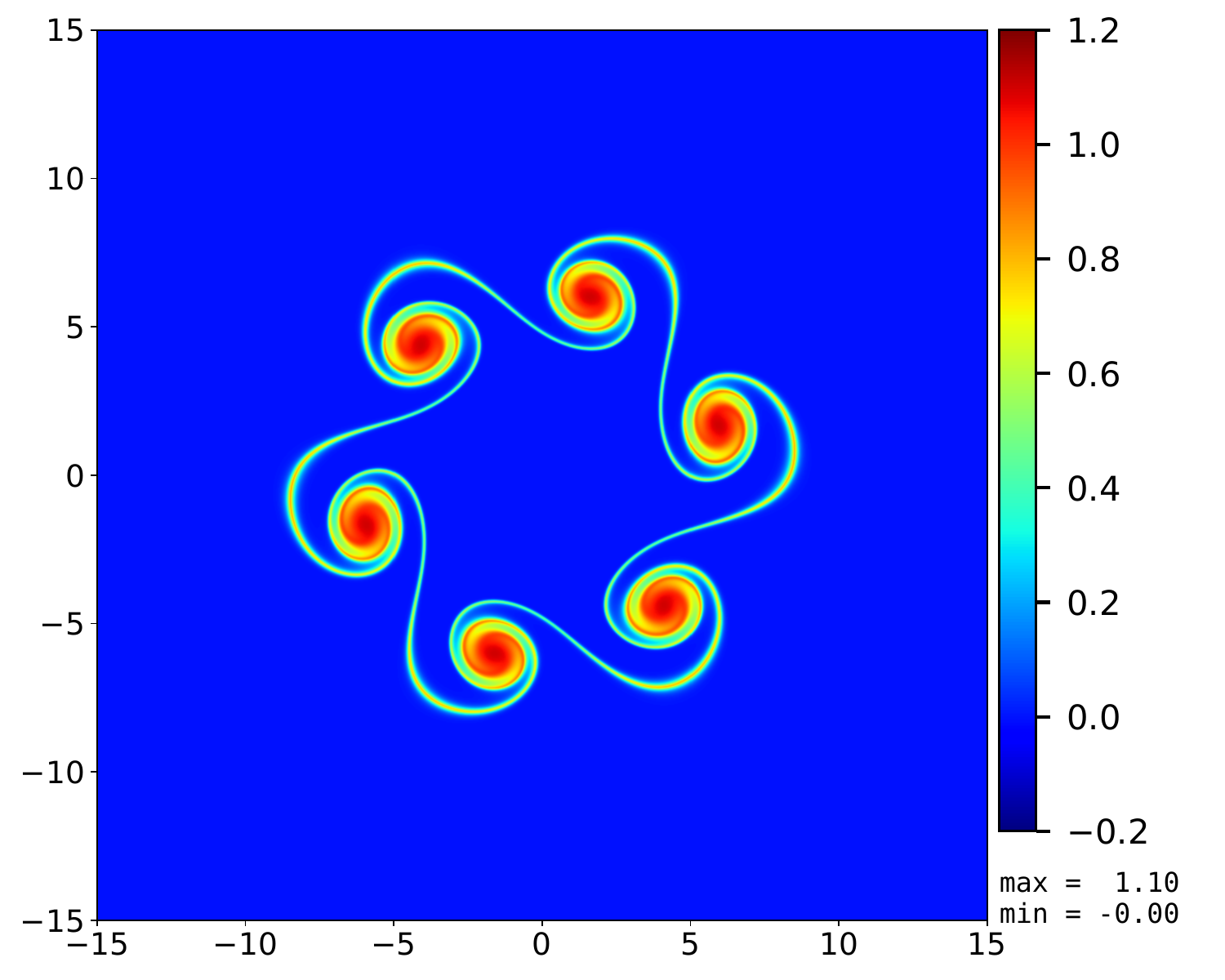} \\
    \includegraphics[width=0.35\textwidth]{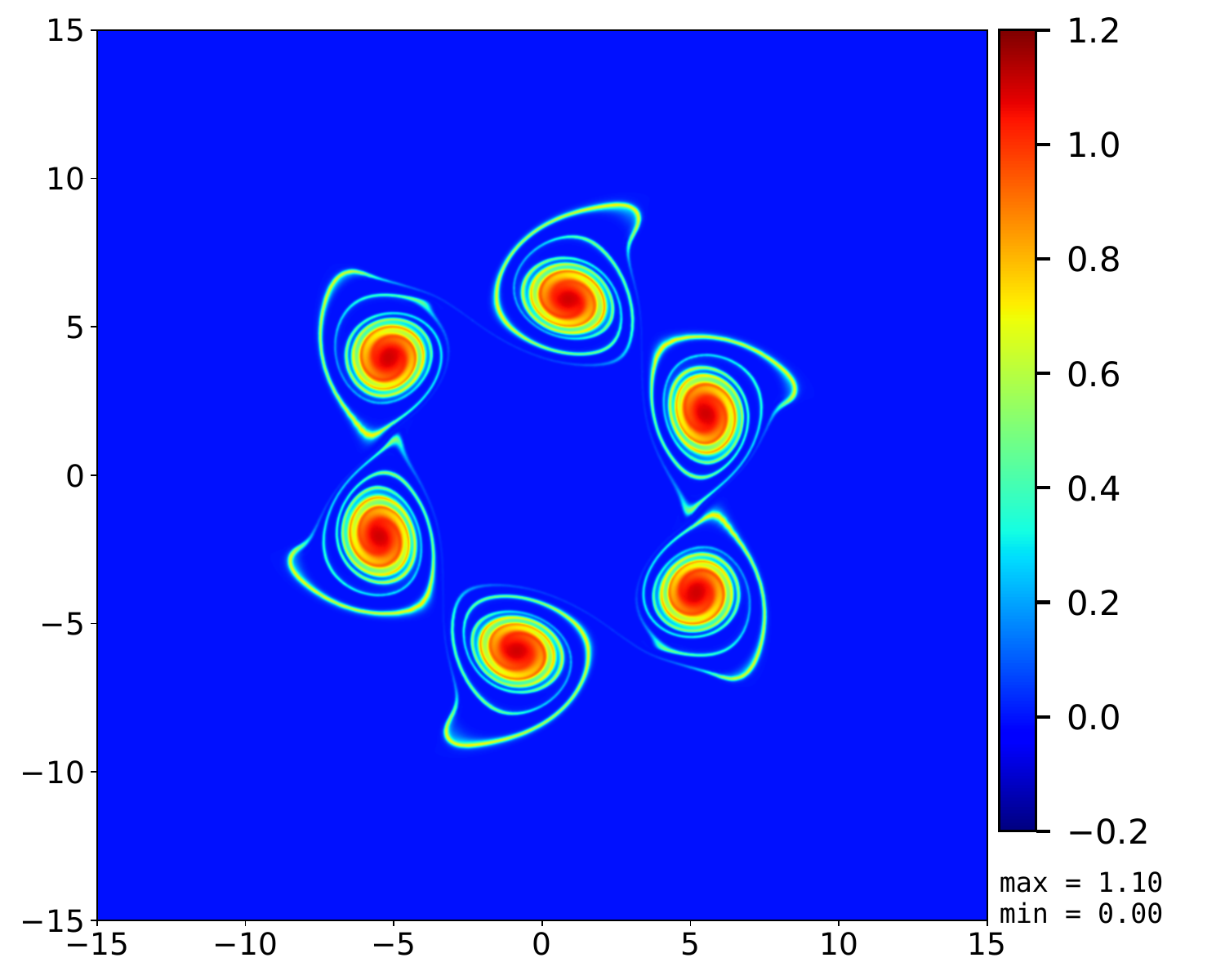} &
    \includegraphics[width=0.35\textwidth]{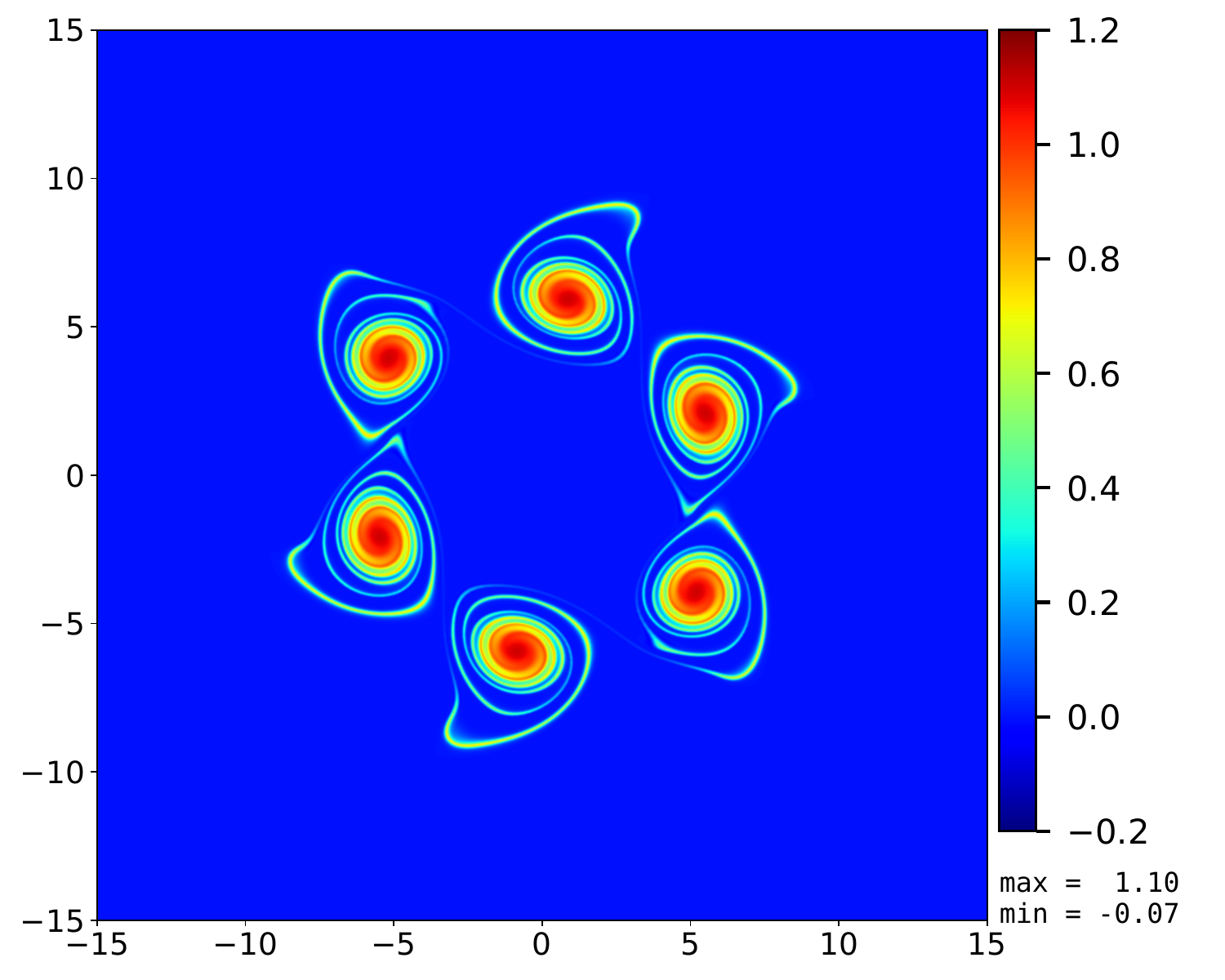} \\
    \includegraphics[width=0.35\textwidth]{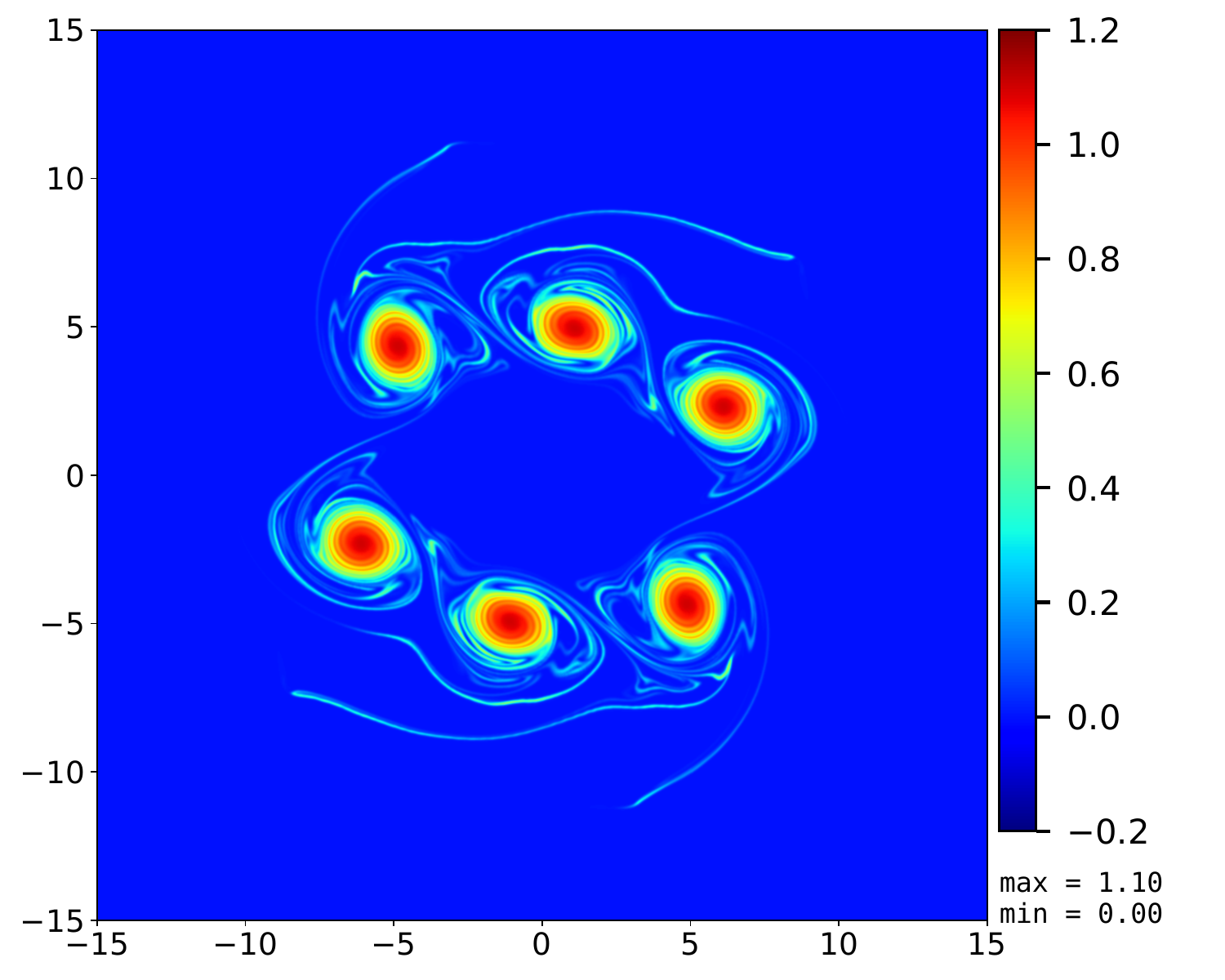} &
    \includegraphics[width=0.35\textwidth]{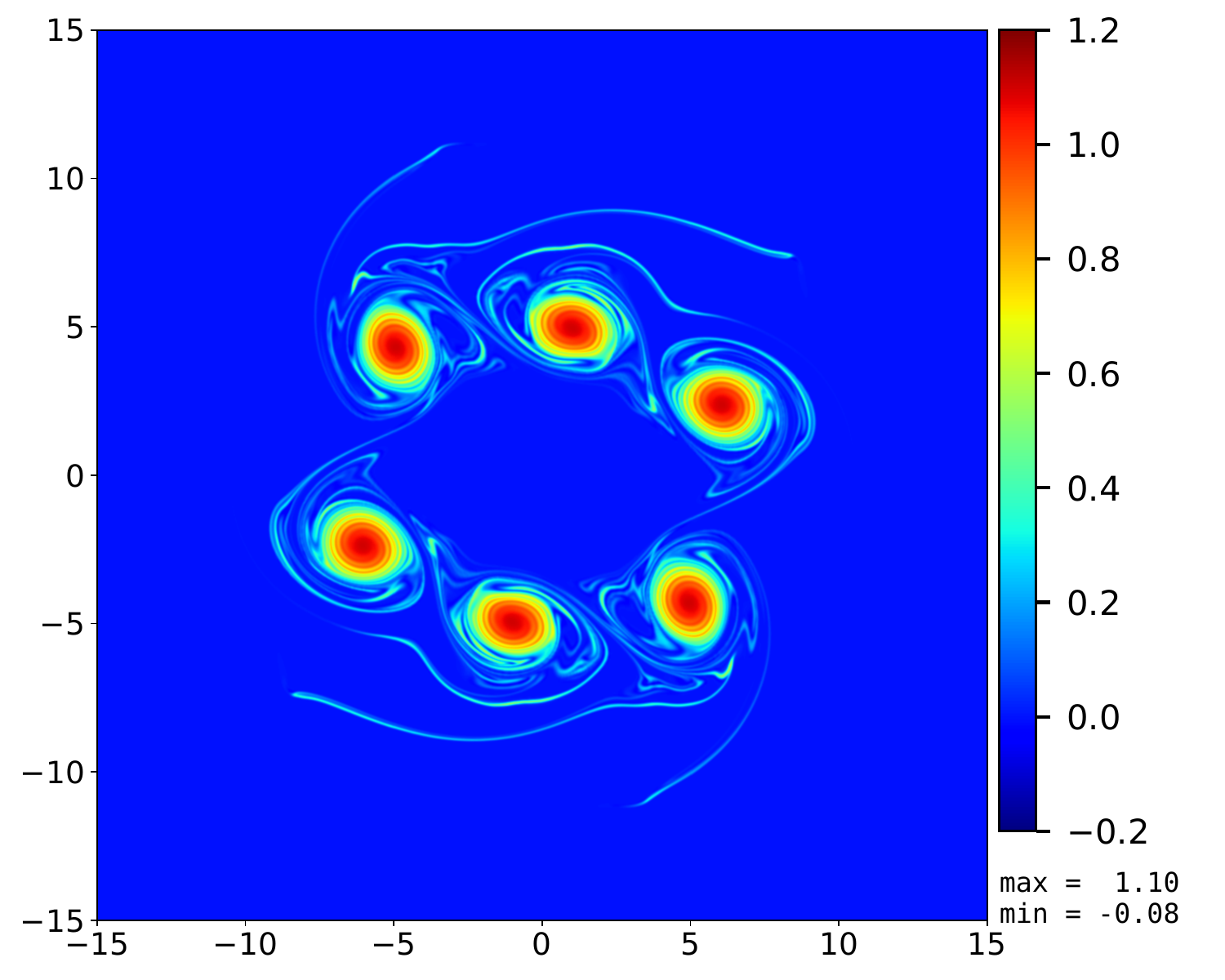} \\
  \end{tabular}

  \caption{Diocotron test: density distributions obtained by the improved CCSL (left column) 
  and BSL (right column) methods at different time spots
  ($t = 10, 30, 50, 100$) with $N_x = N_y = 1024$.}
  \label{Fig_density_diocotron}
\end{figure}

Although the original CCSL method guarantees mass conservation, it fails to preserve the volume of backtracked cells even when the velocity field is divergence free. To illustrate the importance of freestream preservation, we design a modified test based on the guiding center model. 

In this test, the formulation of the guiding center model remains unchanged except for the Poisson equation, which is modified as follows:
\begin{equation}
    \label{Eq. ITG test}
    -\Delta \phi = kf + \varepsilon\exp\left[-\frac{(x/l_x-0.5)^2+(y/l_y-0.5)^2}{2\sigma^2}\right],
\end{equation}
where $k$, $\varepsilon$, and $\sigma$ are constants, and $l_x$ and $l_y$ denote the domain lengths in the $x$- and $y$-directions, respectively. 
To emphasize the influence of the perturbation, we set $k = 10$, $\varepsilon = 0.8$, and $\sigma = 0.1$ in this test.

The computational domain is $[0, 16] \times [0, 8]$, discretized by a uniform grid of $128 \times 64$ cells. Periodic boundary conditions are applied in both directions. The initial condition is defined as $f(x,y,0) = 1$, corresponding to a uniform distribution.

The theoretical solution of this problem remains constant in time when the velocity field is divergence free. However, the original CCSL method fails to maintain this constant solution, exhibiting a noticeable drift from the uniform state $f=1$ as time evolves. This deviation originates from the loss of freestream preservation in the backtracked cells. In contrast, the improved CCSL scheme—with its freestream-preserving correction—maintains the constant solution exactly, as expected for a divergence-free flow field. 

Fig. \ref{Fig_ITG_CCSL_compare} illustrates the numerical results obtained with the original CCSL method. As time increases, the solution progressively departs from the uniform state and eventually becomes unstable, highlighting the necessity of enforcing freestream preservation in the numerical scheme.
\begin{figure}[htbp]
  \centering
  \setlength{\tabcolsep}{3pt} % 调整左右间距
  \renewcommand{\arraystretch}{0}

  \begin{tabular}{cc}
    \includegraphics[width=0.35\textwidth]{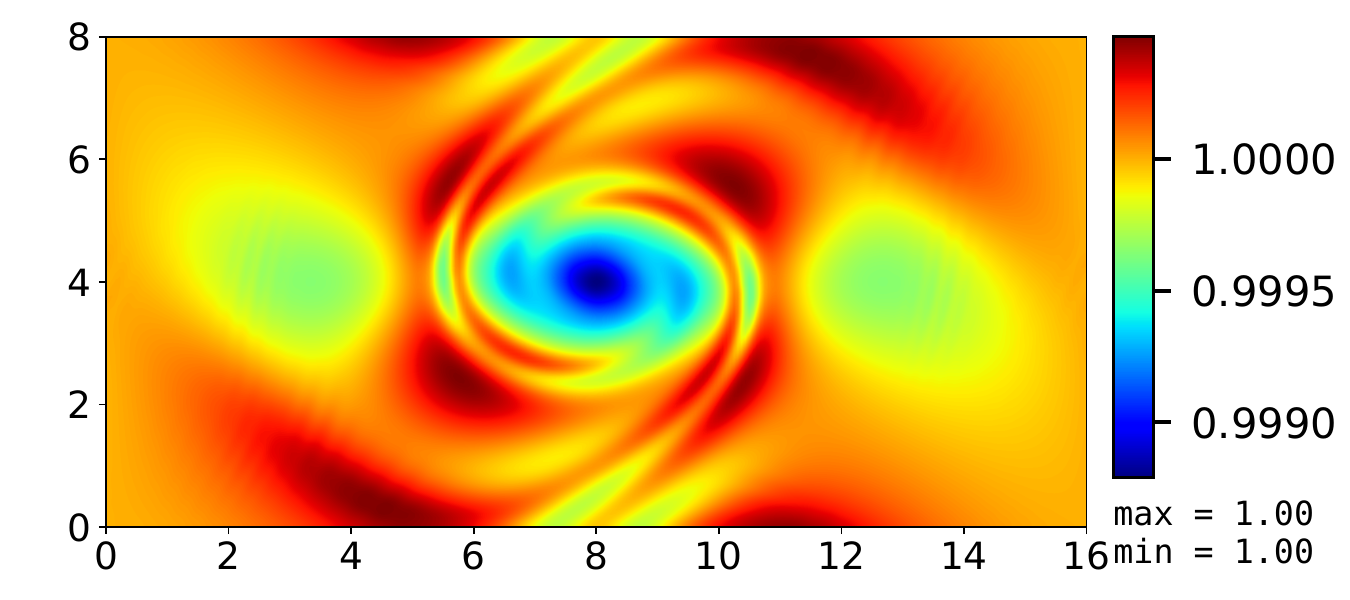} &
     \includegraphics[width=0.35\textwidth]{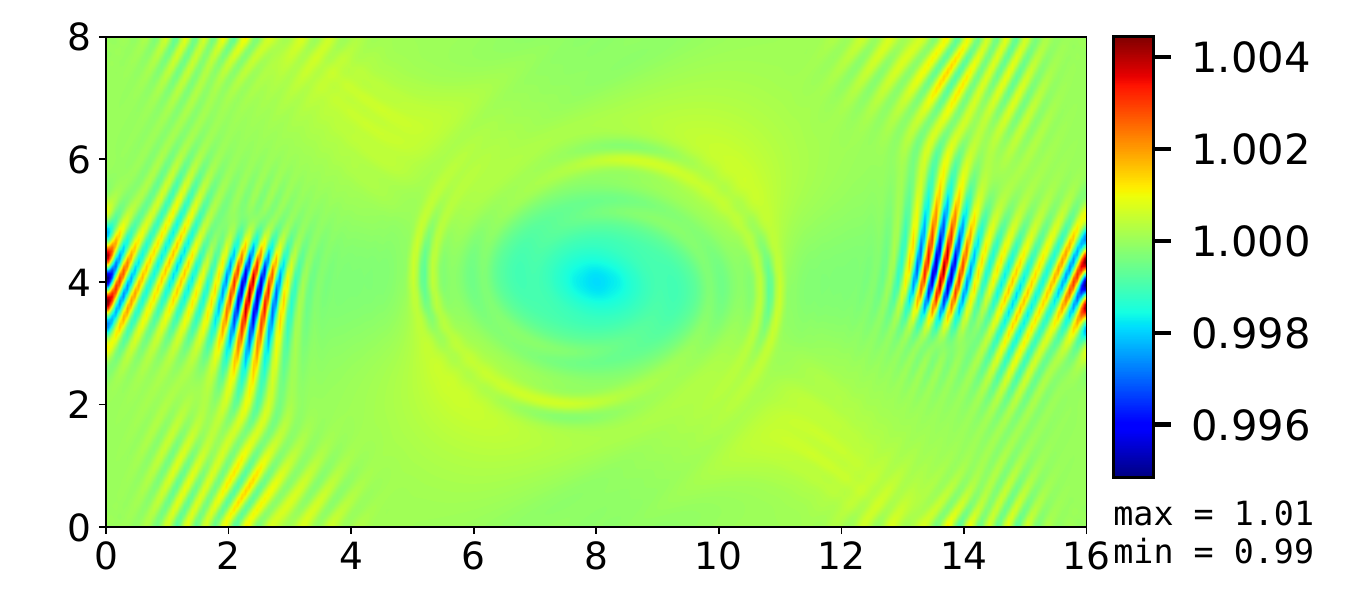}  \\
   \includegraphics[width=0.35\textwidth]{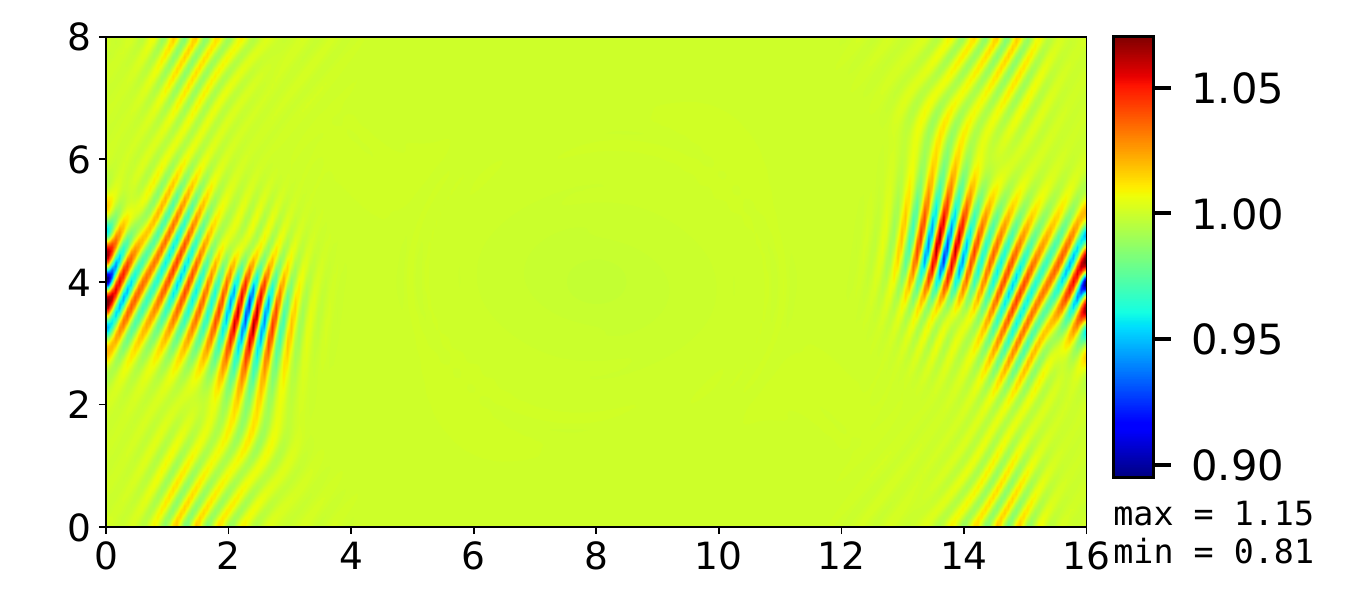} &
   \includegraphics[width=0.35\textwidth]{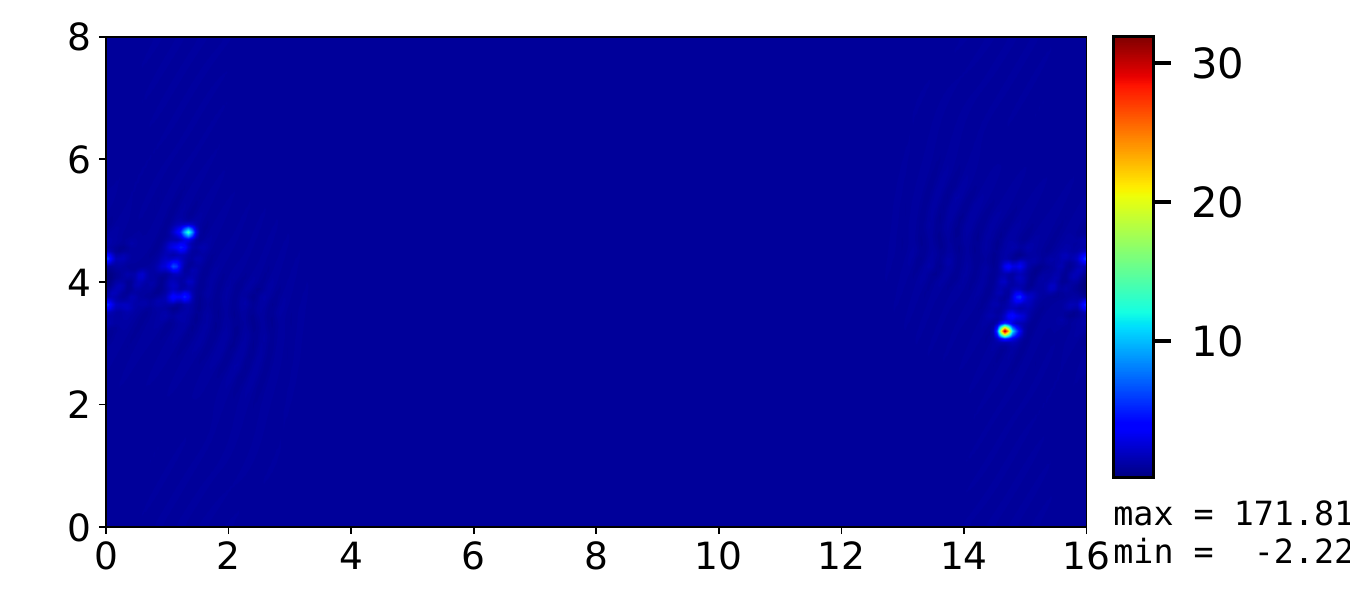}  \\
  \end{tabular}

  \caption{Failure of freestream preservation in the original CCSL method for the modified guiding-center model with a uniform initial condition $f=1$. Although the theoretical solution should remain constant under a divergence-free velocity field, the CCSL solution progressively departs from $f=1$ and develops spurious structures, as shown at $t=20,\;30,\;35,$ and $38$. The drift grows into a visible instability, revealing the loss of volume preservation in the backtracked cells. The improved CCSL method maintains the uniform state exactly and is therefore not displayed here.}
  \label{Fig_ITG_CCSL_compare}
\end{figure}

Having verified that the improved CCSL restores freestream preservation for a uniform state, we next consider a nontrivial initial condition to examine the long-time performance of different schemes. Five methods are compared: BSL, splitting BSL, CCSL, splitting CSL, and improved CCSL. Although the original CCSL and splitting CSL schemes eventually become unstable due to the lack of freestream preservation, their results before breakdown are still included for comparison. Under the modified Poisson equation, the guiding-center model remains mass- and norm-conservative, while the total energy is no longer strictly conserved because of the external source term. Since all values of $f$ are positive, we only compare the $L^1$- and $L^2$-norms conservation in the following analysis.

The computational domain is $[0,16] \times [0,8]$,
discretized by a uniform grid of $256\times128$ cells.
The initial condition is defined as
\begin{equation}
\label{Eq_ITG_init}
    \begin{cases}
        f(x,y,0) = \dfrac{1}{\sqrt{2\pi\,m_i\,T(x)}},\\[3pt]
        T(x) = 1 - \dfrac{L_x}{74\pi}\cos\!\left(\dfrac{2\pi x}{L_x}\right),
    \end{cases}
\end{equation}
which develops vortex structures during the evolution. 
The time step is $\Delta t = 1$ ($\text{CFL}\approx10.5$), and periodic boundary conditions are applied in both directions.

To illustrate the long-time behavior of the two schemes, Fig.~\ref{Fig_distribution_ITG} compares the spatial distributions obtained by the original CCSL and the improved CCSL at representative times 
$t=10,\,50,\,80,\,90,$ and $110$.  The two methods remain in close agreement at early times, but the CCSL solution gradually develops noticeable nonphysical oscillations as time progresses.  Although the CCSL still produces a seemingly reasonable distribution at $t=110$, the accumulated geometric inconsistency eventually leads to a breakdown shortly afterward (around $t\approx118$).  In contrast, the improved CCSL maintains a smooth and physically consistent solution throughout the entire simulation, benefiting from its freestream-preserving correction, which prevents the long-time degradation observed in CCSL.
\begin{figure}[htbp]
  \centering
  \setlength{\tabcolsep}{3pt} 
  \renewcommand{\arraystretch}{0}

  \begin{tabular}{cc}
    \includegraphics[width=0.35\textwidth]{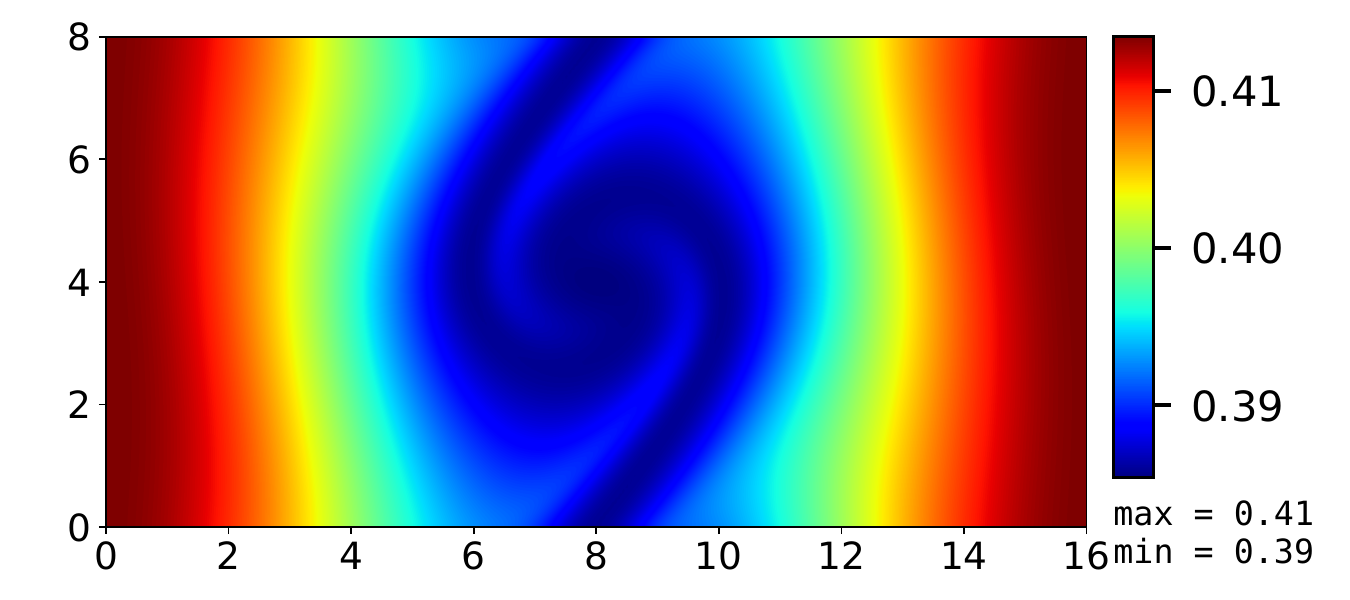} &
    \includegraphics[width=0.35\textwidth]{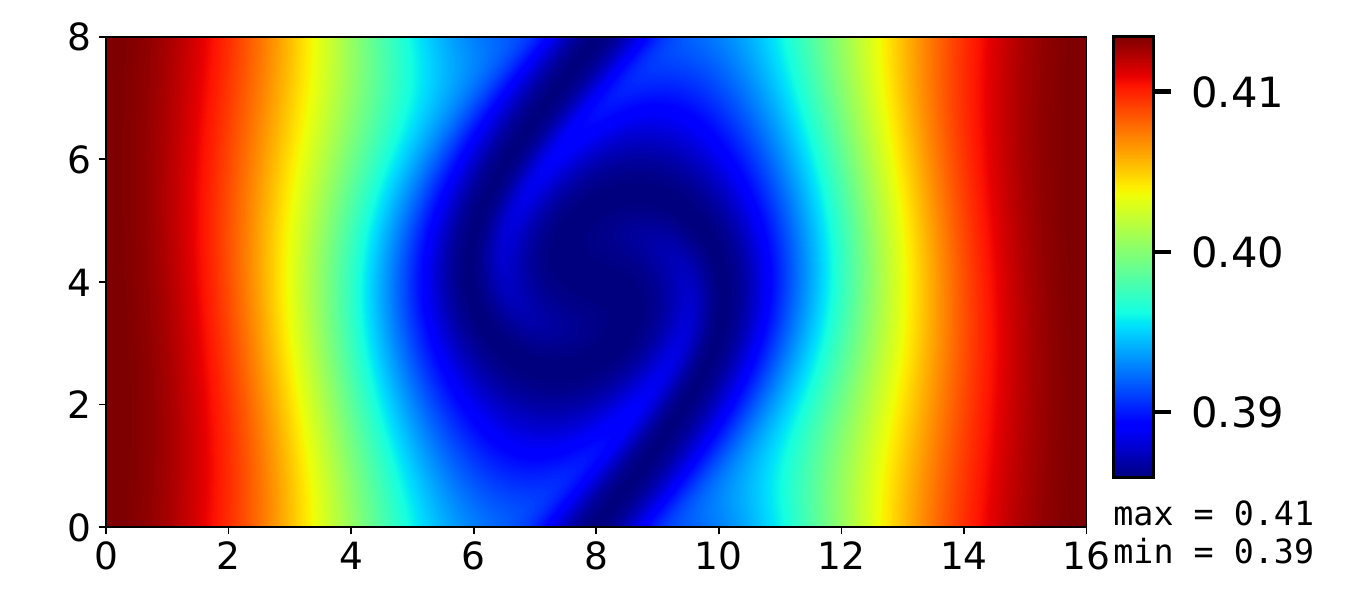} \\
    \includegraphics[width=0.35\textwidth]{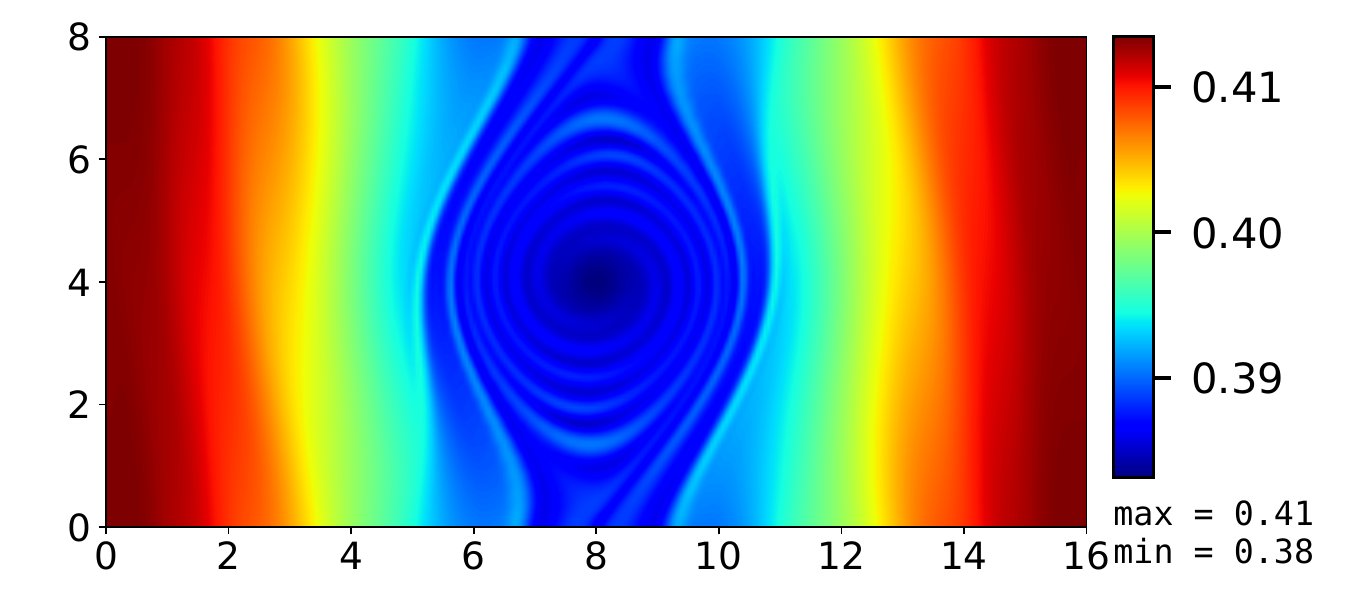} &
    \includegraphics[width=0.35\textwidth]{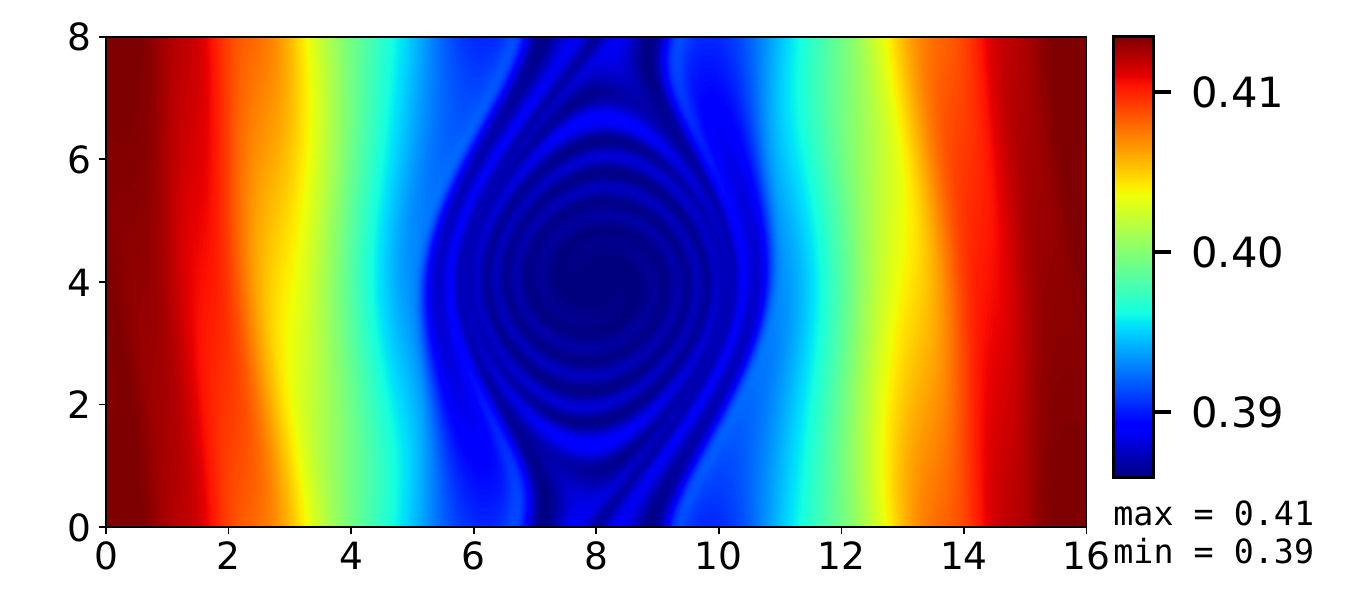} \\
    \includegraphics[width=0.35\textwidth]{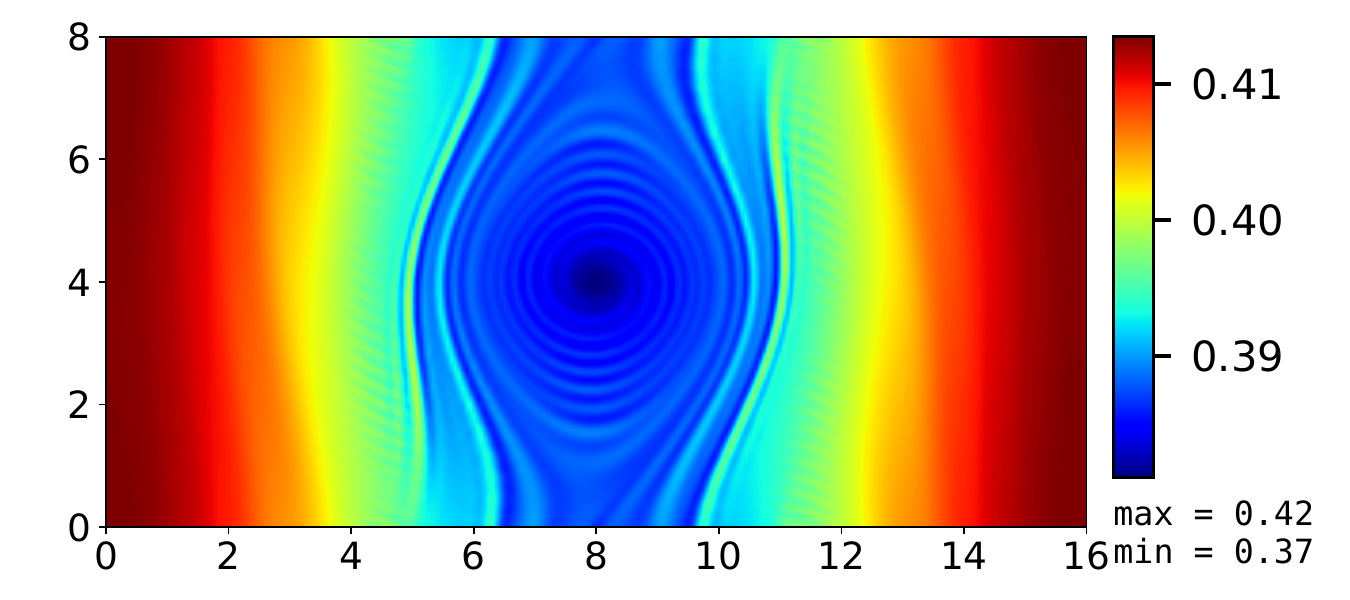} &
    \includegraphics[width=0.35\textwidth]{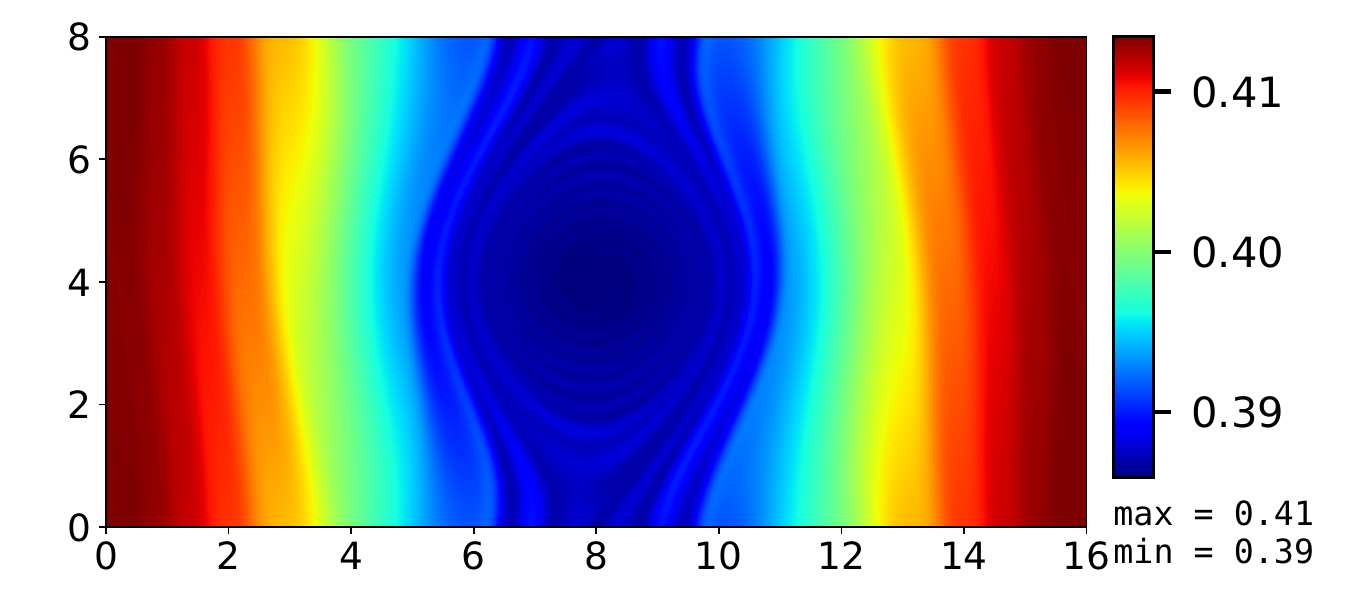} \\
    \includegraphics[width=0.35\textwidth]{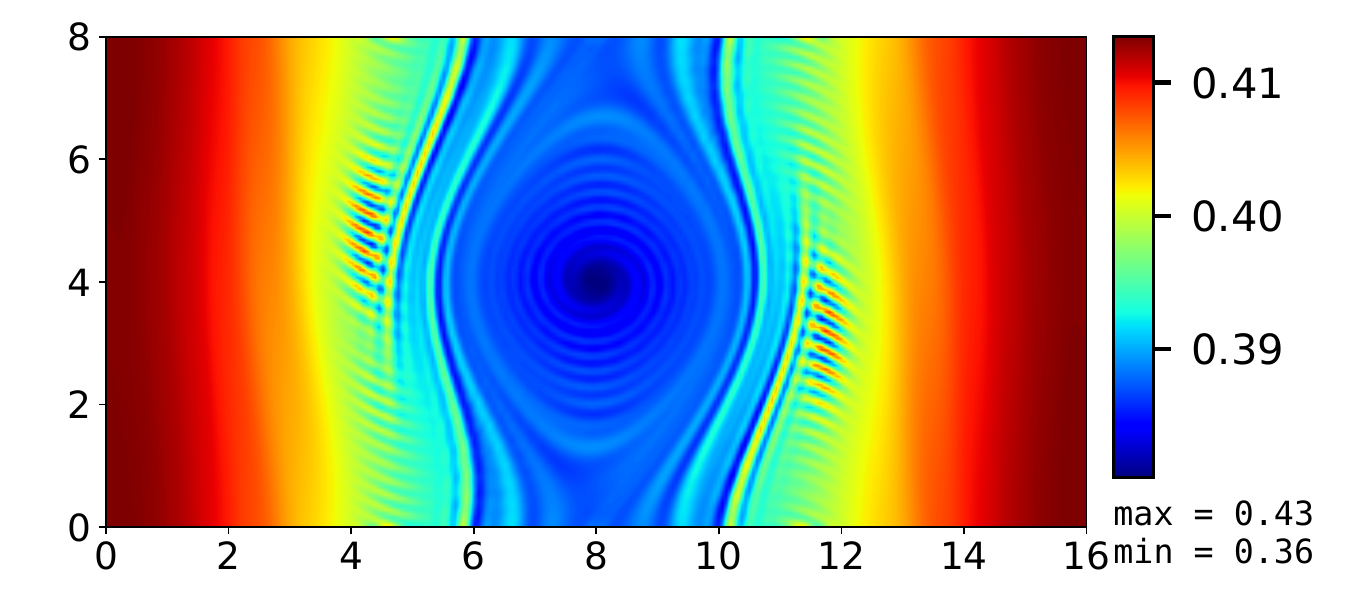} &
    \includegraphics[width=0.35\textwidth]{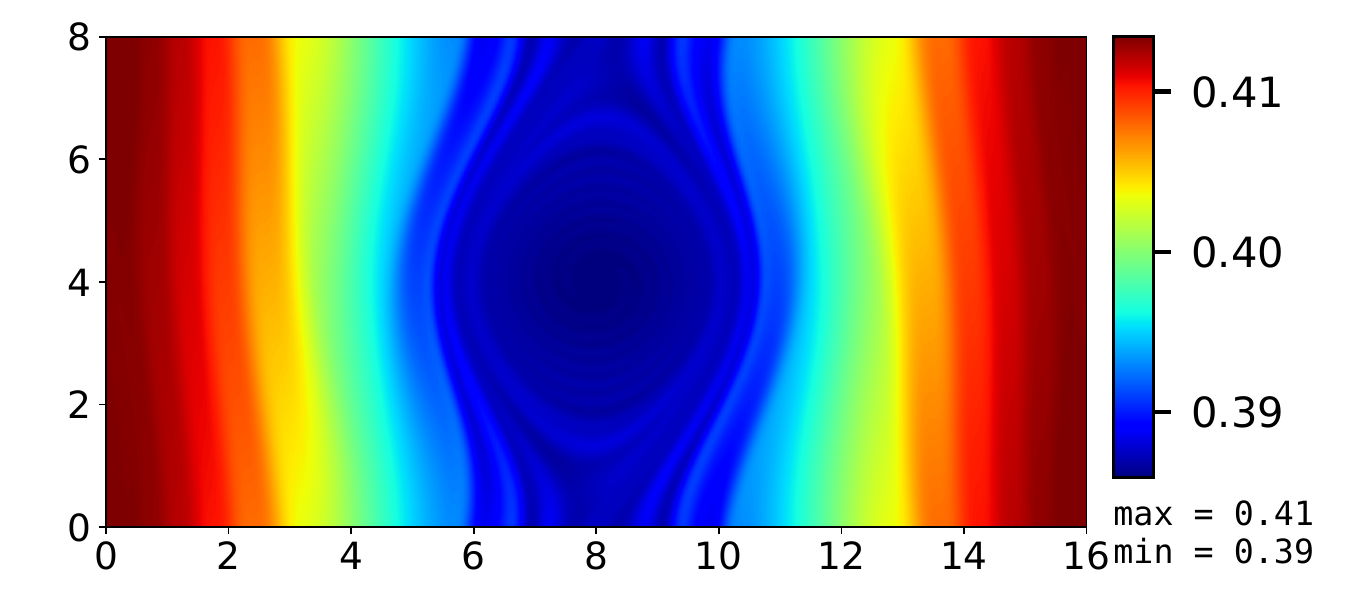} \\
    \includegraphics[width=0.35\textwidth]{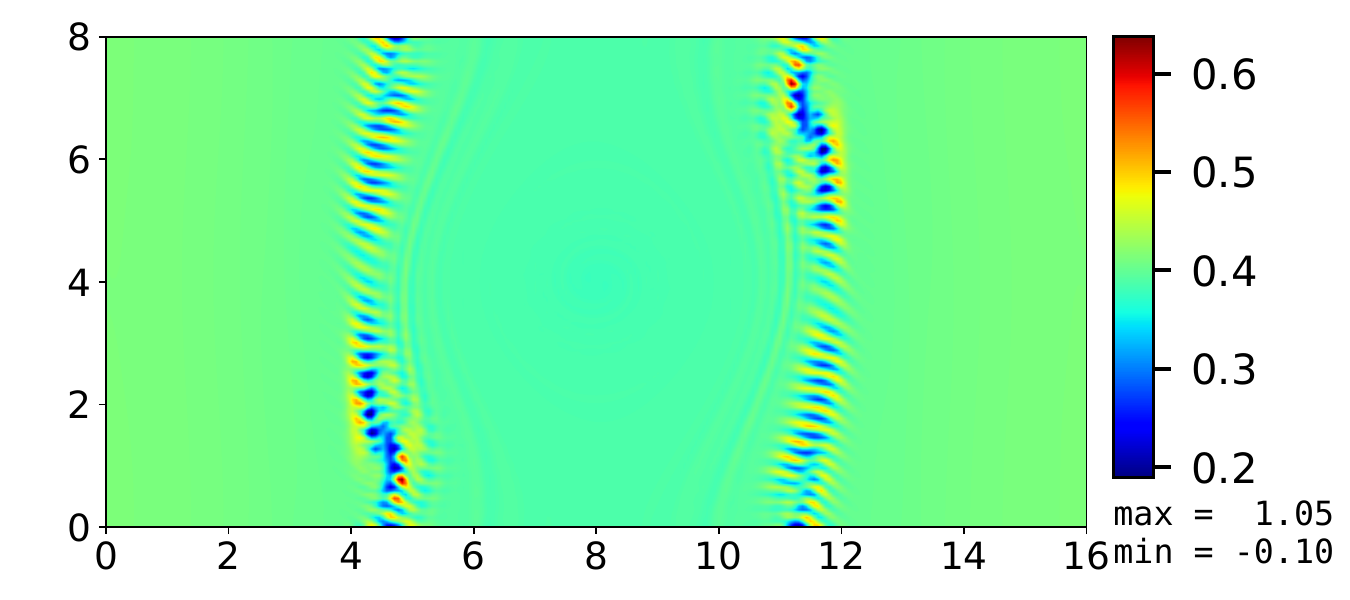} &
    \includegraphics[width=0.35\textwidth]{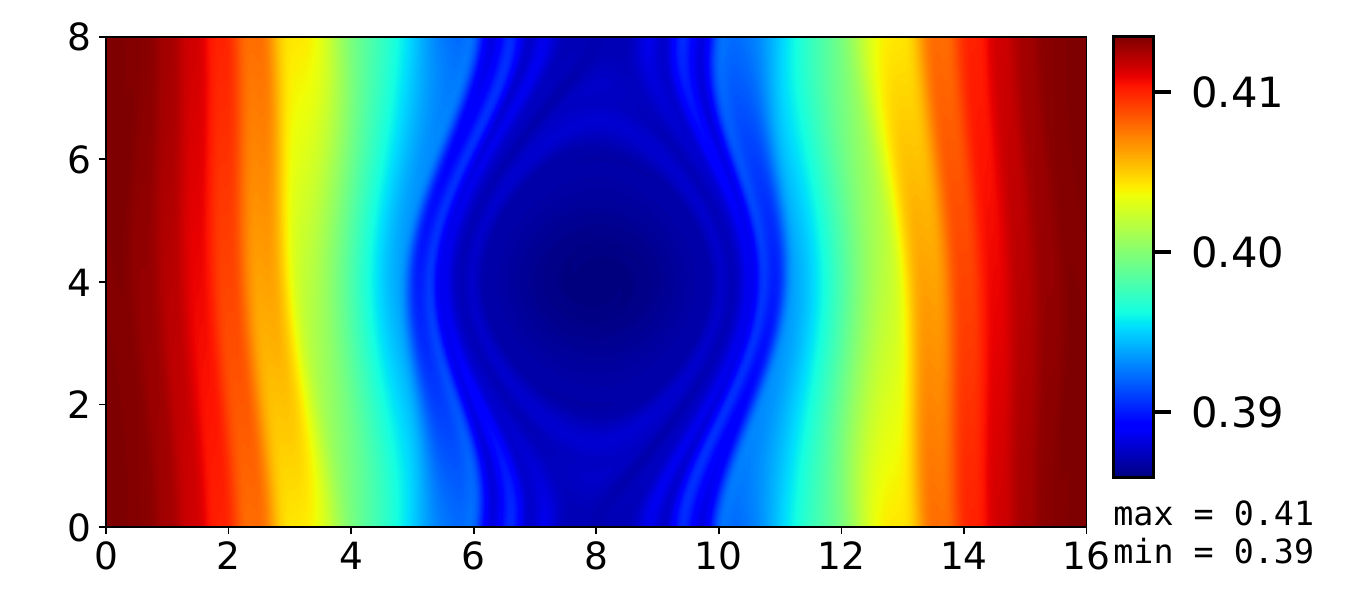} \\
  \end{tabular}

  \caption{Comparison of the spatial distributions of $f$ obtained by the CCSL (left column) and the improved CCSL (right column) for the modified Poisson equation with the initial  condition given in Eq. \eqref{Eq_ITG_init}, at times $t=10,\,50,\,80,\,90,$ and $110$ with $N_x = 256$ and $N_y = 128$.  The CCSL solution develops noticeable nonphysical oscillations as time evolves and breaks down shortly after $t=110$ (around $t\approx118$), whereas the improved CCSL remains stable and smooth throughout the simulation due to its freestream-preserving correction.}
   \label{Fig_distribution_ITG}
\end{figure}

Fig. \ref{Fig_curves_ITC} shows the time evolution of the $L^1$- and $L^2$-norm errors for the three schemes. The improved CCSL exhibits the smallest deviation in both $L^1$- and $L^2$-norms and maintains a remarkably stable error evolution, consistent with its strict freestream-preserving property and effective suppression of numerical oscillations. In comparison, the BSL and splitting BSL schemes show visibly larger fluctuations. For reference, the results of the original CCSL and splitting CSL schemes are also reported in Fig.~\ref{Fig_curves_ITC_ns}; both methods eventually lose stability and diverge because they do not preserve the phase-space volume.

\begin{figure}[htbp]
    \centering
    \begin{subfigure}[b]{0.4\textwidth}
        \centering
        \includegraphics[width=\linewidth]{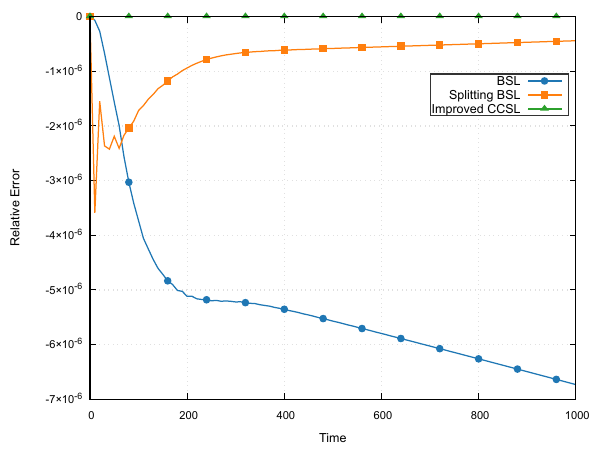}
        \caption{$L^1$-norm.}
        \label{Fig_l1norm_error_ITG}
    \end{subfigure}
    \vspace{0.5em}
    \begin{subfigure}[b]{0.4\textwidth}
        \centering
        \includegraphics[width=\linewidth]{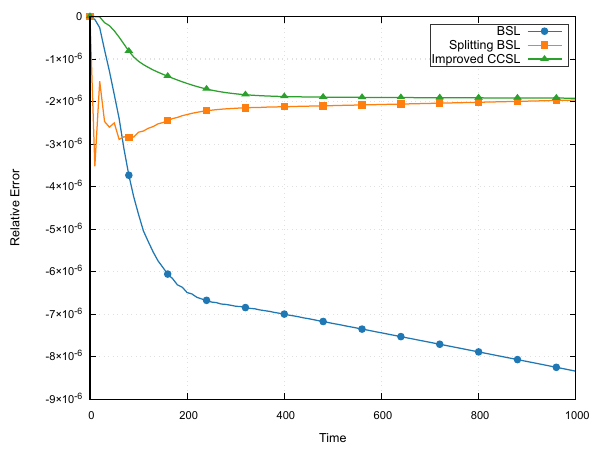}
        \caption{$L^2$-norm.}
        \label{Fig_l2norm_error_ITG}
    \end{subfigure}
    \caption{Time evolution of the $L^1$- and $L^2$-norms conservation obtained by three different methods—splitting BSL, BSL, and improved CCSL—for the guiding-center model with the modified Poisson equation and the initial condition defined in Eq.~\eqref{Eq_ITG_init}.}
    \label{Fig_curves_ITC}
\end{figure}  

\begin{figure}[htbp]
    \centering
    \begin{subfigure}[b]{0.4\textwidth}
        \centering
        \includegraphics[width=\linewidth]{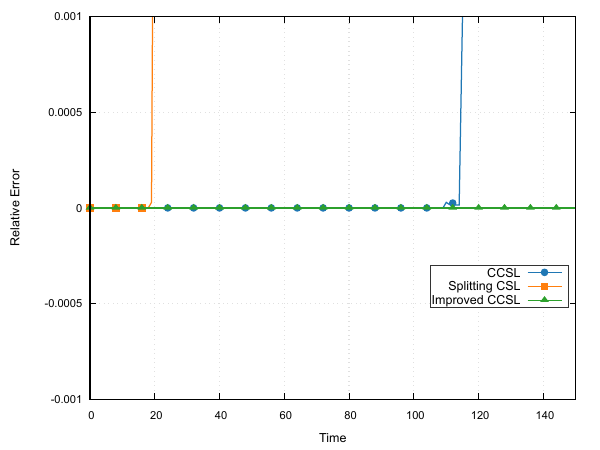}
        \caption{$L^1$-norm.}
        \label{Fig_l1norm_error_ITG1}
    \end{subfigure}
    \vspace{0.5em}
    \begin{subfigure}[b]{0.4\textwidth}
        \centering
        \includegraphics[width=\linewidth]{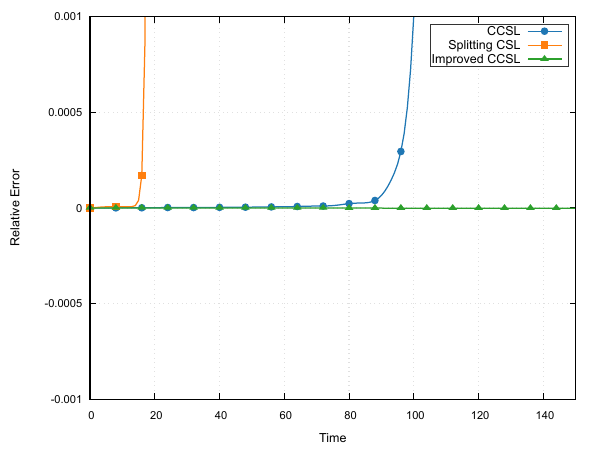}
        \caption{$L^2$-norm.}
        \label{Fig_l2norm_error_ITG1}
    \end{subfigure}
    \caption{Time evolution of the $L^1$- and $L^2$-norms conservation obtained with the splitting CSL, CCSL, and improved CCSL schemes for the guiding-center model with the modified Poisson equation and the initial condition given in Eq.~\eqref{Eq_ITG_init}. The CCSL and splitting CSL schemes lose stability and diverge at later times due to the lack of phase-space freestream preservation.}
    \label{Fig_curves_ITC_ns}
\end{figure}

\subsection{1D relativistic Vlasov–Maxwell model}

We next consider the 1D relativistic Vlasov–Maxwell model \cite{crouseilles2007}, which describes the interaction between a population of relativistic electrons and self-consistent electromagnetic fields. This model can be regarded as the relativistic extension of the classical 1D 
Vlasov–Maxwell system and is commonly used to study laser–plasma interactions in the relativistic regime. 

The governing equations of the 1D relativistic Vlasov--Maxwell system are written as
\begin{equation}
\label{Eq_RVM}
\frac{\partial f}{\partial t}+\frac{p}{m\gamma}\frac{\partial f}{\partial x}
+\left(eE_x - \frac{m c^2}{2\gamma}\frac{\partial a^2}{\partial x}\right)
\frac{\partial f}{\partial p} = 0,
\end{equation}
where $\gamma=\sqrt{1+p^2/(m^2 c^2) + a^2}$ is the relativistic factor, and $a(x,t)=eA(x,t)/mc$ is the
normalized amplitude of the potential vector $A=(0,A_y,A_z)$. The electromagnetic fields evolve according to
\begin{equation}
\label{Eq_Maxwell}
\begin{cases}
    \dfrac{\partial E_y}{\partial t}=-c^2\dfrac{\partial B_z}{\partial t}+\omega_p^2A_y\rho_{\gamma},\;\dfrac{\partial E_z}{\partial t}=c^2\dfrac{\partial B_y}{\partial t}+\omega_p^2A_z\rho_{\gamma},\\[6pt]
    \dfrac{\partial B_y}{\partial t}=\dfrac{\partial E_z}{\partial x},\dfrac{\partial B_z}{\partial t}=-\dfrac{\partial E_y}{\partial x},
\end{cases}
\end{equation}
where $\omega_p$ is the usual plasma frequency and $\rho_{\gamma}$ is defined as:
\begin{equation*}
    \rho_{\gamma}(x,t)=\int_{\mathbb{R}}\dfrac{f(x,p,t)}{m\gamma}\,dp.
\end{equation*}
The components of the vector potential are then obtained from
\begin{equation*}
    \dfrac{\partial A_y}{\partial t}=- E_y,\dfrac{\partial A_z}{\partial t}=-E_z.
\end{equation*}
In the 1D relativistic Vlasov-Maxwell system, the longitudinal electric field $E_x$ can be computed in two equivalent ways. First, $E_x$ can be obtained from the Poisson equation,
\begin{equation*}
\frac{\partial E_x}{\partial x} = \dfrac{e}{\epsilon_0}\left(\int_{\mathbb{R}}f(x,p,t)\,dp-n_i\right).
\end{equation*}
This formulation ensures that Gauss’s law is exactly satisfied at each time step. Alternatively, $E_x$ can be computed from the Ampère equation,
\begin{equation*}
\frac{\partial E_x}{\partial t} = -J,
\end{equation*}
where $J(x,t) = \int_{\mathbb{R}} (pf/m\gamma) f\,dp$ denotes the longitudinal current density.
This time-dependent formulation is particularly useful for schemes that directly evolve the electromagnetic fields.

The computational domain is $x \in [0,2\sqrt{2}\pi]$ and $v \in [-2.5,\, 2.5]$, discretized by a uniform grid of $N_x \times N_v = 256 \times 256$. The time step is set to $\Delta t = 0.01$. Periodic boundary conditions are imposed in the spatial direction, while zero boundary conditions are applied in velocity space.

The initial distribution function is a Maxwellian with a small sinusoidal perturbation in $x$:
\begin{equation*}
  f(x,v,0) \;=\; \frac{1}{\sqrt{2\pi}\,v_{\mathrm{th}}}\,
  \exp\!\Big(-\frac{v^{2}}{2v_{\mathrm{th}}^{2}}\Big)\,
  \Big[\,1+\varepsilon\cos\!\Big(\frac{2\pi}{L_x}\,(x+\Delta x)\Big)\Big],
\end{equation*}
where $v_{\mathrm{th}}=\sqrt{3/511}$ and $\varepsilon=0.1$ represent the thermal velocity 
and the density perturbation amplitude, respectively. 

The transverse components of the electric, magnetic, and vector potential fields are initialized as
\begin{equation*}
\begin{aligned}
  E_y(x,-\Delta t/2) &= Q\cos\!\big(k_0 x+\tfrac{\omega_0\,\Delta t}{2}\big), 
  &\quad E_z(x,-\Delta t/2) &= Q\sin\!\big(k_0 x+\tfrac{\omega_0\,\Delta t}{2}\big), \\[3pt]
  B_y(x,0) &= -\,\frac{k_s Q}{\omega_s}\sin(k_0x), 
  &\quad B_z(x,0) &= \frac{k_s Q}{\omega_s}\cos(k_0x), \\[3pt]
  A_y(x,0) &= -\,\frac{Q}{\omega_s}\sin(k_0 x), 
  &\quad A_z(x,0) &= \frac{Q}{\omega_s}\cos(k_0 x).
\end{aligned}
\end{equation*}
This initialization corresponds to a monochromatic electromagnetic wave of small amplitude propagating in a uniform plasma. The parameters are chosen as
\[
k_0=\frac{1}{\sqrt{2}}, \qquad
k_s=\frac{2}{\Delta x}\sin\!\Big(\frac{k_0\Delta x}{2}\Big), \qquad
P_{\mathrm{osc}}=\sqrt{3}, \qquad
\gamma_0=\sqrt{1+P_{\mathrm{osc}}^{2}},
\]
\[
\omega_s=\sqrt{\frac{1}{\gamma_0}+k_s^{2}}, \qquad
\omega_0=\frac{4}{\Delta t}\arcsin\!\Big(\frac{\Delta t\,\omega_s}{4}\Big), \qquad
Q=P_{\mathrm{osc}}\,\omega_s.
\]
These parameters satisfy the discrete dispersion relation and ensure consistency between the field and particle dynamics in the simulation.

The relativistic Vlasov–Maxwell model described above is used as a benchmark to assess the behavior of the BSL, CCSL, and improved CCSL schemes. As shown in Fig. \ref{Fig_f_RVM}, both the BSL and improved CCSL methods are able to capture the detailed deformation of the distribution function in phase space with high accuracy. At early times, the two solutions are almost indistinguishable; however, as the system evolves, the BSL result develops negative values near regions of strong gradients. In contrast, the improved CCSL, benefiting from the limiter, maintains a smooth and strictly positive distribution throughout the simulation, while preserving filament features and overall coherence of the phase-space structure.
\begin{figure}[htbp]
  \centering
  \setlength{\tabcolsep}{3pt}
  \renewcommand{\arraystretch}{0}

  \begin{tabular}{cc}
    \includegraphics[width=0.35\textwidth]{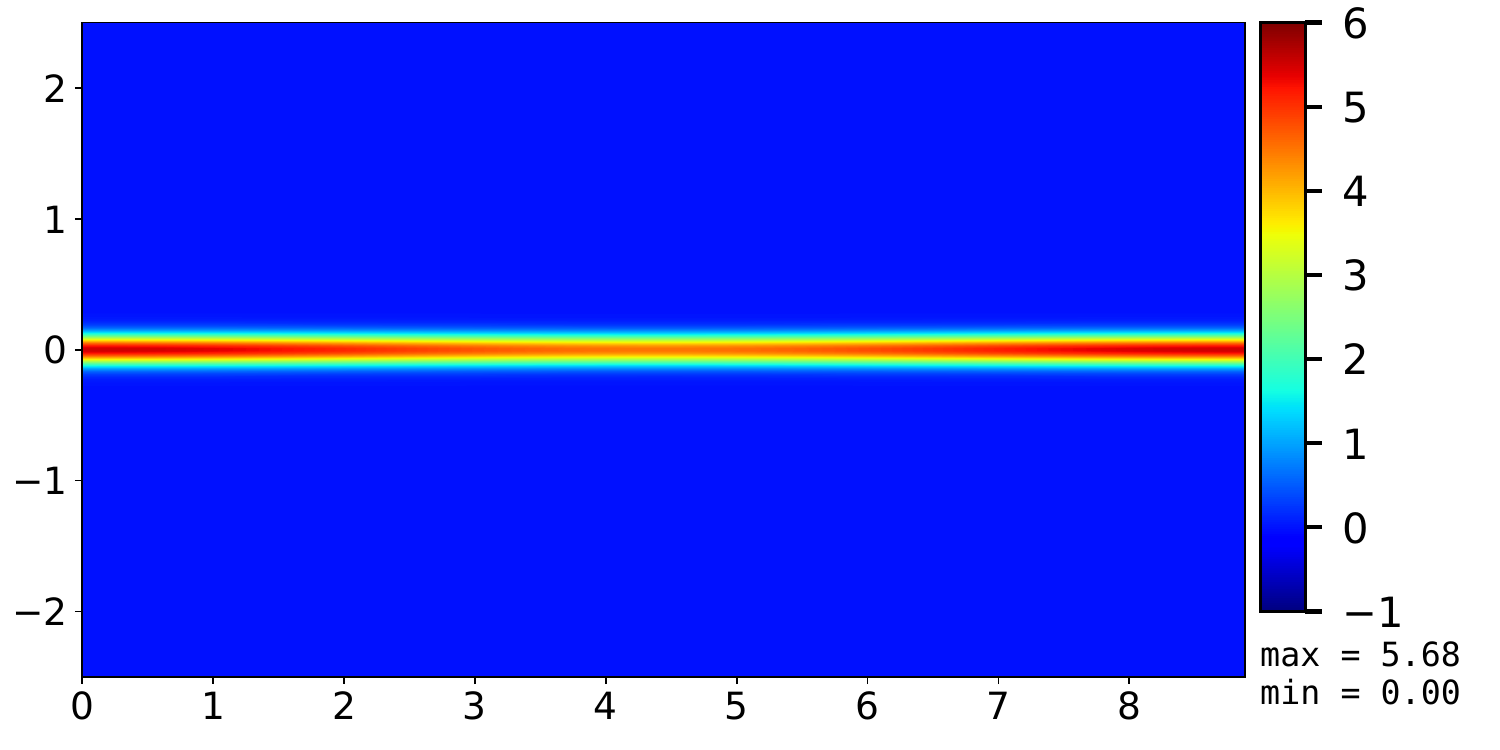} &
    \includegraphics[width=0.35\textwidth]{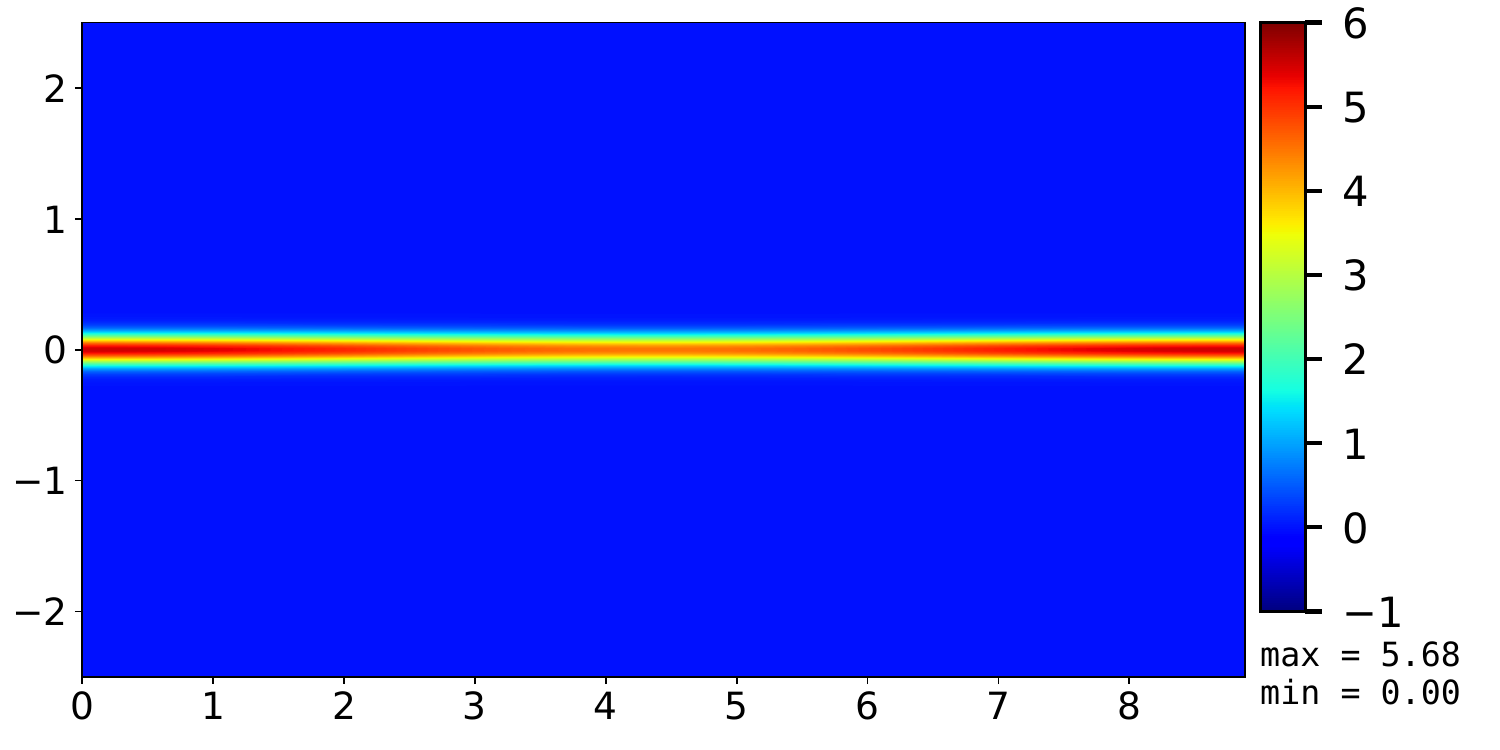} \\[-2pt]
    \includegraphics[width=0.35\textwidth]{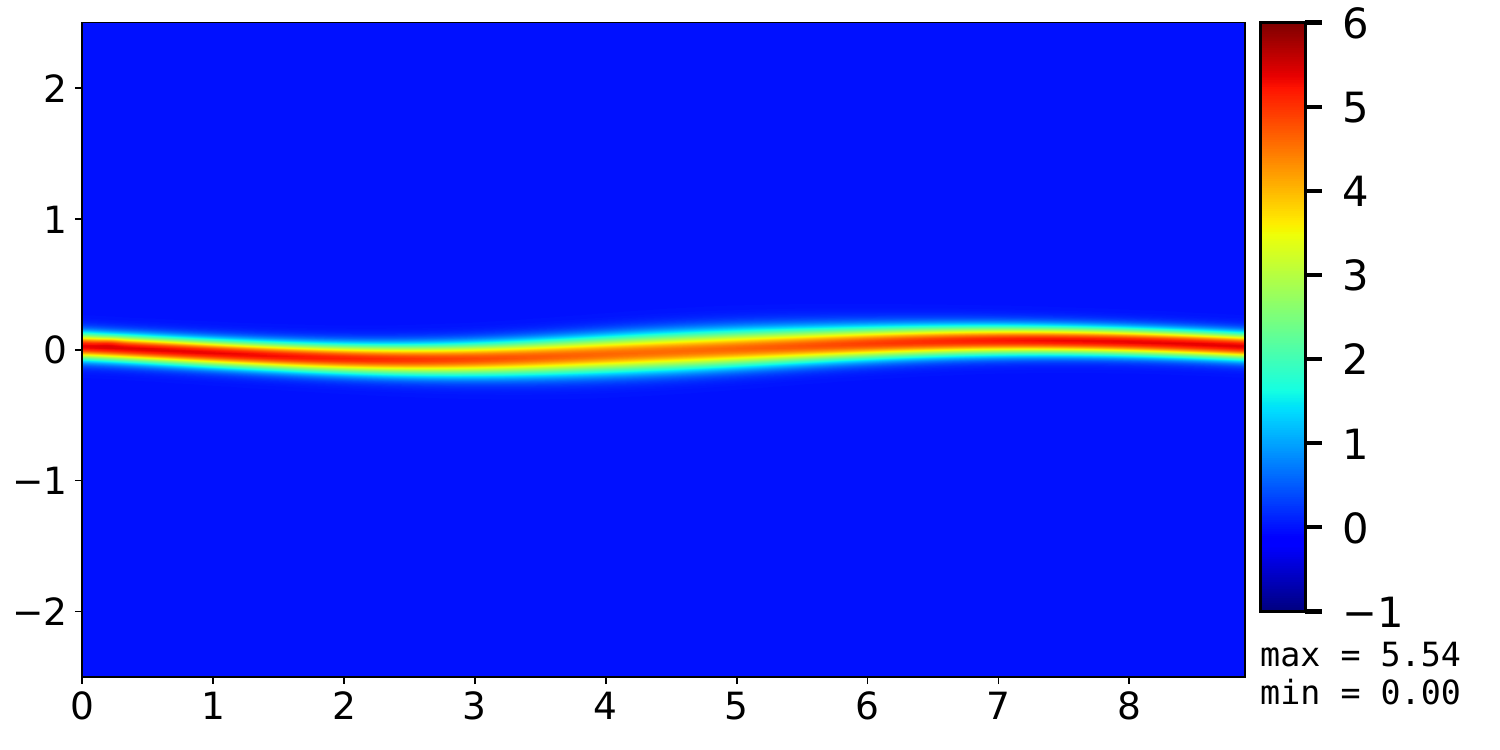} &
    \includegraphics[width=0.35\textwidth]{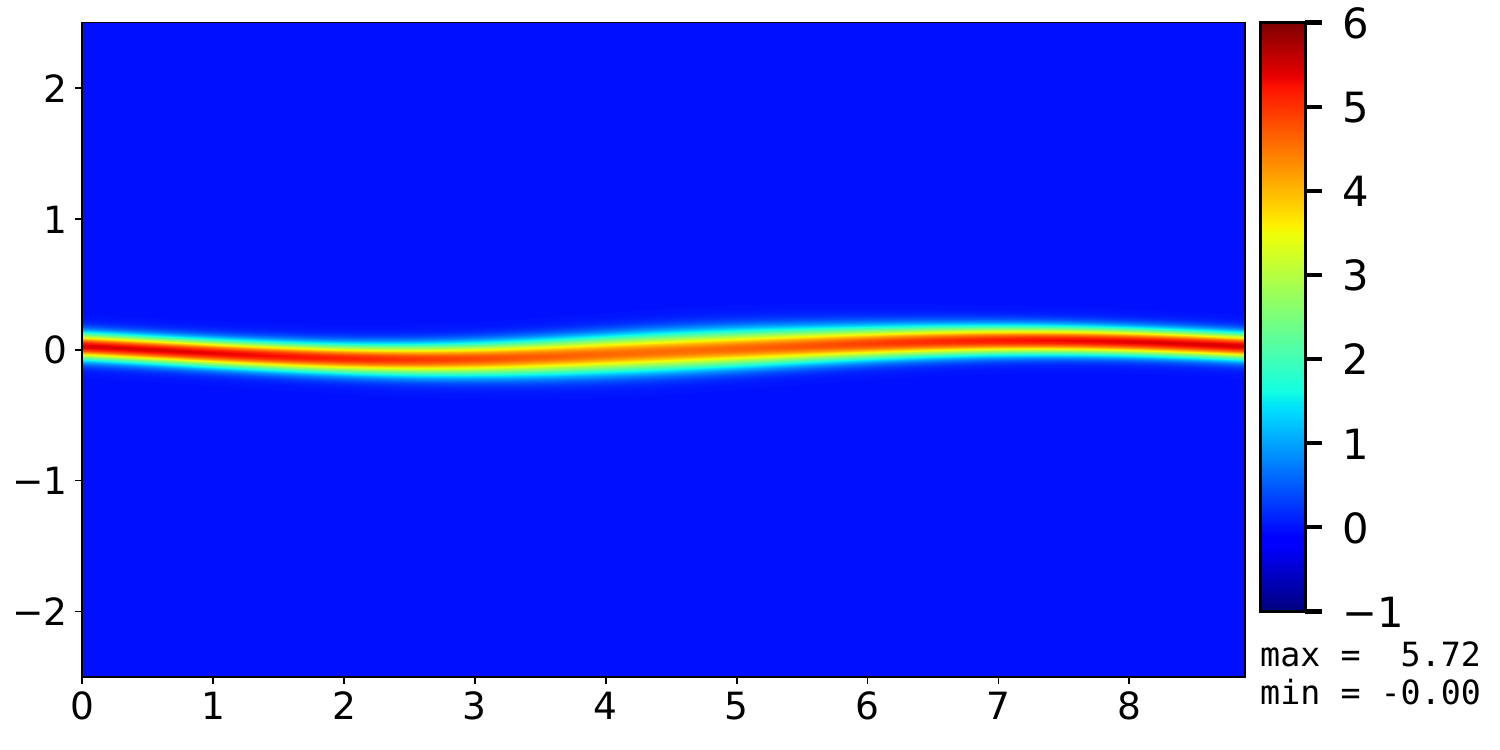} \\[-2pt]
    \includegraphics[width=0.35\textwidth]{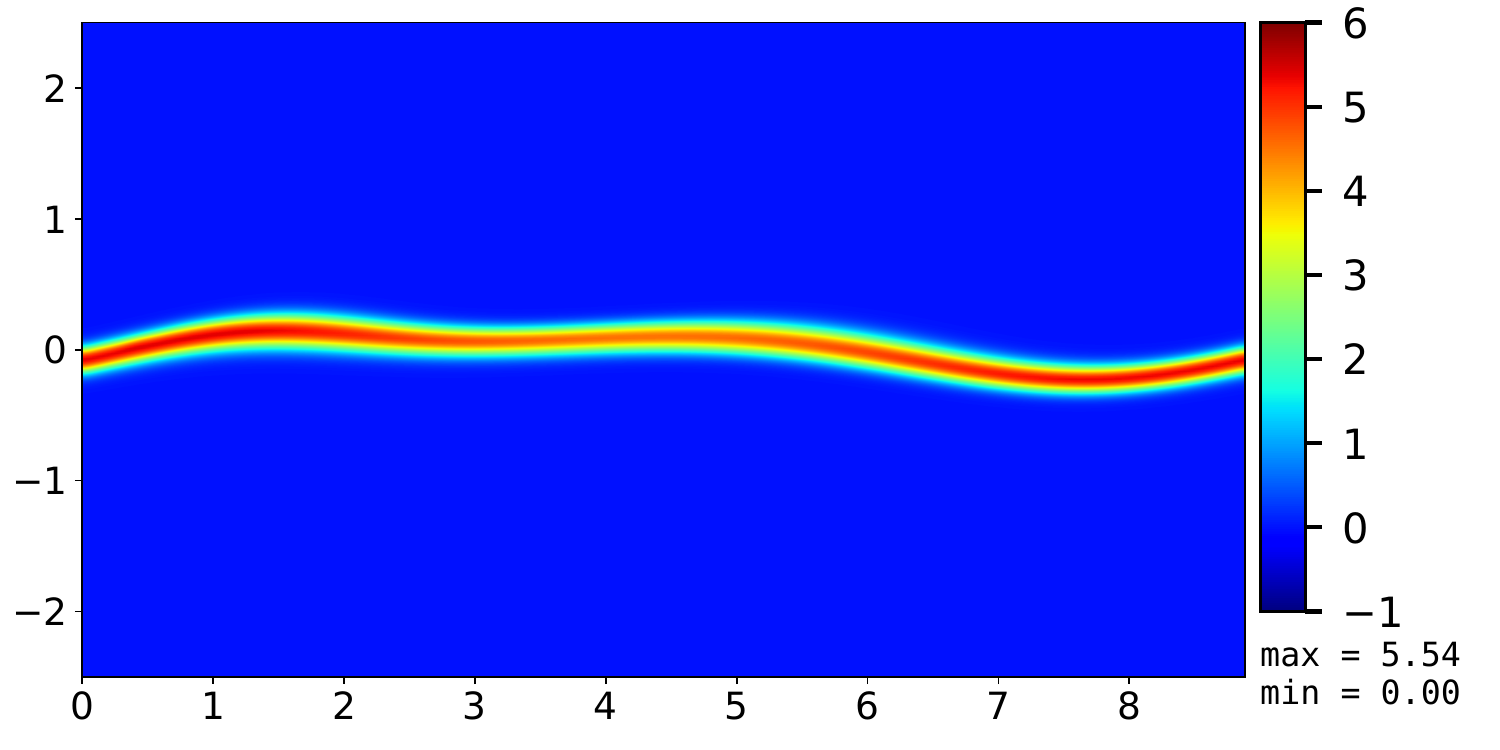} &
    \includegraphics[width=0.35\textwidth]{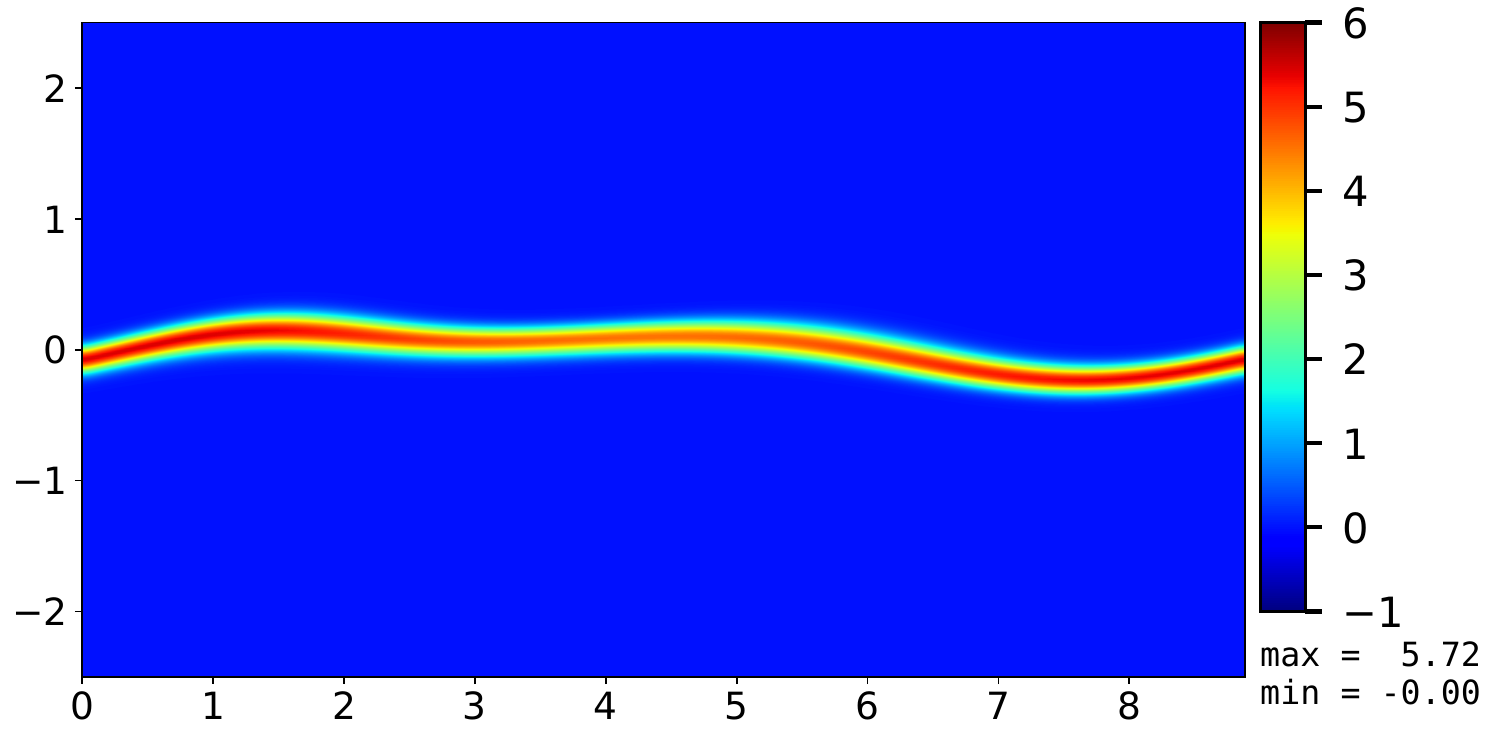} \\[-2pt]
    \includegraphics[width=0.35\textwidth]{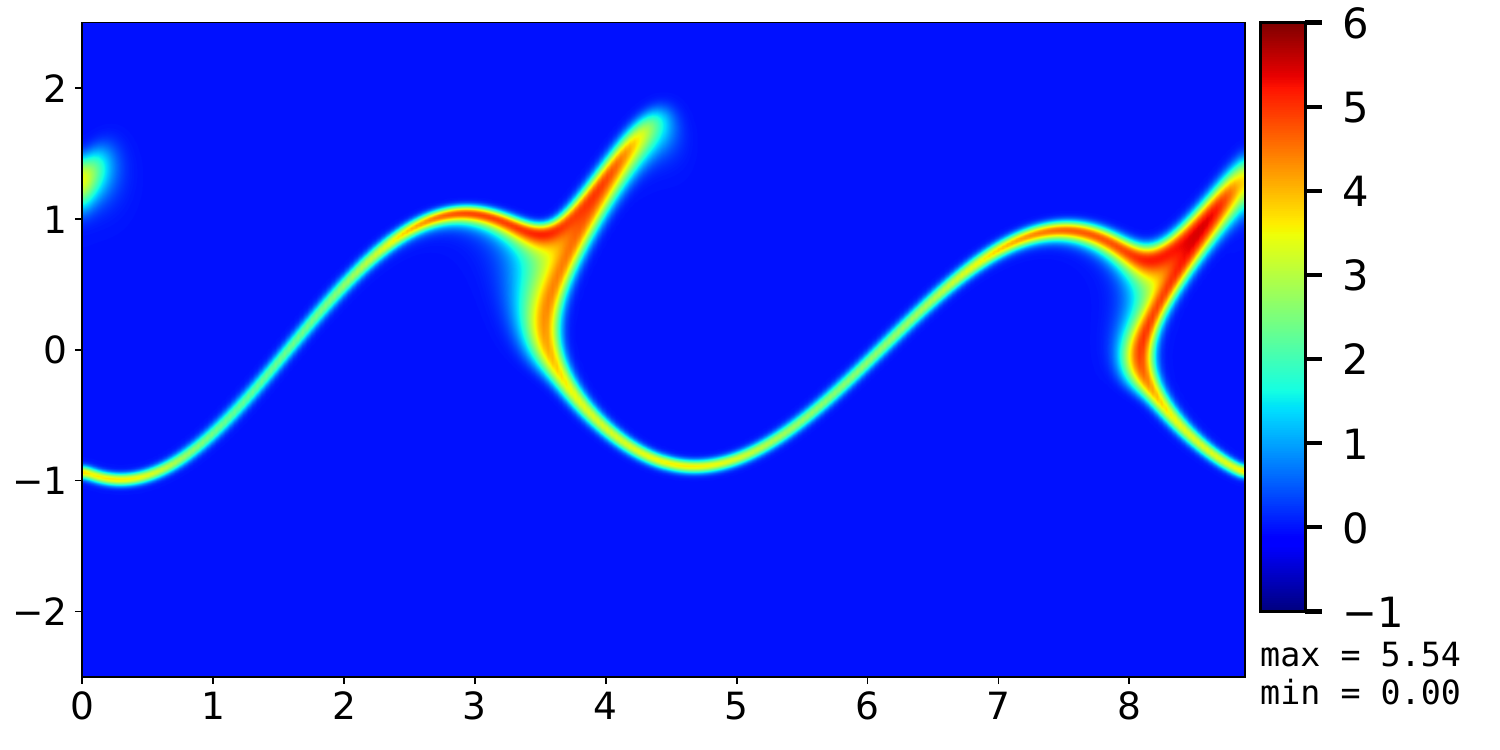} &
    \includegraphics[width=0.35\textwidth]{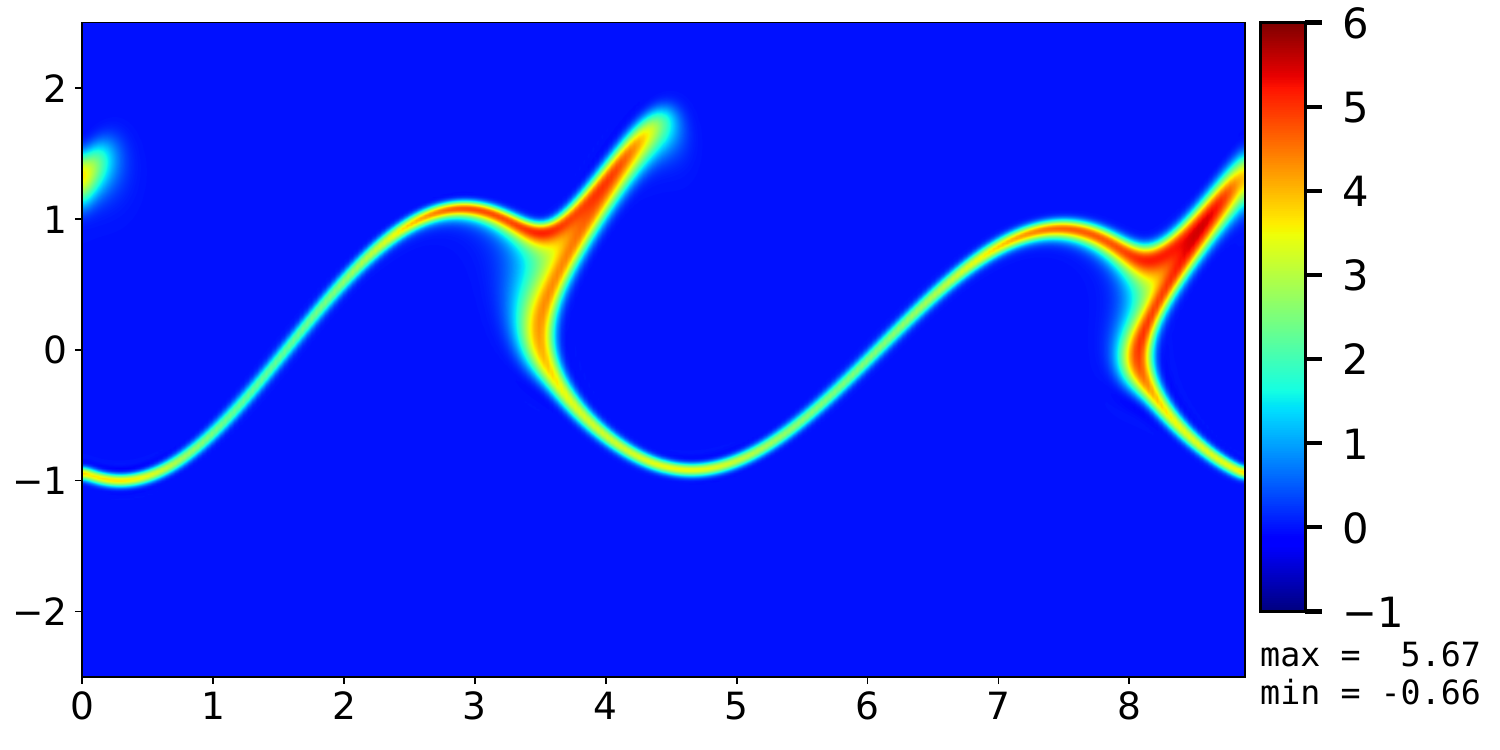} \\[-2pt]
    \includegraphics[width=0.35\textwidth]{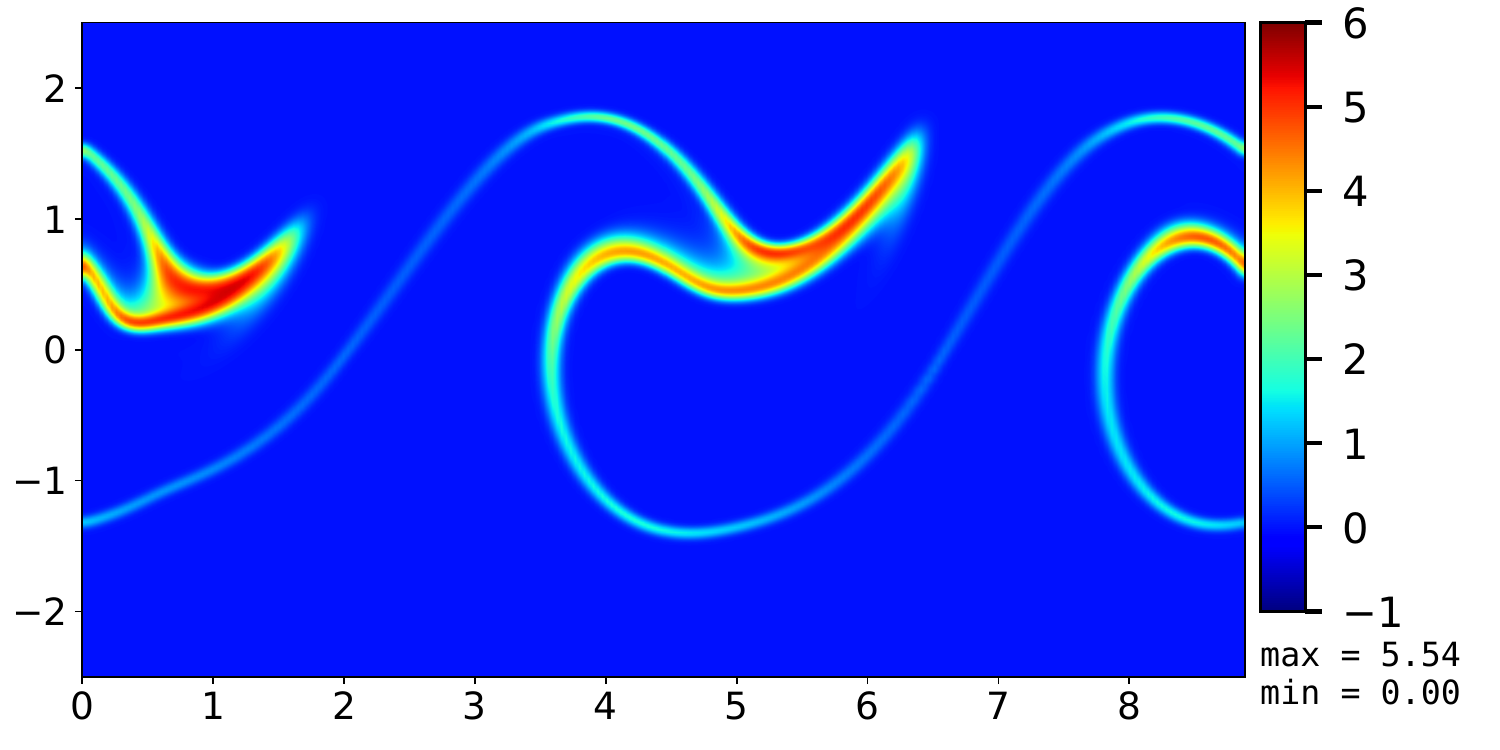} &
    \includegraphics[width=0.35\textwidth]{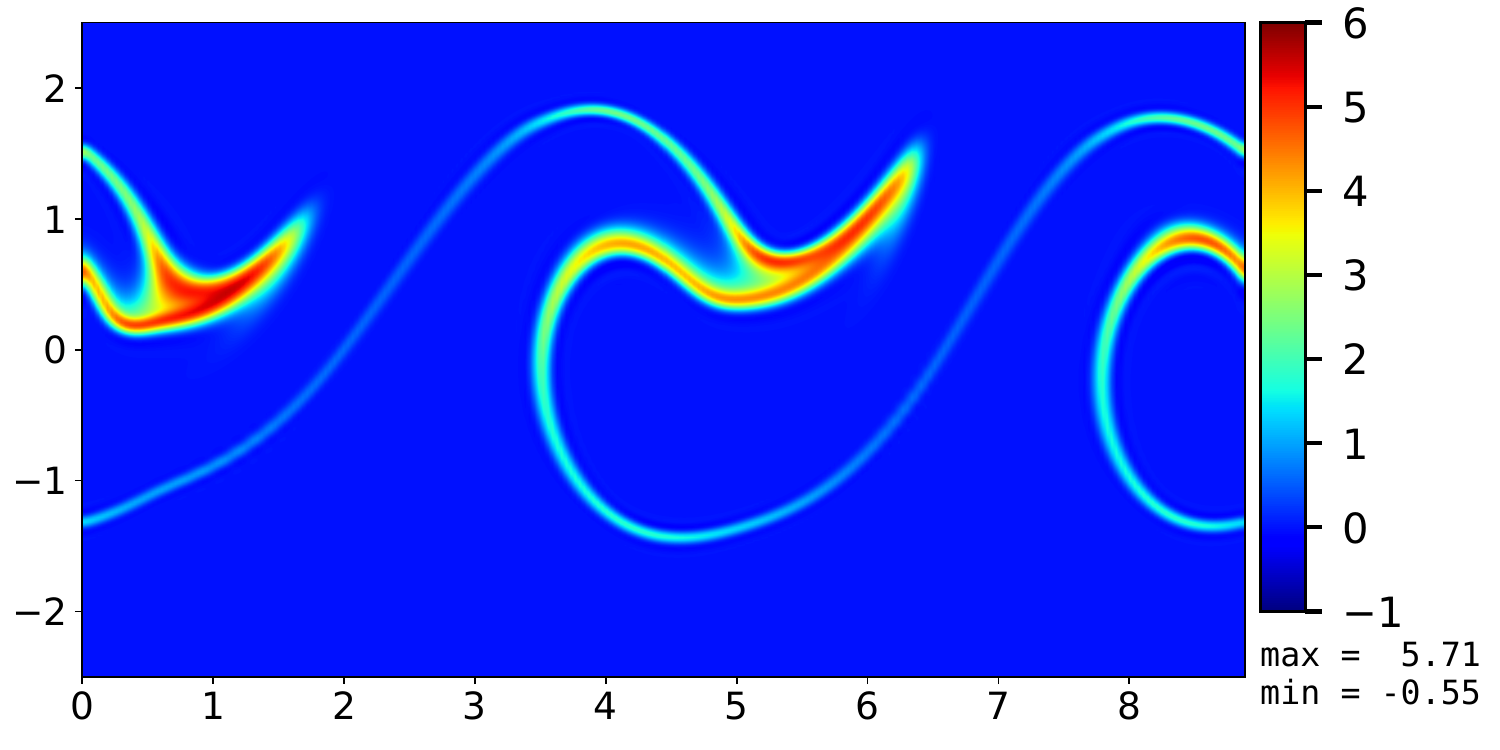} \\[-2pt]
    \includegraphics[width=0.35\textwidth]{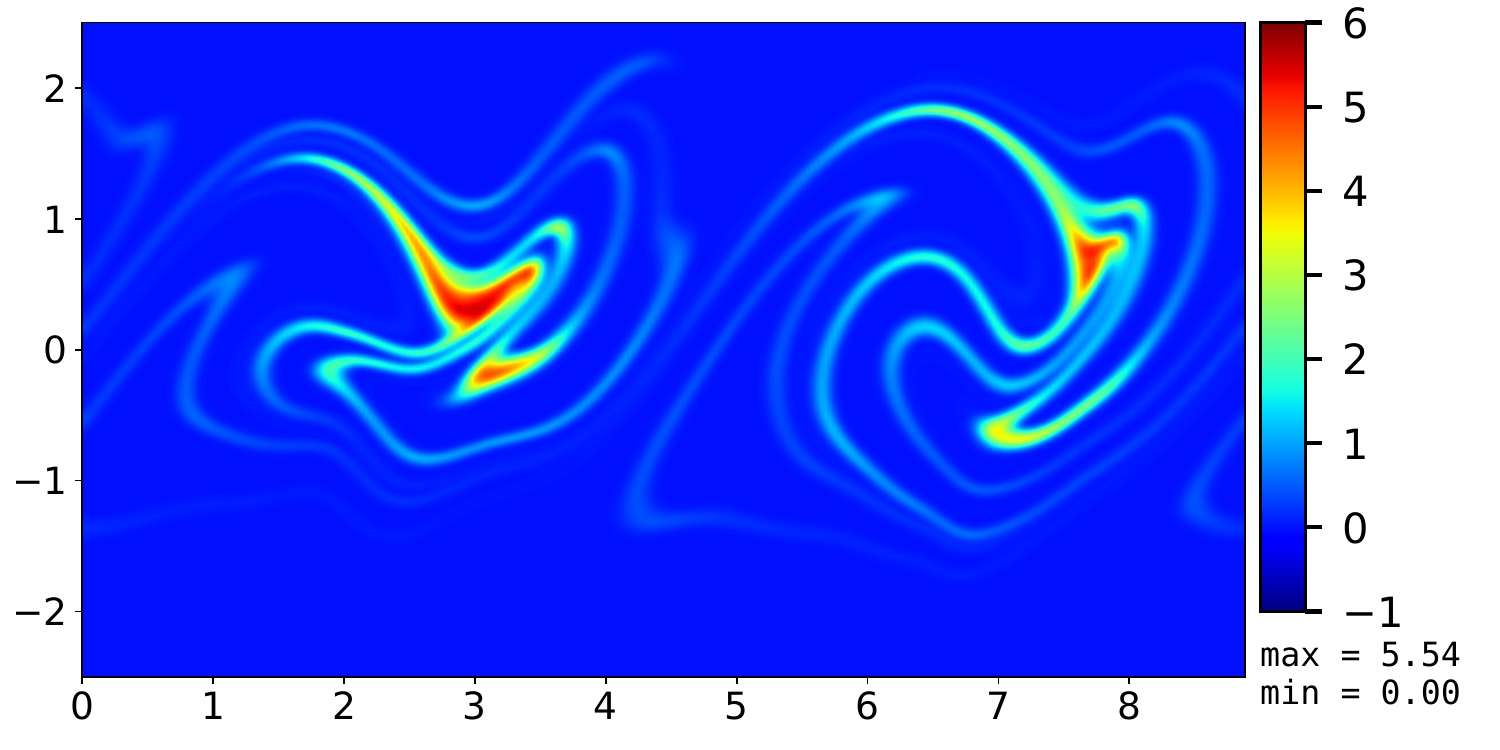} &
    \includegraphics[width=0.35\textwidth]{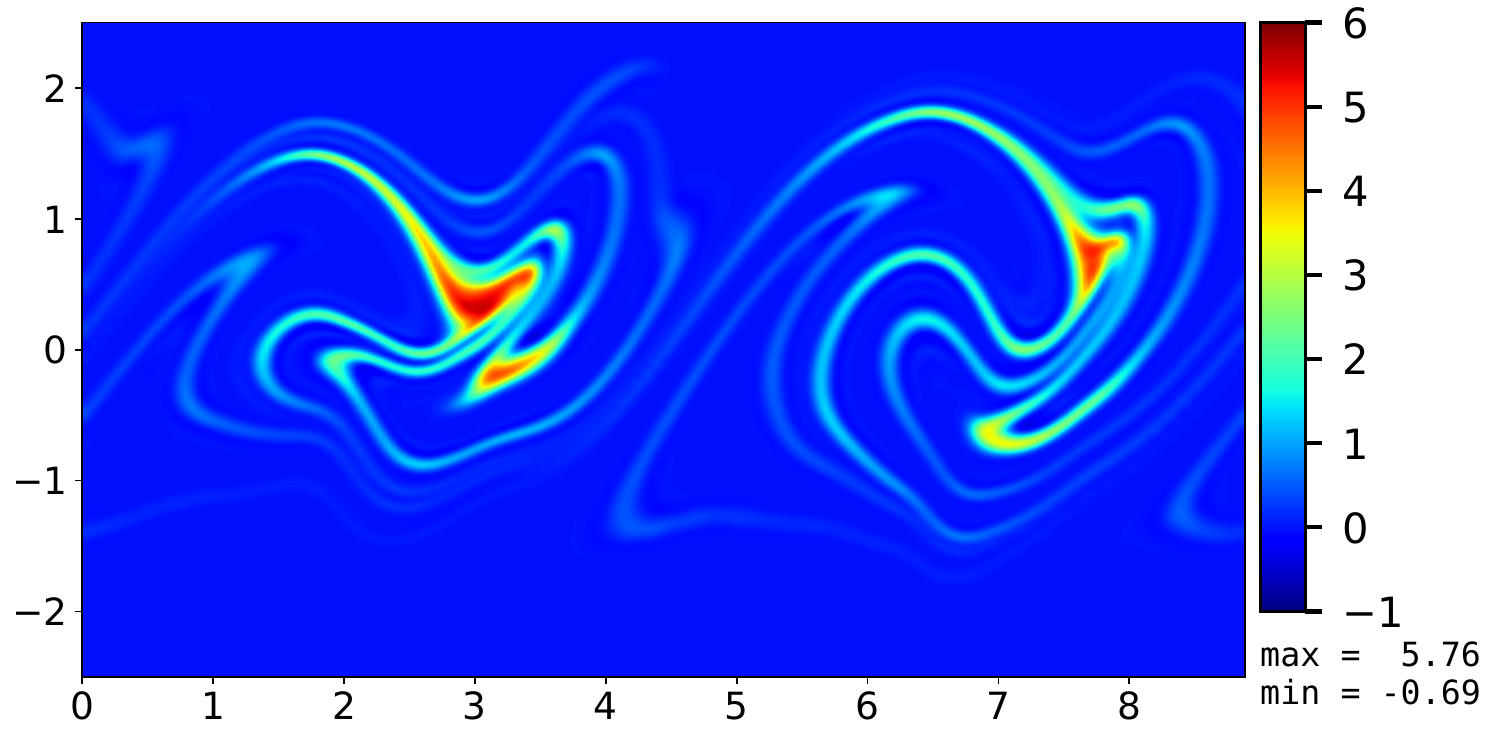} \\
  \end{tabular}

  \caption{
    Time evolution of the distribution function $f(x,v)$ obtained by the improved CCSL (left column) and the BSL (right column) for the 1D relativistic Vlasov--Maxwell model at $t = 0,\,5,\,10,\,17,\,20,$ and $30$ with $N_x = N_v = 256$ and $\Delta t = 0.01$.
    }
  \label{Fig_f_RVM}
\end{figure}
Fig. \ref{Fig_curves_RVM} shows the time evolution of the relative errors of mass, $L^1$- and $L^2$-norms, and total energy conservation. Both the CCSL and improved CCSL schemes strictly conserve mass during the simulation, whereas the BSL method exhibits a gradual deviation over time. The improved CCSL shows the smallest variation in the $L^1$-norm and remains almost constant in time, while the BSL and CCSL results display small oscillations. For the $L^2$-norm, the BSL performs slightly better, but the difference remains marginal throughout the simulation. The total energy evolution of the improved CCSL is nearly identical to that of the BSL method, indicating that both schemes achieve comparable accuracy in global energy conservation. Overall, the improved CCSL maintains excellent long-time stability and demonstrates the best overall conservation of mass and $L^1$-norm among the three schemes.

\begin{figure}[htbp]
    \centering
    \begin{subfigure}[b]{0.4\textwidth}
        \centering
        \includegraphics[width=\linewidth]{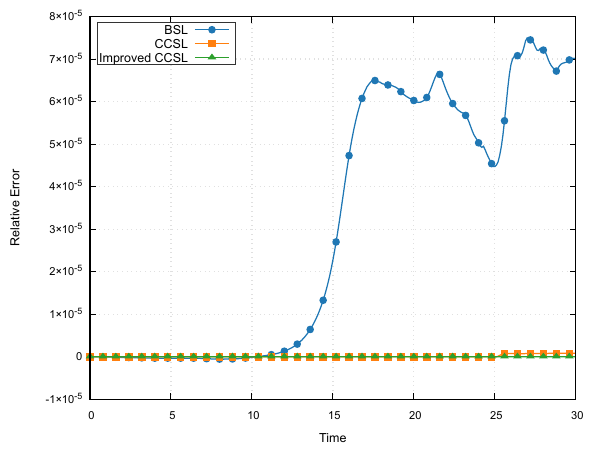}
        \caption{Mass.}
    \end{subfigure}
     \vspace{0.5em}
    \begin{subfigure}[b]{0.4\textwidth}
        \centering
        \includegraphics[width=\linewidth]{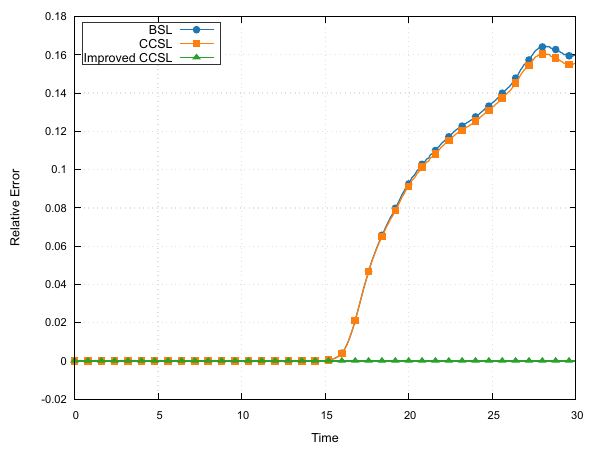}
        \caption{$L^1$-norm.}
    \end{subfigure}
    
    \begin{subfigure}[b]{0.4\textwidth}
        \centering
        \includegraphics[width=\linewidth]{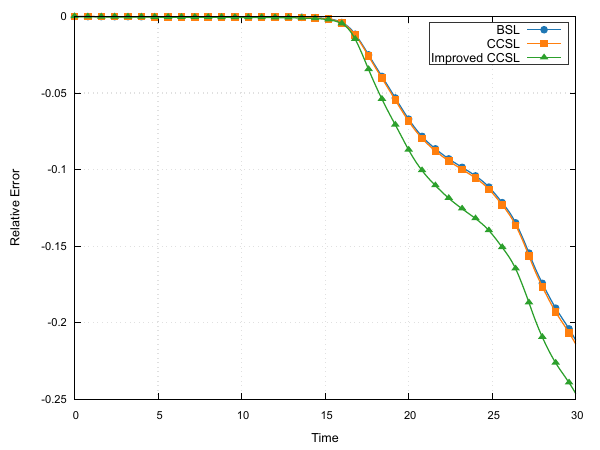}
        \caption{$L^2$-norm.}
    \end{subfigure}
     \vspace{0.5em}
    \begin{subfigure}[b]{0.4\textwidth}
        \centering
        \includegraphics[width=\linewidth]{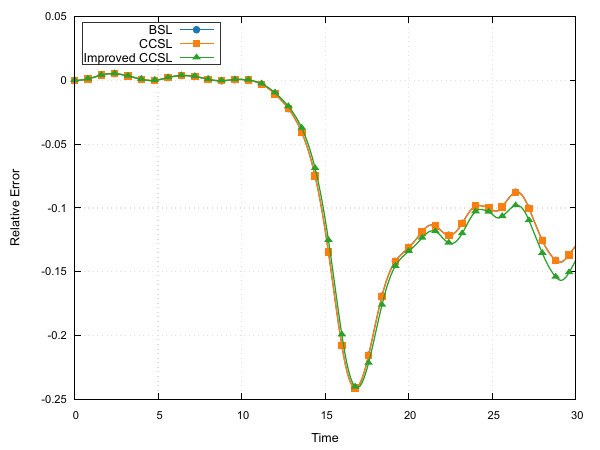}
        \caption{Total energy.}
    \end{subfigure}

    \caption{Time evolution of mass, $L^1$-norm, $L^2$-norm, and total energy conservation obtained by three different methods: CCSL, the improved CCSL, and BSL for the 1D relativistic Vlasov--Maxwell model.}
    \label{Fig_curves_RVM}
\end{figure}

\section{Conclusion}
\label{sec_conclusion}
In this work, we have employed the conservative cascade semi-Lagrangian (CCSL) method for solving the Vlasov equation. The proposed scheme combines the conservative framework of the classical CSL method with the Cascade remapping strategy, enabling accurate volume-to-volume transport in multiple dimensions. A consistency analysis demonstrates that the CCSL scheme attains second-order spatial accuracy, with the dominant error arising from the geometric approximation of the backtracked region. Building upon the original CCSL formulation, two main improvements were proposed. First, a freestream-preserving correction was applied to restore the exact volume of the backtracked cells and ensure physical consistency under divergence-free velocity fields. Second, a maximum-principle limiter was implemented to eliminate spurious oscillations while maintaining both mass conservation and positivity. Comprehensive numerical experiments, including linear advection, guiding-center dynamics, and the relativistic Vlasov–Maxwell system, confirm the theoretical results. 
The improved CCSL method exhibits excellent mass conservation, robustness, and long-time stability, and it outperforms traditional CSL and BSL schemes, particularly in problems where phase-space freestream preservation is essential. Future work will focus on extending the CCSL framework to higher-dimensional systems and implementing parallelization strategies to further improve its computational efficiency for large-scale plasma simulations.

\section*{Acknowledgments}

This work has been supported by the  Natural Science Foundation of China under Grant (12371432). The second author gratefully acknowledges the invitation and support of the Harbin Institute of Technology. This work was conducted within the framework of the ANR project COOKIE, to which the second author contributed through Task I.3.
\appendix
\section{Proof of Lemma \ref{Lemm:integral_error}}
\label{Ap. Proof of Lemma.integral}
\begin{proof}
We begin by expressing the error term $\epsilon$ as:
\[
\epsilon = h\int_{0}^{1}\int_{\frac{p(x_1) + p(x_1+h)}{2}}^{p(x_1+sh)} f(x_1+sh,y)\,dy\,ds.
\]

Applying the coordinate transformation:
\[
y = \frac{p(x_1) + p(x_1+h)}{2} + t\delta p(s,h),
\]
where $\delta p(s,h) = p(x_1+sh) - \frac{p(x_1) + p(x_1+h)}{2}$, we obtain:
\[
\epsilon = h\int_{0}^{1}\int_{0}^{1} \phi(s,t,h)\,dt\,ds
\]
with
\[
\phi(s,t,h) = \delta p(s,h) f\left(x_1+sh, \frac{p(x_1) + p(x_1+h)}{2} + t\delta p(s,h)\right).
\]

Expanding $\phi$ via Taylor series in $h$ yields:
\[
\epsilon = h\int_{0}^{1}\int_{0}^{1} \left(\phi(s,t,0) + \partial_h\phi(s,t,0)h + \frac{1}{2}\partial_h^2\phi(s,t,0)h^2\right)\,dt\,ds + \mathcal{O}(h^4).
\]

Through direct computation (verified symbolically), we find:
\[
\epsilon= \left(\frac{1}{24}\partial_yf(p')^2 + \frac{1}{12}\partial_xfp' - \frac{1}{12}fp''\right)\bigg|_{(x_1,p(x_1))}h^3 + \mathcal{O}(h^4).
\]
\end{proof}
\section{Proof of Proposition \ref{Prop_intermediate}}
\label{Ap. Proof of Proposition Prop_intermediate}
\begin{proof}
    We first define the intermediate cell and its relationship to the true intermediate cell. Without loss of generality, as shown in Fig. \ref{Fig_mass inter}, the intermediate cell $K'L'N'M'$ (dashed lines) approximates the true intermediate cell $KLNM$ (solid curves) with piecewise linear boundaries, where the vertical boundaries are determined by $\overline{y}_{i,l-\frac{1}{2}}$ and $\overline{y}_{i,l+\frac{1}{2}}$. The mass of the intermediate cell, $M_{\text{intermediate}}$, is given by:
    \begin{figure}[htbp]
        \centering
        \includegraphics[width=0.3\textwidth]{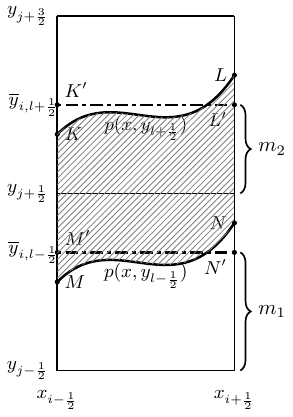}
        \caption{Geometric relationship between the intermediate cell $K'L'N'M'$ and the true intermediate cell $KLNM$. The intermediate-cell boundaries are constructed by connecting the midpoint values $\overline{y}_{i,l-\frac{1}{2}} = \tfrac{1}{2}\bigl(p(x_{i-\frac{1}{2}},y_{l-\frac{1}{2}})+p(x_{i+\frac{1}{2}},y_{l-\frac{1}{2}})\bigr)$ and $\overline{y}_{i,l+\frac{1}{2}} = \tfrac{1}{2}\bigl(p(x_{i-\frac{1}{2}},y_{l+\frac{1}{2}})+p(x_{i+\frac{1}{2}},y_{l+\frac{1}{2}})\bigr)$, where the function $p$ is defined in Eq.~\ref{Eq. p(x)}.}
        \label{Fig_mass inter}
    \end{figure}
    \begin{equation}
    \label{Eq. mass inter}
    \begin{aligned}
        M_{\text{intermediate}}=&m_2+m_{i,j}-m_1\\
        =&F_2(\overline{y}_{i,l+\frac{1}{2}})+m_{i,j}-F_1(\overline{y}_{i,l-\frac{1}{2}})\\
        =&\int_{x_{i-\frac{1}{2}}}^{x_{i+\frac{1}{2}}}\int_{\overline{y}_{i,l-\frac{1}{2}}}^{\overline{y}_{i,l+\frac{1}{2}}} f(x,y)\,dy\,dx+R^{2d+1},
    \end{aligned}
    \end{equation}
    where $m_{i,j}$ denotes the mass within the cell $[x_{i-\frac{1}{2}}, x_{i+\frac{1}{2}}] \times [y_{j-\frac{1}{2}}, y_{j+\frac{1}{2}}]$. The function $F_s(y)$ satisfies:
    \begin{equation}
    \label{Eq. primitive function}
    F_s(y)=\int_{x_{i-\frac{1}{2}}}^{x_{i+\frac{1}{2}}}\int_{y_{j-\frac{3}{2}+s}}^yf(\alpha,\beta)\,d\beta\,d\alpha
    \end{equation}
    where $s=1,2$.
    $R^{2d+1}$ denotes the interpolation error of order $2d+1$ for reconstruction.
    
    Therefore, according to Lemma \ref{Lemm:integral_error}, the error between the intermediate cell and its corresponding true intermediate cell is
    \begin{equation}
    \label{Eq. error_Min}
    \begin{aligned}
        |M_{\text{intermediate}}-M_{\text{true}}|=&\Bigg|\int_{x_{i-\frac{1}{2}}}^{x_{i+\frac{1}{2}}}\int_{\overline{y}_{i,l-\frac{1}{2}}}^{\overline{y}_{i,l+\frac{1}{2}}} f(x,y)\,dy\,dx+R^{2d+1}\\ &-\int_{x_{i-\frac{1}{2}}}^{x_{i+\frac{1}{2}}}\int_{p(x,y_{l-\frac{1}{2}})}^{p(x,y_{l+\frac{1}{2}})} f(x,y)\,dy\,dx\Bigg|\\
        \le&\Bigg|\int_{x_{i-\frac{1}{2}}}^{x_{i+\frac{1}{2}}}\int_{p(x,y_{l+\frac{1}{2}})}^{\overline{y}_{i,l+\frac{1}{2}}} f(x,y)\,dy\,dx-\\ &\int_{x_{i-\frac{1}{2}}}^{x_{i+\frac{1}{2}}}\int_{p(x,y_{l-\frac{1}{2}})}^{\overline{y}_{i,l-\frac{1}{2}}} f(x,y)\,dy\,dx\Bigg|+|R^{2d+1}|\\
        %=&\Bigg|\bigg(\frac{1}{24}\partial_yf(p_2')^2 + \frac{1}{12}\partial_xfp_2' - \frac{1}{12}fp''_2\bigg)_{(x_{i-\frac{1}{2}}, p_2(x_{i-\frac{1}{2}}))}\\
        %&-\bigg(\frac{1}{24}\partial_yf(p'_1)^2 + \frac{1}{12}\partial_xfp'_1 - \frac{1}{12}fp''_1\bigg)_{(x_{i-\frac{1}{2}}, p_1(x_{i-\frac{1}{2}}))}\Bigg|h^3 + \mathcal{O}(h^4)+|R^{2d+1}|.
        %=&\Bigg|\bigg(\frac{1}{24}\partial_yf(x,p(x,y))(\partial_xp(x,y))^2 + \frac{1}{12}\partial_xf(x,p(x,y))\partial_xp(x,y) \\
        %&-\frac{1}{12}f(x,p(x,y))\partial_{xx}p(x,y)\bigg)\bigg|_{x=x_{i-\frac{1}{2}}, y=y_k}^{x=x_{i-\frac{1}{2}}, y=y_{k+1}}\Bigg|h^3 + \mathcal{O}(h^4)+|R^{2d+1}|.\\
        =&\Bigg|\mathcal{E}(x,y)\Big|_{x=x_{i-\frac{1}{2}}, y=y_{l-\frac{1}{2}}}^{x=x_{i-\frac{1}{2}}, y=y_{l+\frac{1}{2}}}\Bigg|h^3 + \mathcal{O}(h^4)+|R^{2d+1}|,
    \end{aligned}
    \end{equation}
    where 
    \begin{align*}
        \mathcal{E}(x,y)=&\frac{1}{24}\partial_yf(x,p(x,y))(\partial_xp(x,y))^2 + \frac{1}{12}\partial_xf(x,p(x,y))\partial_xp(x,y)\\
        &-\frac{1}{12}f(x,p(x,y))\partial_{xx}p(x,y).
    \end{align*}
    
    Given the interpolation error satisfies $R^{2d+1}=\mathcal{O}(h^4)$, the substitution of the above expression into Eq.\eqref{Eq. error_Min} combined with Taylor expansion of $\mathcal{E}$ gives:
    \begin{align*}
        \Big|M_{\text{intermediate}}-M_{\text{true}}\Big|\le&\Big|\partial_y\mathcal{E}(x_{i-\frac{1}{2}},y_{l-\frac{1}{2}})\Big|h^4+\mathcal{O}(h^4)+|R^{2d+1}|\\
        =&\mathcal{O}(h^4)
    \end{align*}
    Therefore, the mass error between the intermediate cell and the true intermediate cell is indeed of order $\mathcal{O}(h^4)$, which completes the proof.
\end{proof}

\section{General cases of CCSL method}
\label{Ap. General cases of CCSL method}
For more general cases, consider the region $ABCD$ with corners at $(x_{i\pm\frac{1}{2}}, y_{j\pm\frac{1}{2}})$; the backtracked coordinates of these corners at time $t_n$ are $(x_{i\pm\frac{1}{2}, j\pm\frac{1}{2}}^*, y_{i\pm\frac{1}{2}, j\pm\frac{1}{2}}^*)$. Assume the coordinate distribution satisfies the condition in Proposition \ref{Prop_ CCSL condition}, i.e., the original top-bottom and left-right order is preserved. Define $\overline{x}_1 = \frac{1}{2}(x_{i-\frac{1}{2}, j+\frac{1}{2}}^* + x_{i-\frac{1}{2}, j-\frac{1}{2}}^*)$ and $\overline{x}_2 = \frac{1}{2}(x_{i+\frac{1}{2}, j+\frac{1}{2}}^* + x_{i+\frac{1}{2}, j-\frac{1}{2}}^*)$ and let the nearest half-grid points to their left be $x_{\lambda-\frac{1}{2}}$ and $x_{\mu-\frac{1}{2}}$, respectively, where $\lambda \geq \mu$. In this case, the corresponding Eq. \eqref{Eq. mass of CCSL} should be:
\begin{equation*}
    M_{\text{CCSL}}=
    \begin{cases}
        G_{\lambda,j}^{3}(\overline{x}_2)+\displaystyle\sum_{k=\mu} ^{\lambda}m_{I_{k,j}}-G_{\mu,l}^{3}(\overline{x}_1),\;& \text{if}\;\lambda > \mu,\\
        G_{\lambda,j}^{3}(\overline{x}_2)-G_{\mu,l}^{3}(\overline{x}_1),\;&\text{if}\;\lambda=\mu.
    \end{cases}
\end{equation*}
The corresponding Eq. \eqref{Eq. true intermediate mass} should be
\begin{equation*}
    M^*=\int_{\overline{x}_1}^{\overline{x}_2}\int_{p(x,y_{j-\frac{1}{2}})}^{p(x,y_{j+\frac{1}{2}})}f(x,y,t^n)\,dy\,dx.
\end{equation*}

Thus, for other more general cases, the proof method is analogous.
\section{\texorpdfstring{Reconstruction of function $G$}{Reconstruction of function G}}
\label{Ap. Reconstruction of function G}
For the interpolation order of $2d+1$, the function $G_{i,l}^{2d+1}$ satisfies:
\begin{equation*}
    G_{i,l}^{2d+1}(x_{i-\frac{1}{2}+k})=
    \begin{cases}
        -\displaystyle\sum_{c=-d+k}^{-1}m_{I_{i+c,l}},\;&k<0;\\
        0,\;&k=0;\\
        \displaystyle\sum_{c=0}^{k-1}m_{I_{i+c,l}}-\displaystyle\sum_{c=-d}^{-1}m_{I_{i+c,l}},\;&k>0,
    \end{cases}
\end{equation*}
where $k = -d, \dots, d+1$.
    
%\bibliographystyle{elsarticle-harv} 
% \bibliographystyle{abbrvnat}
% \bibliography{bib}
%\begingroup
%  % 在参考文献环境里关闭斜体
%  \renewcommand{\em}[1]{#1}
%  % 如果你想更保险一点，也可以同时关掉 \emph：
%  % \renewcommand{\emph}[1]{#1}
%
%  \bibliographystyle{abbrvnat}
%  \bibliography{bib}
%\endgroup
\bibliographystyle{abbrv}
\bibliography{bib}
%\begingroup
%  % 在参考文献环境里关闭斜体
%  \renewcommand{\em}[1]{#1}
%  % 如果你想更保险一点，也可以同时关掉 \emph：
%  % \renewcommand{\emph}[1]{#1}
%
%  \bibliographystyle{abbrvnat}
%  \bibliography{bib}
%\endgroup

\end{document}